\documentclass[journal]{IEEEtran}

\usepackage{graphicx}
\usepackage{graphics} 
\usepackage{times} 
\usepackage{amsmath} 
\usepackage{amssymb} 
\usepackage{algorithm2e}
\usepackage{pseudocode}

\usepackage{enumitem}
\usepackage{cite}
\usepackage{url}

\usepackage[usenames, dvipsnames]{color}
\definecolor{darkblue}{rgb}{0.0, 0.0, 0.45}
\usepackage[colorlinks	= true,
			raiselinks	= true,
			linkcolor	= darkblue, %MidnightBlue,
			citecolor	= Mahogany,
			urlcolor	= ForestGreen,
			pdfauthor	= {Peyman Mohajerin Esfahani},
			pdftitle	= {},
			pdfkeywords	= {},
			pdfsubject	= {},
			plainpages	= false]{hyperref}

\newtheorem{theorem}{\bf Theorem}

\newtheorem{corollary}{\bf Corollary}

\newtheorem{remark}{\bf Remark}

\newtheorem{lemma}{\bf Lemma}

\def\QED{~\rule[-1pt]{5pt}{5pt}\par\medskip}
\newenvironment{proof}{{\bf Proof: \ }}{ \hfill \QED}

\newcommand{\bx}{{\bf x}}
\newcommand{\by}{{\bf y}}
\newcommand{\bz}{{\bf z}}
\newcommand{\bu}{{\bf u}}
\newcommand{\bv}{{\bf v}}
\newcommand{\bw}{{\bf w}}
\newcommand{\bg}{{\bf g}}
\newcommand{\ba}{{\bf a}}
\newcommand{\bpsi}{\boldsymbol\psi}

\newcommand{\PP}{\mathbb{P}}
\newcommand{\GG}{\mathbb{G}}
\newcommand{\QQ}{\mathbb{Q}}
\newcommand{\LL}{\mathbb{L}}

\newcommand{\diff}{\mathrm{d}}

\definecolor{darkblue}{rgb}{0, 0, 0.8}

\newcommand{\upd}{}
\newcommand*{\LONGVERSION}{}%  %To be posted online
\newcommand*{\DOUBLECOLUMN}{}%  %To be posted online
\title{
LQG Control with Minimum Directed Information: \\
Semidefinite Programming Approach 
}

\author{Takashi Tanaka$^{1}$ \and \hspace{3ex} Peyman Mohajerin Esfahani$^{2}$ \and \hspace{3ex} Sanjoy K. Mitter$^{3}$% <-this % stops a space
\thanks{$^{1}$TT is with the Department of Aerospace Engineering and Engineering Mechanics at the University of Texas at Austin, USA.
        {\tt\small ttanaka@utexas.edu;}$^{2}$PME is with the Delft Center for Systems and Control at the Delft University of Technology, Netherlands.
        {\tt\small P.MohajerinEsfahani@tudelft.nl;}$^{3}$ SM is with the Laboratory for Information and Decision Systems, Massachusetts Institute of Technology, USA.
        {\tt\small mitter@mit.edu.}}%
}

\begin{document}

\maketitle
%\thispagestyle{empty}
%\pagestyle{empty}

%%%%%%%%%%%%%%%%%%%%%%%%%%%%%%%%%%%%%%%%%%%%%%%%%%%%%%%%%%%%%%%%%%%%%%%%%%%%%%%%
\begin{abstract}
We consider a discrete-time Linear-Quadratic-Gaussian (LQG) control problem in which Massey's directed information from the observed output of the plant to the control input is minimized while required control performance is attainable. This problem arises in several different contexts, including joint encoder and controller design for data-rate minimization in networked control systems. 
We show that the optimal control law is a Linear-Gaussian randomized policy. We also identify the state space realization of the optimal policy, which can be synthesized by an efficient algorithm based on semidefinite programming.
Our structural result indicates that the filter-controller separation principle from the LQG control theory, and the sensor-filter separation principle from the zero-delay rate-distortion theory for Gauss-Markov sources hold simultaneously in the considered problem.
A connection to the data-rate theorem for mean-square stability by Nair \& Evans is also established.
\end{abstract}

%%%%%%%%%%%%%%%%%%%%%%%%%%%%%%%%%%%%%%%%%%%%%%%%%%%%%%%%%%%%%%%%%%%%%%%%%%%%%%%%


\section{Introduction}
There is a fundamental trade-off between the best achievable control performance and the data-rate at which plant information is fed back to the controller.
Studies of such a trade-off hinge upon analytical tools developed at the interface between traditional feedback control theory and Shannon's information theory.
Although the interface field has been significantly expanded by the surged research activities on \emph{networked control systems (NCS)} over the last two decades \cite{nair2007, hespanha2007survey, baillieul2007, yuksel2013stochastic,you2015analysis}, many important questions concerning the rate-performance trade-off studies are yet to be answered.

A central research topic in the NCS literature has been the stabilizability of a linear dynamical system using a rate-constrained feedback \cite{TatikondaThesis,baillieul2002feedback,hespanha2002towards, nair2004stabilizability}. The critical data-rate below which stability cannot be attained by any feedback law has been extensively studied in various NCS setups. 
As pointed out by \cite{nair2004topological}, many results including \cite{TatikondaThesis,baillieul2002feedback,hespanha2002towards, nair2004stabilizability} share the same conclusion that this critical data-rate is characterized by an intrinsic property of the open-loop system known as topological entropy, which is determined by the unstable open-loop poles. This result holds irrespective of different definitions of the  ``data-rate''  considered in these papers. For instance, in \cite{nair2004stabilizability} the data-rate is defined as the log-cardinality of channel alphabet, while in \cite{hespanha2002towards}, it is the frequency of the use of noiseless binary channel.

As a natural next step, the rate-performance trade-offs are of great interest from both theoretical and practical perspectives.
The trade-off between Linear-Quadratic-Gaussian (LQG) performance and the required data-rate has attracted attention in the literature
\cite{borkar1997lqg,tatikonda2004,matveev2004problem,charalambous2008,fu2009linear,
you2011linear,freudenberg2011,bao2011iterative,yuksel2014jointly,huang2005finite,lemmon2006performance,
silva2011framework,silva2011achievable,silva2013characterization}.
Generalized interpretations of the classical Bode's integral also provide fundamental performance limitations of closed-loop systems in the information-theoretic terms \cite{iglesias2001analogue,zang2003nonlinear,elia2004bode,martins2008}.
However, the rate-performance trade-off analysis introduces additional challenges that were not present through the lens of the stability analysis.
First, it is largely unknown whether different definitions of the data-rate  considered in the literature listed above lead to different conclusions.  
This issue is less visible in the stability analysis, since the critical data-rate for stability turns out to be invariant across several different definitions of the data-rate \cite{TatikondaThesis,baillieul2002feedback,hespanha2002towards, nair2004stabilizability}.
Second, for many operationally meaningful definitions of the data-rate considered in the literature, computation of the rate-performance trade-off function involves intractable optimization problems (e.g., dynamic programming \cite{lemmon2006performance} and iterative algorithm \cite{bao2011iterative}), and trade-off achieving controller/encoder policies are difficult to obtain. This is not only inconvenient in practice, but also makes theoretical analyses difficult.

In this paper, we study the information-theoretic requirements for LQG control using the notion of \emph{directed information} \cite{marko1973bidirectional, massey1990causality, kramer2003capacity}.
In particular, we define the rate-performance trade-off function as the minimal directed information from the observed output of the plant to the control input, optimized over the space of causal decision policies that achieve the desired level of LQG control performance.
Among many possible definitions of the ``data-rate'' 
 as mentioned earlier, we focus on  directed information for the following reasons. 

{\upd
First, directed information (or related quantity known as \emph{transfer entropy}) is a widely used \emph{causality measure} in science and engineering  \cite{gourieroux1987kullback,amblard2011directed,jiao2015justification}.
Applications include communication theory (e.g., the analysis of channels with feedback), portfolio theory, neuroscience, social science, macroeconomics, statistical mechanics, and potentially more.
Since it is natural to measure
the ``data-rate''  in networked control systems by a causality measure from the observation to action, directed information is a natural option.}

Second, it is recently reported by Silva et al. \cite{silva2011framework,silva2011achievable,silva2013characterization} that directed information has an important operational meaning in a practical NCS setup.
Starting from an LQG control problem over a noiseless binary channel with prefix-free codewords, they show that the directed information obtained by solving the aforementioned optimization problem provides a tight lower bound for the minimum data-rate (defined operationally) required to achieve the desired level of control performance.

\subsection{Contributions of this paper}

The central question in this paper is the characterization of the most ``data-frugal'' LQG controller that minimizes directed information of interest among all decision policies achieving a given LQG control performance.
In this paper, we make the following contributions. 
%\begin{enumerate}[leftmargin=*]
\begin{enumerate}[label=(\roman*), itemsep = 0mm, topsep = 0mm] 
\item \label{cnt:structure}  In a general setting including MIMO, time-varying, and partially observable plants, we identify the structure of an optimal decision policy in a state space model.
\item \label{cnt:computation} Based on the above structural result, we further develop a tractable optimization-based framework to synthesize the optimal decision policy.  
\item \label{cnt:stationary} In the stationary setting with MIMO plants, we show how our proposed computational framework, as a special case, recovers the existing data-rate theorem for mean-square stability.
\end{enumerate}

Concerning \ref{cnt:structure}, we start with general time-varying, MIMO, and fully observable plants.
We emphasize that the optimal decision policy in this context involves two important tasks: (1) the  \emph{sensing} task, indicating which state information of the plant should be dynamically measured with what precision, and  (2) the \emph{control} task, synthesizing an appropriate control action given available sensing information. To this end, we first show that the optimal policy that minimizes directed information from the state to the control sequences under the LQG control performance constraint is linear. 
In this vein, we illustrate that the optimal policy can be realized by a three-stage architecture comprising linear sensor with additive Gaussian noise, Kalman filter, and certainty equivalence controller (Theorem~\ref{maintheorem}). We then show how this result can be extended to partially observed plants (Theorem~\ref{theorempartially}).

Regarding \ref{cnt:computation}, 
we provide a semidefinite programming (SDP) framework characterizing the optimal policy proposed in step \ref{cnt:structure} (Sections~\ref{secmainresult} and \ref{secpo}). As a result, 
we obtain a computationally accessible form  of the considered rate-performance trade-off functions. 

Finally, as highlighted in \ref{cnt:stationary}, we analyze the horizontal asymptote of the considered rate-performance trade-off function for MIMO time-invariant plants (Theorem~\ref{theostationary}), which coincides with the critical data-rate identified by  Nair and Evans \cite{nair2004stabilizability} (Corollary~\ref{cordatarate}).

\subsection{Organization of this paper}
The rest of this paper is organized as follows.
After some notational remarks, the problem considered in this paper is formally introduced in Section~\ref{secformulation}, and its operational interpretation is provided in Section~\ref{secmotivation}. Main results are summarized in Section~\ref{secmainresult}, where connections to the existing results are also explained in detail.  Section~\ref{secexample} contains a simple numerical example, and the derivation of the main results is presented in  Section~\ref{secderivation}. 
The results are extended to partially observable plants in Section~\ref{secpo}.
We conclude in Section~\ref{secconclusion}.

\subsection{Notational remarks}
%{\bf Notation:}
Throughout this paper, random variables are  denoted by lower case bold symbols such as $\bx$.  Calligraphic symbols such as $\mathcal{X}$ are used to denote sets, and $x\in\mathcal{X}$ is an element.  We denote by $x^t$ a sequence $x_1,x_2,...,x_t$, and $\bx^t$ and $\mathcal{X}^t$ are understood similarly. All random variables in this paper are Euclidean valued, and is measurable with respect to the usual topology. A probability distribution of $\bx$ is demoted by $\PP_\bx$. A Gaussian distribution with mean $\mu$ and covariance $\Sigma$ is denoted by $\mathcal{N}(\mu,\Sigma)$.
The relative entropy of $\QQ$ from $\PP$ is a non-negative quantity defined by
\[
D(\PP \| \QQ )\triangleq\begin{cases}
\int \log_2 \frac{\diff \PP(x)}{\diff \QQ(x)}\diff \PP(x) &\text{if } \PP \ll \QQ \\
+\infty & \text{otherwise}
\end{cases}
\]
where $\PP\ll \QQ$ means that $\PP$ is absolutely continuous with respect to $\QQ$, and $\frac{\diff \PP(x)}{\diff \QQ(x)}$ denotes the Radon-Nikodym derivative.  The mutual information between $\bx$ and $\by$ is defined by
$I(\bx;\by)\triangleq D(\PP_{\bx,\by}\|\PP_\bx\otimes \PP_\by)$, where $\PP_{\bx,\by}$ and $\PP_\bx\otimes \PP_\by$ are joint and product probability measures respectively. 
The entropy of a discrete random variable $\bx$ with probability mass function $\PP(x_i)$ is defined by $H(\bx)\triangleq -\sum_i \PP(x_i)\log_2 \PP(x_i)$.

\section{Problem Formulation}
\label{secformulation}

\ifdefined\DOUBLECOLUMN
		\begin{figure}[t]
		\centering
		\includegraphics[width=0.65\columnwidth]{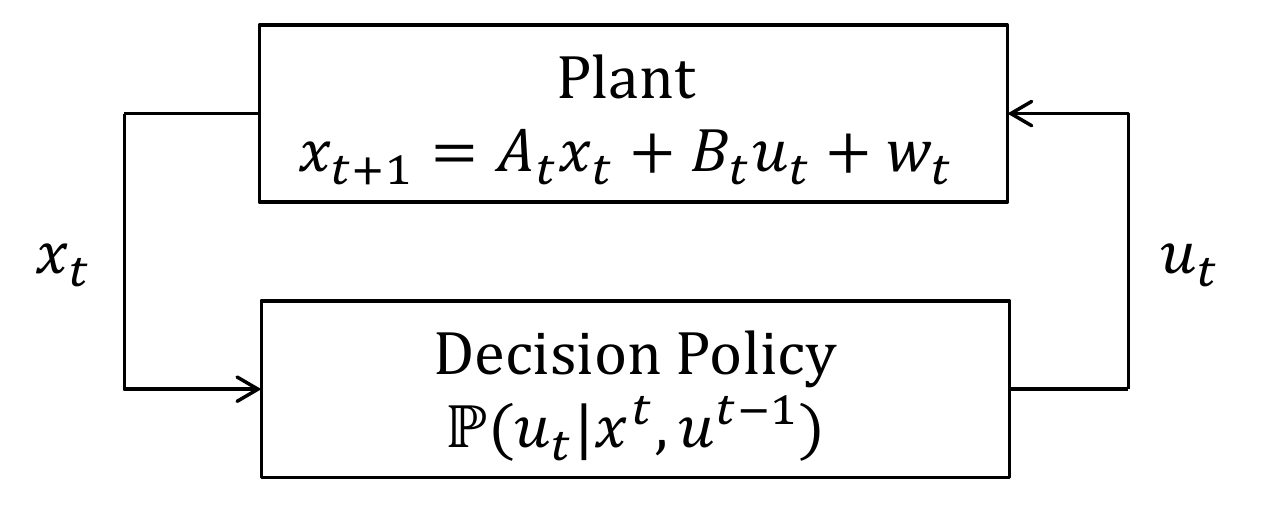}
		\caption{LQG control of fully observable plant with minimum directed information.}
		\label{fig:mainprob}
		\vspace{-3ex}
	\end{figure}
\fi
\ifdefined\SINGLECOLUMN
		\begin{figure}[t]
		\centering
		\includegraphics[width=0.5\columnwidth]{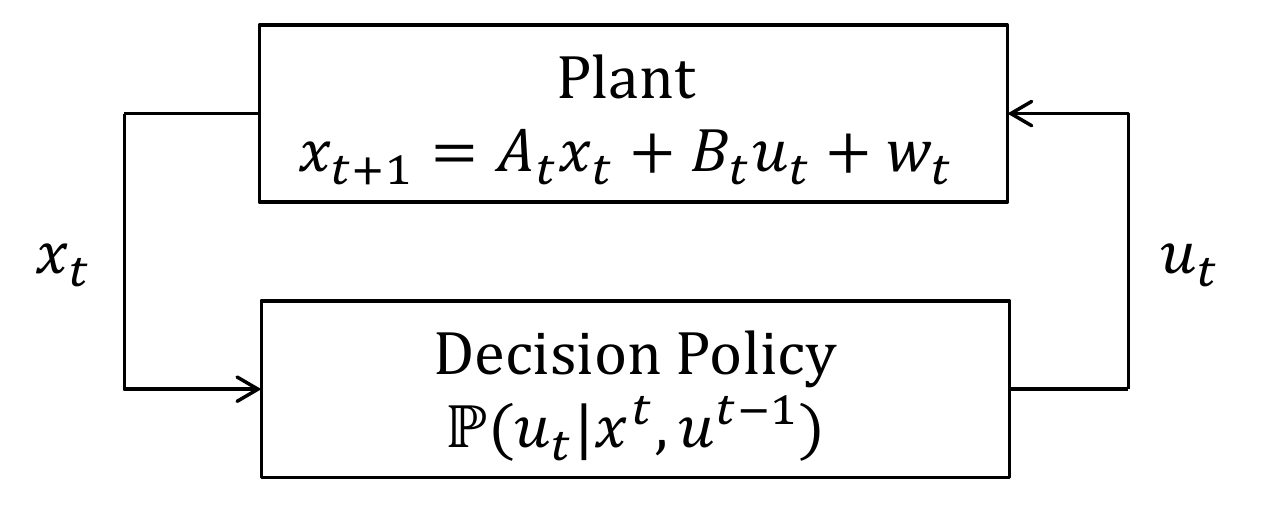}
		\caption{LQG control of fully observable plant with minimum directed information.}
		\label{fig:mainprob}
		\vspace{-2ex}
	\end{figure}
\fi
	
Consider a linear time-varying stochastic plant
	\begin{align}
	\label{eqsystem}
	\bx_{t+1}=A_t\bx_t+B_t\bu_t+\bw_t, \quad t=1,\cdots, T,
	\end{align} 
	where $\bx_t$ is an $\mathbb{R}^n$-valued  state of the plant, and $\bu_t$ is the control input. 
	We assume that initial state $\bx_1\sim\mathcal{N}(0, P_{1|0}), P_{1|0}\succ 0$ and noise process  $\bw_t\sim\mathcal{N}(0, W_t)$, $W_t \succ 0, t=1,...,T$ are mutually independent.
	%\subsection{Main problem}
	
	The design objective is to synthesize a decision policy that ``consumes" the least amount of information among all policies achieving the required LQG control performance  (Figure~\ref{fig:mainprob}). Specifically, let  $\Gamma$ be the space of decision policies, i.e., the space of sequences of  Borel measurable stochastic kernels  \cite{bertsekas1978stochastic} 
	\[
	 \PP(u^T || x^T) \triangleq \{ \PP(u_t|x^t,u^{t-1})\}_{t=1,...,T}.
	\]
	 A decision policy $\gamma \in \Gamma$ is evaluated by two criteria:
	\begin{enumerate}[label=(\roman*), itemsep = 1mm, topsep = 1mm]
		\item the LQG control cost
		\begin{equation}
		\label{deflqgcost}
		J(\bx^{T+1},\bu^T) \triangleq \sum\nolimits_{t=1}^T\!\mathbb{E}\left( \|\bx_{t+1}\|_{Q_t}^2\!+\|\bu_t\|_{R_t}^2\right);
		\end{equation}
		\item and  \emph{directed information}
		\begin{equation}
		\label{defdirectinfo}
		I(\bx^T\rightarrow \bu^T)\triangleq \sum\nolimits_{t=1}^T I(\bx^t;\bu_t|\bu^{t-1}).
		\end{equation}
	\end{enumerate}
	The right hand side of \eqref{deflqgcost} and \eqref{defdirectinfo} are evaluated with respect to the joint probability measure induced by the state space model (\ref{eqsystem}) and a decision policy $\gamma$. 
	In what follows, we often write \eqref{deflqgcost} and \eqref{defdirectinfo} as $J_\gamma$ and $I_\gamma$ to indicate their dependency on $\gamma$.
	The main problem studied in this paper is formulated as 
	\begin{subequations}
		\label{mainprob}
		\begin{align}
		\mathsf{DI}_T(D)\triangleq \min_{\gamma\in\Gamma} & \quad  I_\gamma(\bx^T\rightarrow \bu^T) \label{objdi}\\
		\text{ s.t. } &\quad  J_\gamma(\bx^{T+1},\bu^T) \leq D,
		\end{align}
	\end{subequations}
	where $D > 0$ is the desired LQG control performance.

	Directed information \eqref{defdirectinfo} can be interpreted as the information flow from the state random variable $\bx_t$ to the control random variable $\bu_t$.  The following equality called  {\em conservation of information} \cite{massey2005conservation} shows a connection between directed information and the standard mutual information:
	\[ I(\bx^T;\bu^T)=I(\bx^T\rightarrow \bu^T)+I(\bu_+^{T-1}\rightarrow \bx^T).\]
Here, the sequence $\bu_+^{T-1}=(0,\bu_1,\bu_2,\cdots, \bu_{T-1})$ denotes an index-shifted version of $\bu^T$. Intuitively, this equality shows that the standard mutual information can be written as a sum of two directed information terms corresponding to feedback (through decision policy) and feedforward (through plant) information flows. Thus \eqref{mainprob} is interpreted as the minimum information that must ``flow" through the decision policy to achieve the LQG control performance $D$.
	
	We also consider time-invariant and infinite-horizon LQG control problems. 
	Consider a time-invariant plant
	\begin{equation}
	\label{eqltiplant}
	\bx_{t+1}=A\bx_t+B\bu_t+\bw_t, \;\; t\in\mathbb{N}
	\end{equation}
with $\bw_t\sim\mathcal{N}(0,W)$, and assume $Q_t=Q$ and $R_t=R$ for $t\in\mathbb{N}$. We also assume   $(A,B)$ is stabilizable, $(A,Q)$ is detectable, and $R\succ 0$. 
Let $\Gamma$ be the space of Borel-measurable stochastic kernels $\PP(u^\infty || x^\infty)$.
The problem of interest is
\begin{subequations}
\label{stationaryprob}
\begin{align}
\mathsf{DI}(D)\triangleq \min_{\gamma\in\Gamma} & \quad \limsup_{T\rightarrow \infty} \tfrac{1}{T} I_\gamma(\bx^T\rightarrow \bu^T) \label{objdi2}\\
\text{ s.t. } &\quad \limsup_{T\rightarrow \infty} \tfrac{1}{T} J_\gamma(\bx^{T+1},\bu^T) \leq D.
\end{align}
\end{subequations}
More general problem formulations with partially observable plants will be discussed in Section~\ref{secpo}.

\section{Operational meaning}
\label{secmotivation}

In this section, we revisit a networked LQG control problem considered in \cite{silva2011framework,silva2013characterization,silva2011achievable}.
Here we consider time-invariant  MIMO plants while  \cite{silva2011framework,silva2013characterization,silva2011achievable} focus on SISO plants. For simplicity, we consider fully observable plants only.
Consider a feedback control system in Figure~\ref{fig:com}, where the state information is encoded by the ``sensor + encoder'' block and is transmitted to the controller over a noiseless binary channel.  For each $t=1, ... ,T$, let 
	$ \mathcal{A}_t \subset \{0,1,00,01,10,11,000, \cdots\} $
be a set of uniquely decodable variable-length codewords \cite[Ch.5]{CoverThomas}.  Assume that codewords are generated by a causal policy
\[  \PP(a^\infty || x^\infty) \triangleq \{\PP(a_t|x^t,a^{t-1})\}_{t=1,2,...}. \]
The ``{decoder + controller}" block interprets codewords and computes control input according to a causal policy
\[  \PP(u^\infty || a^\infty) \triangleq \{\PP(u_t|a^t,u^{t-1})\}_{t=1,2,...}. \]
The length of a  codeword $\ba_t \in \mathcal{A}_t$ is denoted by a random variable $l_t$. Let $\Gamma'$ be the space of triplets $\{\PP(a^\infty || x^\infty),  \mathcal{A}^\infty,  \PP(u^\infty || a^\infty)\}$.
Introduce a quadratic control cost
\[
J (\bx^{T+1},\bu^T) \triangleq \sum\nolimits_{t=1}^T\!\mathbb{E}\left( \|\bx_{t+1}\|_{Q}^2\!+\|\bu_t\|_{R}^2\right) \]
with $Q \succ 0$ and $R \succ 0$. We are interested in a design $\gamma'\in\Gamma'$ that minimizes data-rate among those attaining control cost smaller than $D$.
Formally, the problem is formulated as
	\begin{subequations}
		\label{quantizedLQG}
		\begin{align}
		\mathsf{R}(D)\triangleq \min_{\gamma' \in \Gamma'} & \quad   \limsup_{T\rightarrow +\infty}\frac{1}{T} \sum\nolimits_{t=1}^T \mathbb{E}(l_t) \label{objbits}\\
		\text{ s.t. } &\quad  \limsup_{T\rightarrow +\infty}\frac{1}{T}  J (\bx^{T+1},\bu^T) \leq D.
		\end{align}
	\end{subequations}
It is difficult to evaluate $\mathsf{R}(D)$ directly since (\ref{quantizedLQG}) is a highly complex optimization problem. Nevertheless, Silva et al. \cite{silva2011framework} observed that $\mathsf{R}(D)$ is closely related to $\mathsf{DI}(D)$ defined by \eqref{stationaryprob}. The following result is due to \cite{extendedversion}.
\begin{equation}
\label{eqoperationalmeaning}
\mathsf{DI}(D) \leq \mathsf{R}(D) < \mathsf{DI}(D)+\frac{r}{2}\log\frac{4\pi e}{12}+1\;\;  \forall D >0.
\end{equation}
Here, $r$ is an integer no greater than the state space dimension of the plant.\footnote{More precisely, $r$ is the rank of the optimal signal-to-noise ratio matrix obtained by semidefinite programming, as will be clear in Section~\ref{secmainlti}.} The following inequality plays an important role to prove \eqref{eqoperationalmeaning}.
	\begin{lemma}%[Feedback Data-Processing Inequality]
		\label{lemmadirectedinfo}
		Consider a control system (\ref{eqsystem}) with a decision policy $\gamma'\in\Gamma'$. Then, we have an inequality
\[
		I(\bx^T\rightarrow \bu^T) \leq I(\bx^T\rightarrow \ba^T\|\bu_+^{T-1}), 
\]
		
where the right hand side is Kramer's notation \cite{kramer2003capacity} for causally conditioned directed information
$\sum_{t=1}^T I(\bx^t;\ba_t|\ba^{t-1}, \bu^{t-1})$.
	\end{lemma}	
	\begin{proof}
		See Appendix~\ref{adddataprocessing}. 
	\end{proof}	
	Lemma~\ref{lemmadirectedinfo} can be thought of as a generalization of the standard data-processing inequality. It is different from the directed data-processing inequality in \cite[Lemma 4.8.1]{TatikondaThesis} since the source $\bx_t$ is affected by feedback. See also \cite{derpich2013fundamental} for relevant inequalities involving directed information. 

\ifdefined\DOUBLECOLUMN
	\begin{figure}[t]
		\centering
		\includegraphics[width=0.8\columnwidth]{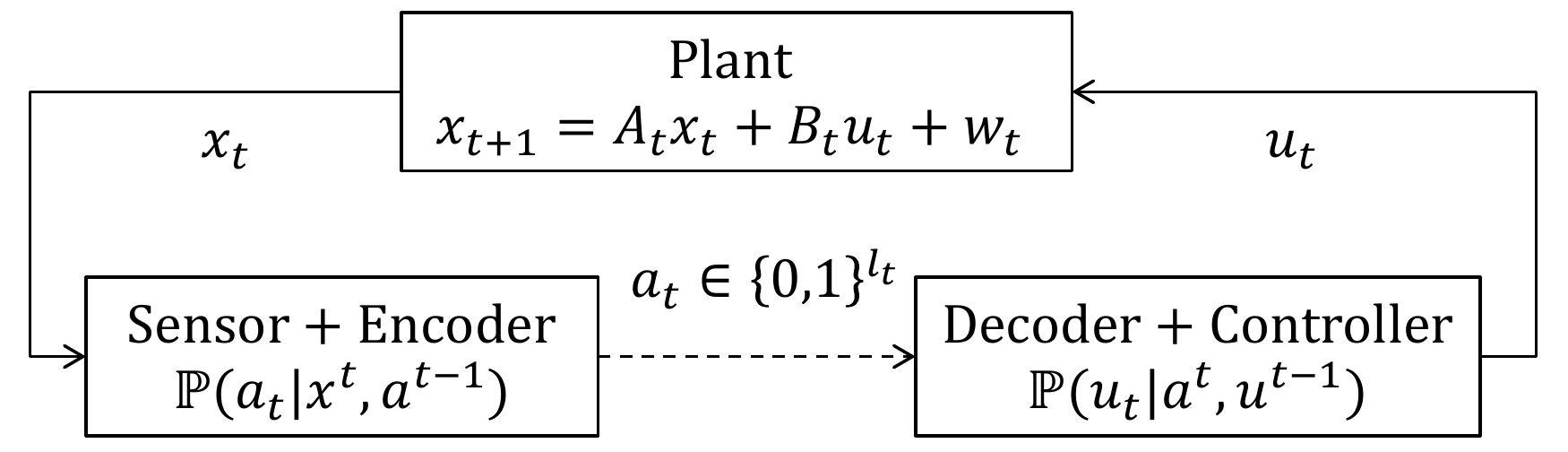}
		\caption{LQG control over noiseless binary channel. }
		\label{fig:com}
		\vspace{-3ex}
	\end{figure}
\fi
\ifdefined\SINGLECOLUMN
	\begin{figure}[t]
		\centering
		\includegraphics[width=0.6\columnwidth]{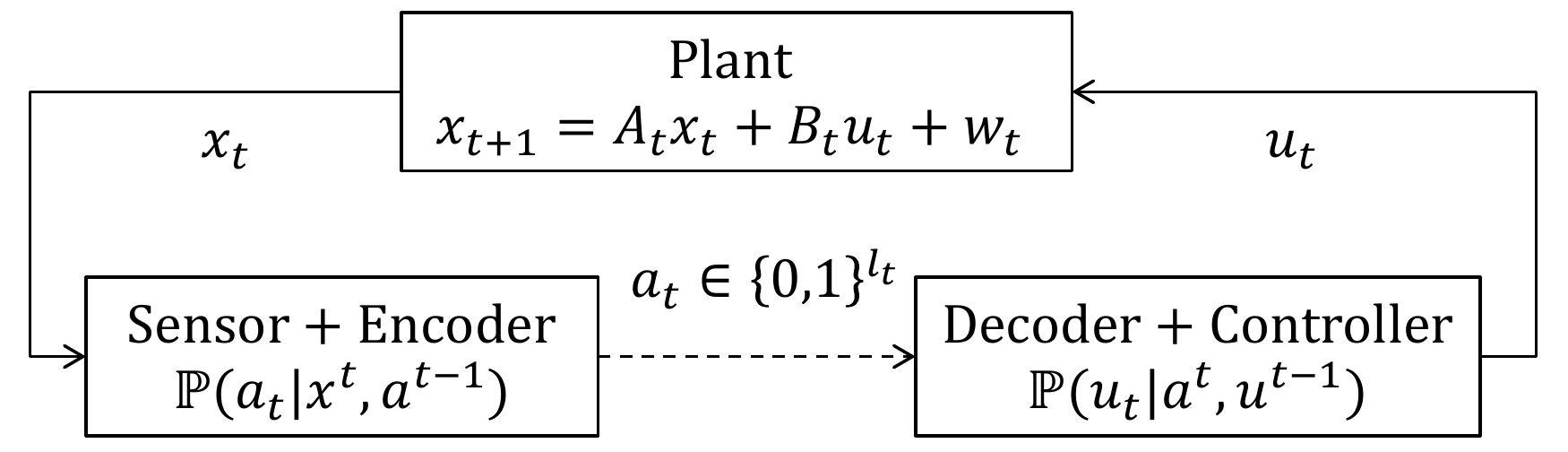}
		\caption{LQG control over noiseless binary channel. }
		\label{fig:com}
	\end{figure}
\fi
	
Now, the first inequality in \eqref{eqoperationalmeaning} can be directly verified  as	
	\begin{subequations}
		\label{directedinfolowerbound}
		\begin{align}
		&I(\bx^T\rightarrow \bu^T) \\
		\leq&\sum\nolimits_{t=1}^T I(\bx^t;\ba_t|\ba^{t-1},\bu^{t-1})  \\
		=&\sum\nolimits_{t=1}^T \left(H(\ba_t|\ba^{t-1},\bu^{t-1})-H(\ba_t|\bx^t,\ba^{t-1},\bu^{t-1})\right) \\
		\leq &\sum\nolimits_{t=1}^T H(\ba_t|\ba^{t-1},\bu^{t-1}) \\
		\leq &\sum\nolimits_{t=1}^T H(\ba_t)  \\
		\leq &\sum\nolimits_{t=1}^T \mathbb{E}(l_t). 
		\end{align}
	\end{subequations}
	Lemma \ref{lemmadirectedinfo} is used in the first step. The last step follows from the fact that expected codeword length of of uniquely decodable codes is lower bounded by its entropy \cite[Theorem 5.3.1]{CoverThomas}.
	
	Proving the second inequality in \eqref{eqoperationalmeaning} requires a key technique proposed in \cite{silva2011framework} involving the construction of dithered uniform quantizer \cite{zamir1992universal}. Detailed discussion is available in \cite{extendedversion}.

\section{Main Result}
\label{secmainresult}
In this section we present the main results of this article. For the clarity of the presentation, this section is only devoted to a setting with full state measurements and shows how the main objective of control synthesis can be achieved by a three-step procedure. We shall later discuss in Section~\ref{secpo} in regard to an extension to partial observable systems.

\subsection{Time-varying plants}
We show that the optimal solution to  (\ref{mainprob}) can be realized by the following three data-processing components as shown in Figure~\ref{fig:sep}.
\begin{itemize}[leftmargin=3ex]
\item[1.] A linear sensor mechanism 
\begin{equation}
\label{eqvirtualsensor}
\by_t=C_t \bx_t+\bv_t, \;\; \bv_t\sim\mathcal{N}(0,V_t),  \;\; V_t \succ 0
\end{equation} 
where $\bv_t, t=1, ... , T$ are mutually independent.
\item[2.] The Kalman filter computing $\hat{\bx}_t=\mathbb{E}(\bx_t|\by^t,\bu^{t-1})$.
\item[3.] The certainty equivalence controller $\bu_t=K_t\hat{\bx}_t$.
\end{itemize}

\ifdefined\DOUBLECOLUMN
\begin{figure}[t]
    \centering
    \includegraphics[width=0.8\columnwidth]{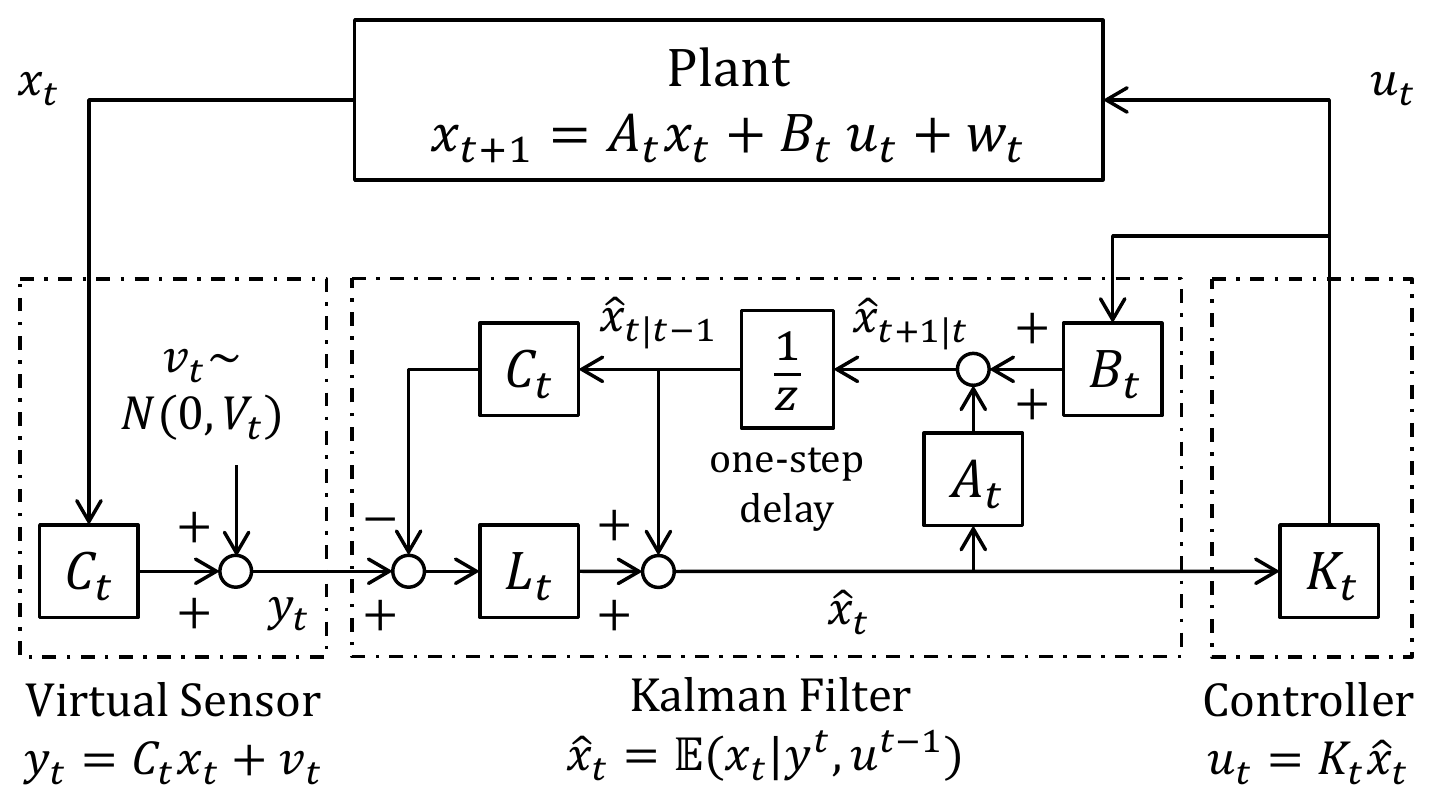}
    \caption{Structure of optimal control policy for problem \eqref{mainprob}. Matrices $C_t$, $V_t$, $L_t$ and $K_t$ are determined by SDP-based algorithm in Section~\ref{secmainresult}.}
    \label{fig:sep}
    \vspace{-3ex}
\end{figure}
\fi
\ifdefined\SINGLECOLUMN
\begin{figure}[t]
    \centering
    \includegraphics[width=0.6\columnwidth]{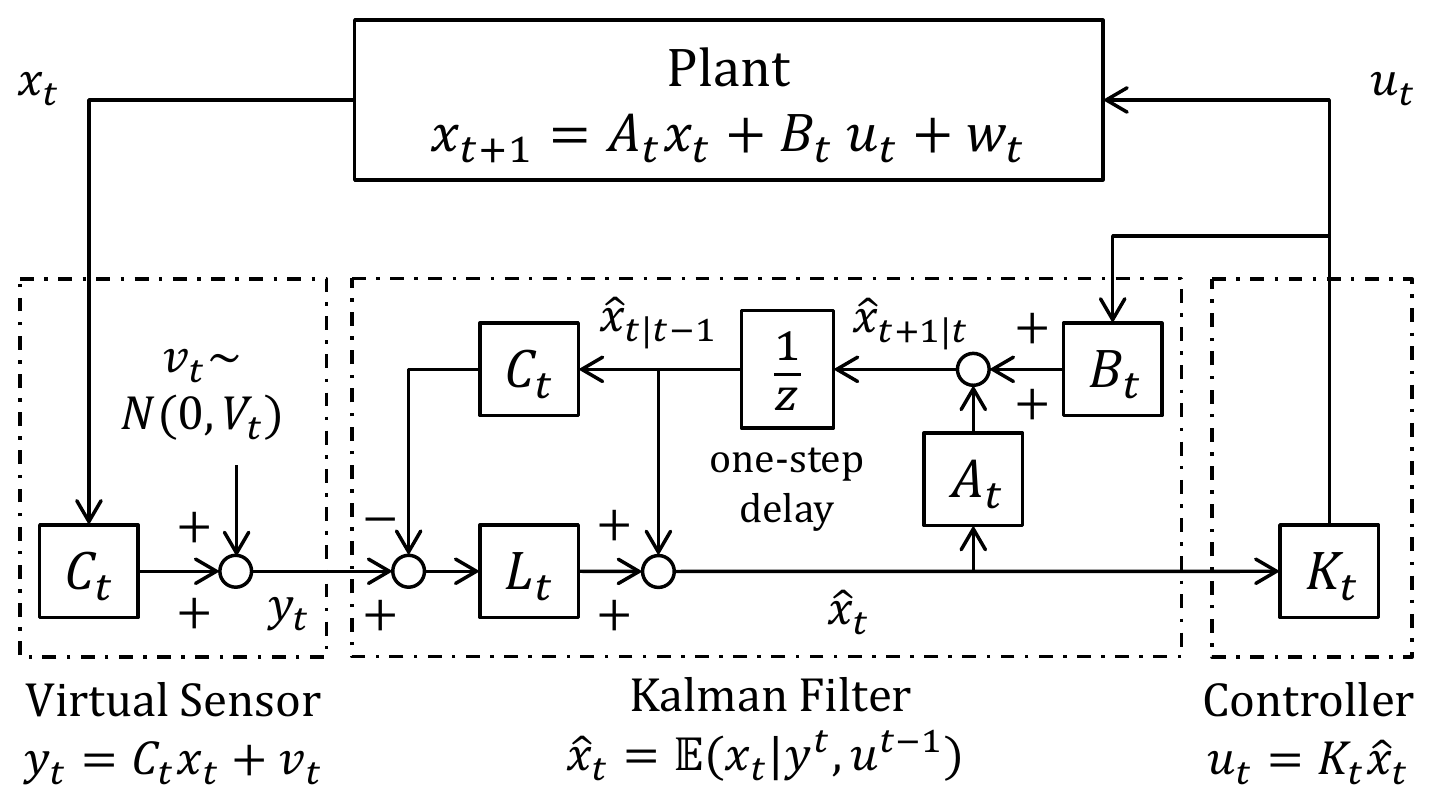}
    \caption{Structure of optimal control policy for problem \eqref{mainprob}. Matrices $C_t$, $V_t$, $L_t$ and $K_t$ are determined by SDP-based algorithm in Section~\ref{secmainresult}.}
    \label{fig:sep}
\end{figure}
\fi

The role of the mechanism \eqref{eqvirtualsensor} is noteworthy. Recall that in the current problem setting in Figure 1, the state vector $\bx_t$ is directly observable by the decision policy.
The purpose of introducing an artificial mechanism \eqref{eqvirtualsensor} is to reduce data ``consumed'' by the decision policy while desired control performance is still attainable.  
Intuitively, the optimal mechanism \eqref{eqvirtualsensor} acquires just enough information from the state vector $\bx_t$ for control purposes and discards less important information.
Since the importance of information is a task-dependent notion, such a mechanism is designed jointly with other components in 2 and 3.
The mechanism \eqref{eqvirtualsensor} may not be a physical sensor mechanism, but rather be a mere computational procedure.  For this reason, we also call \eqref{eqvirtualsensor} a ``virtual sensor." 
A virtual sensor can also be viewed as an instantaneous lossy data-compressor in the context of networked LQG control \cite{silva2011framework,extendedversion}.
As shown in \cite{extendedversion}, the knowledge of the optimal virtual sensor can be used to design a dithered uniform quantizer with desired performance.

We also claim that data-processing components in 1-3 can be synthesized by a tractable computational procedure based on SDP summarized below. The procedure is sequential, starting from controller design, followed by virtual sensor design and Kalman filter design. 

\begin{enumerate}[label=$\bullet$, itemsep = 2mm, topsep = 2mm, leftmargin = 3mm]
\item {\bf Step~1} (Controller design) Determine  feedback control gains $K_t$ via the backward Riccati recursion:
\begin{subequations}
\label{backwardriccati}
\begin{align}
S_t&=\begin{cases} Q_t & \text{ if } t=T \\ Q_t+\Phi_{t+1} & \text{ if } t=1,\cdots, T-1 \end{cases} \\
\Phi_t&=A_t^\top (S_t-S_t B_t (B_t^\top S_t B_t + R_t)^{-1} B_t^\top S_t) A_t \\
K_t&= -(B_t^\top S_t B_t + R_t)^{-1} B_t^\top S_t A_t \label{backricK}\\
\Theta_t&= K_t^\top (B_t^\top S_t B_t + R_t) K_t.
\end{align}
\end{subequations}
Positive semidefinite matrices $\Theta_t$ will be used in Step~2.
\item {\bf Step~2} (Virtual sensor design) Let $\{P_{t|t}, \Pi_t\}_{t=1}^T$ be the optimal solution to a max-det problem:
\begin{subequations}
\label{optprob3}
\begin{align}
\min_{\{P_{t|t}, \Pi_t\}_{t=1}^T} & \quad \frac{1}{2} \sum\nolimits_{t=1}^T \log\det \Pi_t^{-1} + c_1 \\
\text{s.t.} \;\;\;\;\;& \quad \sum\nolimits_{t=1}^T \text{Tr}(\Theta_t P_{t|t}) + c_2 \leq D \\
& \quad  \Pi_t \succ  0,  \label{optprob3Pi}\\
& \quad  P_{1|1}\preceq P_{1|0}, P_{T|T}=\Pi_T,  \\
& \quad  P_{t+1|t+1}\preceq A_t P_{t|t}A_t^\top +W_t, \label{optprob3e} \\
&\hspace{1ex} \left[\!\! \begin{array}{cc}P_{t|t}\!-\!\Pi_t\!\!\! &\!\! P_{t|t}A_t^\top \\
A_tP_{t|t} \!\!\!&\!\! A_t P_{t|t}A_t^\top\!+\!W_t  \end{array}\!\!\right]\! \succeq\! 0. \label{optprob3f}
\end{align}
\end{subequations} 
The constraint (\ref{optprob3Pi}) is imposed for every $t=1, \cdots, T$, while (\ref{optprob3e}) and (\ref{optprob3f}) are for every $t=1, \cdots, T-1$.
Constants $c_1$ and $c_2$ are  given by
\begin{align*}
c_1&=\tfrac{1}{2}\log\det{P_{1|0}}+\tfrac{1}{2}\sum\nolimits_{t=1}^{T-1} \log \det W_t \\
c_2&=\text{Tr}(N_1P_{1|0})+\sum\nolimits_{t=1}^T \text{Tr}(W_t S_t).
\end{align*} 
Define signal-to-noise ratio matrices $\{\mathsf{SNR}_t\}_{t=1}^T$ by
\begin{align*}
 \mathsf{SNR}_t&\triangleq  P_{t|t}^{-1}-P_{t|t-1}^{-1}, \;\; t=1,\cdots, T \\
 P_{t|t-1}&\triangleq A_{t-1}P_{t-1|t-1}A_{t-1}^\top + W_{t-1}, \;\; t=2,\cdots, T 
\end{align*}
and set $r_t=\text{rank}(\mathsf{SNR}_t)$. 
Apply the singular value decomposition to find  $C_t\in \mathbb{R}^{r_t\times n_t}$ and $V_t\in \mathbb{S}_{++}^{r_t}$ such that
\begin{equation}
\label{cvconst}
 \mathsf{SNR}_t= C_t^\top V_t^{-1}C_t, \;\; t=1,\cdots, T.
\end{equation}
If $r_t=0$, $C_t$ and $V_t$ are null (zero dimensional) matrices.

\item {\bf Step~3} (Filter design) Determine the Kalman gains by
\begin{equation}
\label{eqkalmangain}
L_t=P_{t|t-1}C_t^\top (C_t P_{t|t-1}C_t^\top + V_t)^{-1}.
\end{equation}
Construct a Kalman filter by
\begin{subequations}
\label{eqkalmanfilter}
\begin{align}
&\hat{\bx}_t=\hat{\bx}_{t|t-1}+L_t(\by_t-C_t\hat{\bx}_{t|t-1}) \label{eqkalmanfilter1}\\
&\hat{\bx}_{t+1|t}=A_t\hat{\bx}_t+B_t \bu_t. \label{eqkalmanfilter2}
\end{align}
\end{subequations}
If $r_t=0$, $L_t$ is a null matrix and (\ref{eqkalmanfilter1}) becomes $\hat{\bx}_t=\hat{\bx}_{t|t-1}$.
\end{enumerate}

An optimization problem \eqref{optprob3} plays a key role in the proposed synthesis. Intuitively, \eqref{optprob3} ``schedules'' the optimal sequence of covariance matrices $\{P_{t|t}\}_{t=1}^T$ in such a way that there exists a virtual sensor mechanism to realize it and the required data-rate is minimized. The virtual sensor and the Kalman filter are designed later to realize the scheduled covariance.
\begin{theorem}%[Information-Frugal LQG Controller]
\label{maintheorem}
An optimal policy for the problem (\ref{mainprob}) exists if and only if the max-det problem (\ref{optprob3}) is feasible, and the optimal value of (\ref{mainprob}) coincides with the optimal value of (\ref{optprob3}).
If the optimal value of (\ref{mainprob}) is finite, an optimal policy can be realized by a virtual sensor, Kalman filter, and a certainty equivalence controller as shown in Figure~\ref{fig:sep}.
Moreover, each of these components can be constructed by an SDP-based algorithm summarized in Steps 1-3.
\end{theorem}
\begin{proof}
See Section~\ref{secderivation}.
\end{proof}

\begin{remark}
\label{remsingularw}
If $W_t$ is singular for some $t$, we suggest to factorize it as $W_t = F_tF_t^\top$ and use the following alternative
max-det problem instead of \eqref{optprob3}:
\begin{subequations}
\label{optprob3_singular}
\begin{align}
\min_{\{P_{t|t}, \Delta_t\}_{t=1}^T} & \quad \frac{1}{2} \sum\nolimits_{t=1}^T \log\det \Delta_t^{-1} + c_1 \\
\text{s.t.} \;\;\;\;\;& \quad \sum\nolimits_{t=1}^T \text{Tr}(\Theta_t P_{t|t}) + c_2 \leq D \\
& \quad  \Delta_t \succ  0,  \label{optprob3Pi_singular}\\
& \quad  P_{1|1}\preceq P_{1|0}, P_{T|T}=\Delta_T,  \\
& \quad  P_{t+1|t+1}\preceq A_t P_{t|t}A_t^\top +F_tF_t^\top, \label{optprob3e_singular} \\
&\hspace{1ex} \left[\!\! \begin{array}{cc}I \!-\!\Delta_t\!\!\! &\!\! F_t^\top \\
F_t \!\!\!&\!\! A_t P_{t|t}A_t^\top\!+\!F_tF_t^\top  \end{array}\!\!\right]\! \succeq\! 0. \label{optprob3f_singular}
\end{align}
\end{subequations} 
The constraint (\ref{optprob3Pi_singular}) is imposed for every $t=1, \cdots, T$, while (\ref{optprob3e_singular}) and (\ref{optprob3f_singular}) are for every $t=1, \cdots, T-1$.
Constants $c_1$ and $c_2$ are  given by $c_1=\tfrac{1}{2}\log\det{P_{1|0}}+\sum\nolimits_{t=1}^{T-1} \log |\det A_t |$ and $c_2=\text{Tr}(N_1P_{1|0})+\sum\nolimits_{t=1}^T \text{Tr}(F_t^\top S_t F_t)$.
%\begin{align*}
%c_1&=\tfrac{1}{2}\log\det{P_{1|0}}+\sum\nolimits_{t=1}^{T-1} \log |\det A_t | \\
%c_2&=\text{Tr}(N_1P_{1|0})+\sum\nolimits_{t=1}^T \text{Tr}(F_t^\top S_t F_t).
%\end{align*} 
This formulation requires that $A_t, t=1, ... ,T-1$ are non-singular matrices. Derivation
is omitted for brevity.
\end{remark}

\subsection{Time-invariant plants}
\label{secmainlti}
For time-invariant and infinite-horizon problems \eqref{eqltiplant} and \eqref{stationaryprob}, it can be shown that there exists an optimal policy with the same three-stage structure as in Figure~\ref{fig:3sep} in which all components are time-invariant. The optimal policy can be explicitly constructed by the following numerical procedure:

\begin{enumerate}[label=$\bullet$, itemsep = 2mm, topsep = 2mm, leftmargin = 3mm]
\item {\bf Step~1} (Controller design) Find the unique stabilizing solution to an algebraic Riccati equation
\begin{equation}
\label{algriccati}
A^\top SA-S-A^\top SB(B^\top SB+R)^{-1}B^\top SA+Q=0
\end{equation}
and determine the optimal feedback control gain by $K=-(B^\top SB+R)^{-1}B^\top SA$. Set $\Theta=K^\top (B^\top SB+R)K$.

\item {\bf Step~2} (Virtual sensor design) Choose $P$ and $\Pi$ as the solution to a max-det problem:
\begin{subequations}
\label{ltisdp}
\begin{align}
\min_{P, \Pi} & \quad \frac{1}{2} \log\det \Pi^{-1} + \frac{1}{2} \log \det W \\
\text{s.t.} & \quad \text{Tr}(\Theta P) + \text{Tr}(W S) \leq D, \\
& \quad  \Pi \succ  0,  \label{ltisdpPi}\\
& \quad  P\preceq A P A^\top +W, \label{ltisdpe} \\
&\hspace{1ex} \left[\!\! \begin{array}{cc}P-\Pi \!\!\! &\!\! PA^\top \\
AP \!\!\!&\!\! A PA^\top +W \end{array}\!\!\right]\! \succeq\! 0. \label{ltisdpf}
\end{align}
\end{subequations} 
Define $\tilde{P}\triangleq A PA^\top +W$, $\mathsf{SNR}\triangleq P^{-1}-\tilde{P}^{-1}$ and set $r=\text{rank}(\mathsf{SNR})$. Choose a virtual sensor $\by_t=C\bx_t+\bv_t, \; \bv_t\sim\mathcal{N}(0, V)$ with matrices $C\in\mathbb{R}^{r\times n}$ and $V\in\mathbb{S}_{++}^r$ such that $C^\top V^{-1}C=\mathsf{SNR}$.

\item {\bf Step~3} (Filter design) Design a time-invariant Kalman filter
\begin{align*}
&\hat{\bx}_t=\hat{\bx}_{t|t-1}+L(\bz_t-C\hat{\bx}_{t|t-1}) \\
&\hat{\bx}_{t+1|t}=A\hat{\bx}_t+B\bu_t
\end{align*}
with $L=\tilde{P}C^\top (C\tilde{P}C^\top+V)^{-1}$.
\end{enumerate}
\begin{theorem}
\label{theostationary}
An optimal policy for (\ref{stationaryprob}) exists if and only if a max-det problem (\ref{ltisdp}) is feasible, and the optimal value of  (\ref{stationaryprob}) coincides with that of (\ref{ltisdp}). 
Moreover, an optimal policy can be realized by a virtual sensor, Kalman filter, and a certainty equivalence controller as shown in Figure~\ref{fig:3sep}, all of which are time-invariant. Each of these components can be constructed by Steps~1-3.
\end{theorem}
\begin{proof} See Appendix~\ref{apptheostationary}.
\end{proof}
Theorem~\ref{theostationary} shows a noteworthy fact that $\mathsf{DI}(D)$ defined by \eqref{stationaryprob} admits a single-letter characterization, i.e., it can be evaluated by solving a finite-dimensional optimization problem \eqref{ltisdp}.

\subsection{Data-rate theorem for mean-square stabilization}
Theorem~\ref{theostationary} shows that $\mathsf{DI}(D)$ defined by \eqref{stationaryprob} admits a semidefinite representation  \eqref{ltisdp}.
By analyzing the structure of the optimization problem \eqref{ltisdp}, one can obtain a closed-from expression of the quantity $\lim_{D\rightarrow +\infty} \mathsf{DI}(D)$.
Notice that this quantity can be interpreted as the minimum data-rate (measured in directed information) required for mean-square stabilization.
The next corollary shows a connection between our study in this paper and the data-rate theorem by Nair and Evans \cite{nair2004stabilizability}.
\begin{corollary}
\label{cordatarate}
Denote by $\sigma_+(A)$ the set of eigenvalues $\lambda_i$ of $A$ such that $|\lambda_i|\geq 1$ counted with multiplicity. Then,
\begin{equation}
\label{eqmsdatarate}
\lim_{D\rightarrow +\infty} \!\mathsf{DI}(D)=\sum_{\lambda_i \in \sigma_+(A)} \log|\lambda_i|. 
\end{equation}
\end{corollary}
\begin{proof}
See Appendix \ref{appdatarate}.
\end{proof}

Corollary~\ref{cordatarate} indicates that the minimal data-rate for mean-square stabilization does not depend on the noise property $W$. This result is consistent with the observation in \cite{nair2004stabilizability}. However, as is clear from the semidefinite representation \eqref{ltisdp}, minimal data-rate to achieve control performance $J_t\leq D$ depends on $W$ when $D$ is finite.

Corollary~\ref{cordatarate} has a further implication that there exists a quantized LQG control scheme implementable over a noiseless binary channel such that data-rate is arbitrarily close to \eqref{eqmsdatarate} and the closed-loop systems in stabilized in the mean-square sense. See \cite{tanaka2016optimal} for details.

{\upd Mean-square stabilizability of linear systems by quantized feedback with Markovian packet losses is considered in \cite{you2011minimum}, where a necessary and sufficient condition in terms of nominal data-rate and packet dropping probability is obtained.
Although directed information is not  used in  \cite{you2011minimum}, it would be an interesting future work to compute $\lim_{T\rightarrow \infty} \frac{1}{T}I(X^T\rightarrow U^T)$ under the stabilization scheme proposed there and study how it is compared to the right hand side of \eqref{eqmsdatarate}.
}

\ifdefined\DOUBLECOLUMN	
		\begin{figure}[t]
		\centering
		\includegraphics[width=0.8\columnwidth]{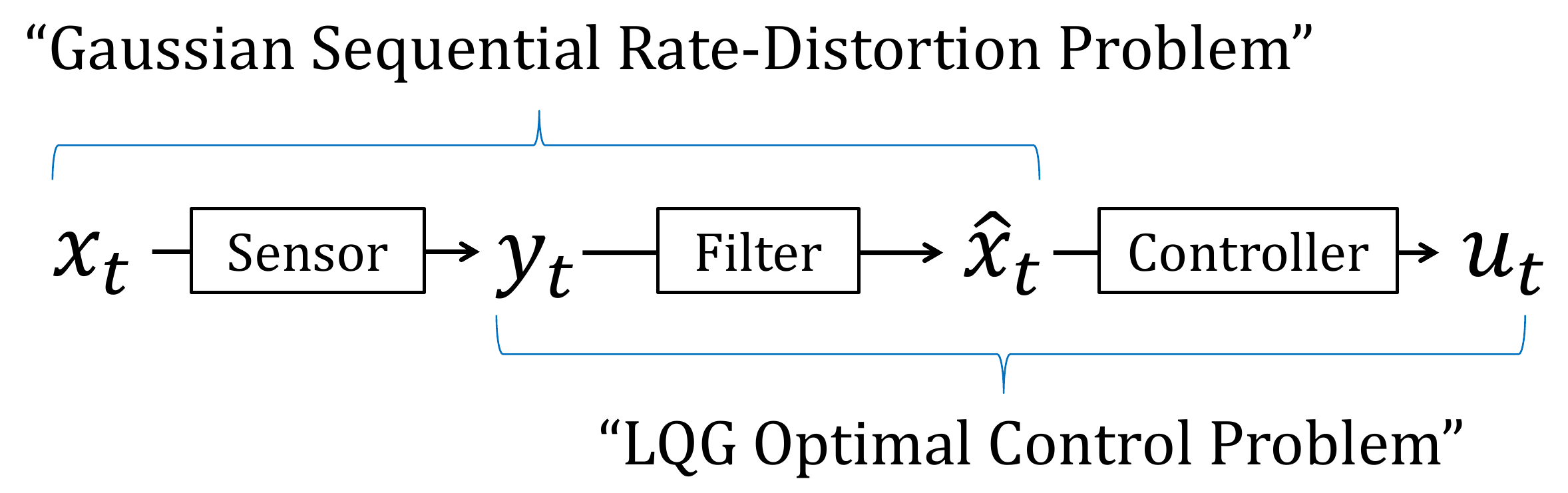}
		\caption{Sensor-filter-controller separation principle: integration of the sensor-filter and filter-controller separation principles.}
		\label{fig:3sep}
		\vspace{-3ex}
	\end{figure}
\fi
\ifdefined\SINGLECOLUMN
		\begin{figure}[t]
		\centering
		\includegraphics[width=0.6\columnwidth]{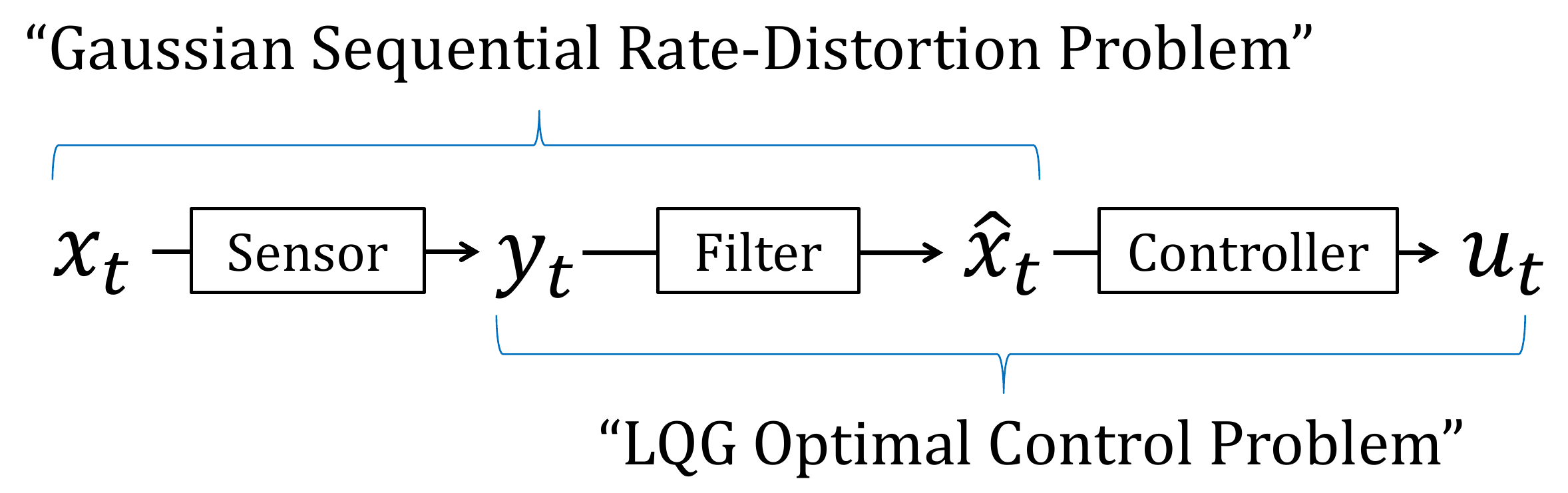}
		\caption{Sensor-filter-controller separation principle: integration of the sensor-filter and filter-controller separation principles.}
		\label{fig:3sep}
	\end{figure}
\fi

\subsection{Connections to existing results}
	\label{seccontribution}
		
	We first note that the ``sensor-filter-controller" structure identified by Theorem~\ref{maintheorem} is not a simple consequence of the filter-controller separation principle in the standard LQG control theory \cite{witsenhausen1971separation}. {\upd Unlike the standard framework in which a sensor mechanism \eqref{eqvirtualsensor} is given \emph{a priori}, in \eqref{mainprob} we \emph{design} a sensor mechanism jointly with other components.
	Intuitively, a sensor mechanism in our context plays a role to reduce information flow from $\by_t$ to $\bx_t$.  
 The proposed sensor design algorithm has already appeared in \cite{1503.01848}. In this paper we strengthen the result by showing that the designed linear sensor turns out to be optimal among all nonlinear (Borel measurable) sensor mechanisms.
	
Information-theoretic fundamental limitations of feedback control are derived in
\cite{iglesias2001analogue,zang2003nonlinear,elia2004bode,martins2008}
 via the ``Bode-like'' integrals. However, the connection between \cite{iglesias2001analogue,zang2003nonlinear,elia2004bode,martins2008} and our problem \eqref{mainprob} is not straightforward, and the structural result shown in Figure~\ref{fig:sep} does not appear in \cite{iglesias2001analogue,zang2003nonlinear,elia2004bode,martins2008}.}
Also, we note that our problem formulation \eqref{mainprob} is different from networked LQG control problem over Gaussian channels \cite{tatikonda2004,braslavsky2007feedback,charalambous2008} where a model of Gaussian channel is given \emph{a priori}. In such problems, linearity of the optimal policy is already reported \cite[Ch.10,11]{yuksel2013stochastic}.

	It should  be noted that problem \eqref{mainprob} is closely related to the sequential rate-distortion problem (also called zero-delay or non-anticipative rate-distortion problem)  \cite{TatikondaThesis,1411.7632,charalambous2014nonanticipative}. In the Gaussian sequential rate-distortion problem where the plant \eqref{eqsystem} is an uncontrolled system (i.e., $\bu_t=0$), it can be shown that the optimal policy can be realized by a two-stage ``sensor-filter'' structure \cite{1411.7632}.  However, the same result is not known for the case in which feedback controllers must be designed simultaneously. 
	{\upd Relevant papers towards this direction include \cite{rezaei2006rate,charalambous2014nonanticipative,charalambous2014optimization}, where Csisz{\'a}r's formulation of rate-distortion functions \cite{csiszar1974extremum} is extended to the non-anticipative regime.} In particular, \cite{charalambous2014optimization} considers non-anticipative rate-distortion problems with feedback. 
In \cite{sims2003implications} and \cite{shafieepoorfard2013rational}, LQG control problems with information-theoretic costs similar to \eqref{mainprob} are considered.
	However, the optimization problem considered in these papers are not equivalent to \eqref{mainprob}, and the structural result shown in Figure~\ref{fig:3sep} does not appear.

{\upd
In a very recent paper \cite[Lemma 3.1]{silva2013characterization}, 
it is independently reported that 
the optimal policy for \eqref{mainprob} can be realized by an additive white Gaussian noise (AWGN) channel and linear filters.
While this result is compatible to ours, it is noteworthy that the proof technique there is different from ours and is based on fundamental inequalities for directed information obtained in \cite{derpich2013fundamental}.
In comparison to \cite{silva2013characterization},
we additionally prove that
the optimal control policy can be realized by a state space model with a three-stage structure (shown in Figure~\ref{fig:sep}, \ref{fig:3sep}), which appears to be a new observation to the best of our knowledge.
}

	The SDP-based algorithms to solve \eqref{mainprob}, \eqref{stationaryprob} and \eqref{partiallyobservableprob} are newly developed in this paper, using the techniques presented in \cite{1411.7632} and \cite{1503.01848}.
	{\upd
Due to the lack of analytical expression of the optimal policy (especially for MIMO and time-varying plants), 
the use of optimization-based algorithms seems critical. 
In \cite{stavrou2016filtering}, an iterative water-filling algorithm is proposed for a highly relevant problem.
	In this paper, the main algorithmic tool is SDP, which allows us to generalize the results in \cite{silva2011framework,silva2011achievable,silva2013characterization} to MIMO and time-varying settings. 
	}

\section{Example}
\label{secexample}

In this section, we consider a simple numerical example to demonstrate the SDP-based control design presented in Section~\ref{secmainlti}. Consider a time-invariant plant \eqref{eqltiplant} with randomly generated matrices
\[
A\!=\!\!{\small \left[\!\! \begin{array}{cccc}0.12 \!\!&\!\!   0.63 \!\!&\!\!  -0.52 \!\!&\!\! 0.33\\
    0.26  \!\!&\!\! -1.28  \!\!&\!\!  1.57   \!\!&\!\! 1.13 \\
   -1.77   \!\!&\!\! -0.30    \!\!&\!\! 0.77   \!\!&\!\! 0.25\\
   -0.16   \!\!&\!\!  0.20  \!\!&\!\!  -0.58   \!\!&\!\!  0.56\end{array}\!\!\right] }\!,
%F\!=\!\!{\footnotesize\left[\!\! \begin{array}{cccc}-1.84\!\!&\!\!  0.55  \!\!&\!\!  1.07  \!\!&\!\!0.32\\
%    0.46 \!\!&\!\!  0.02  \!\!&\!\!    0   \!\!&\!\! 2.31\\
%   -0.78   \!\!&\!\!0.16 \!\!&\!\!  -0.53  \!\!&\!\!  1.05\\
%   -0.49 \!\!&\!\! 1.18 \!\!&\!\!  -1.19 \!\!&\!\! 0.22\end{array}\!\!\right]}
W\!=\!\!{\small\left[\!\! \begin{array}{cccc}  4.94 \!\!&\!\!   -0.10 \!\!&\!\!    1.29  \!\!&\!\!   0.35\\
     \!\!&\!\!   5.55  \!\!&\!\!   2.07 \!\!&\!\!    0.31\\
     \!\!&\!\!     \!\!&\!\!   2.02   \!\!&\!\!  1.43\\
   \text{sym.}  \!\!&\!\!      \!\!&\!\!     \!\!&\!\!   3.10\end{array}\!\!\right]}
  \]
 \[  
B={\small\left[\! \begin{array}{cccc} 0.66  \!&\! -0.58  \!&\!  0.03   \!&\! -0.20\\
    2.61  \!&\!-0.91   \!&\!  0.87  \!&\! -0.07\\
   -0.64  \!&\!  -1.12  \!&\!  -0.19  \!&\!  0.61\\
    0.93   \!&\!  0.58  \!&\!  -1.18 \!&\!  -1.21\end{array}\!\right]}, 
\]
and the optimization problem \eqref{stationaryprob} with $Q=I$ and $R=I$.
By solving (\ref{ltisdp}) with various $D$, we obtain the rate-performance trade-off curve shown in Figure~\ref{fig:tradeoff} (top left).
The vertical asymptote $D=\text{Tr}(WS)$ corresponds to the best achievable control performance when unrestricted amount of information about the state is available. 
This corresponds to the performance of the state-feedback linear-quadratic regulator (LQR). 
The horizontal asymptote $\sum_{\lambda_i\in\sigma_{+}(A)} \log |\lambda_i|=1.169$ [bits/sample] is the minimum data-rate to achieve mean-square stability. Figure~\ref{fig:tradeoff} (bottom left) shows the rank of $\mathsf{SNR}$ matrices as a function of $D$. Since $\mathsf{SNR}$ is computed numerically by an SDP solver with some finite numerical precision, $\text{rank}(\mathsf{SNR})$ is obtained by truncating singular values smaller than $0.1$\% of the maximum singular value. Figure~\ref{fig:tradeoff} (right) shows selected singular values at $D=33,40$ and $80$.
Observe the phase transition (rank dropping) phenomena. The optimal dimension of the sensor output changes as $D$ changes.

 \ifdefined\DOUBLECOLUMN	
\begin{figure}[t]
    \centering
    \begin{minipage}{0.66\columnwidth}
        \centering
        \includegraphics[width=\linewidth]{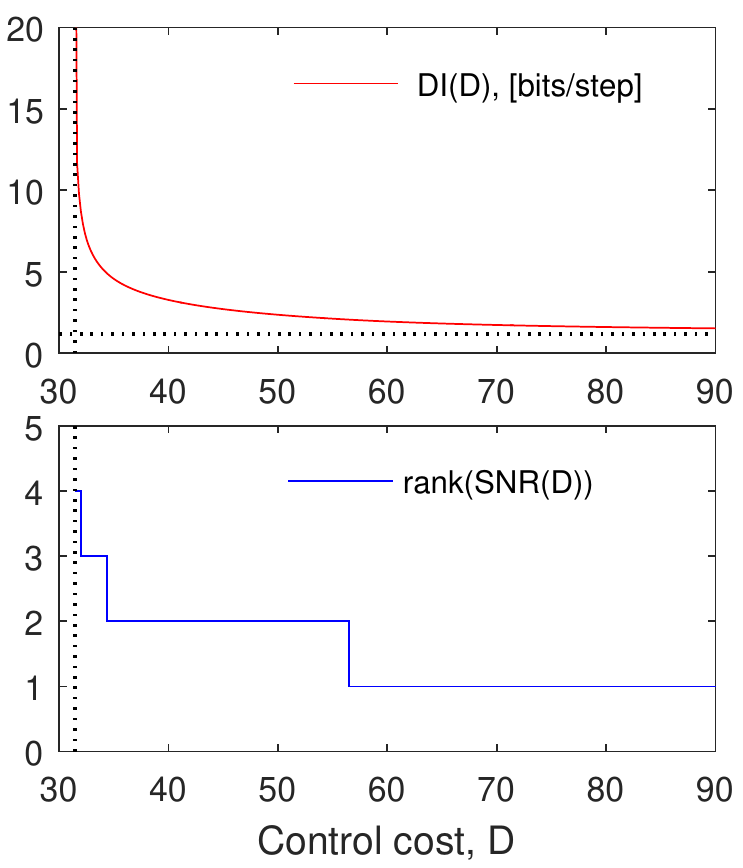}
    \end{minipage}%
    \begin{minipage}{0.33\columnwidth}
        \centering
        \includegraphics[width=\linewidth]{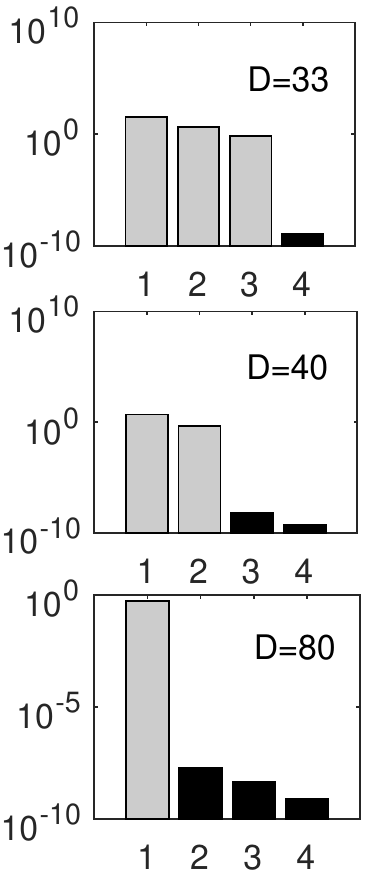}
    \end{minipage}
            \caption{(Top left) Data rate $\mathsf{DI}(D)$ [bits/step] required to achieve control performance $D$. (Bottom left) Rank of $\mathsf{SNR(D)}$, evaluated after truncating singular values smaller than $0.1$\% of the maximum singular value. (Right) Singular values of $\mathsf{SNR(D)}$ evaluated at $D=33$, $40$ and $80$. Truncated singular values are shown in block bars. An SDP solver SDPT3 \cite{toh1999sdpt3} with YALMIP \cite{lofberg2004yalmip} interface is used.}
        \label{fig:tradeoff}
        \vspace{-3ex}
\end{figure}
\fi
 \ifdefined\SINGLECOLUMN	
\begin{figure}[t]
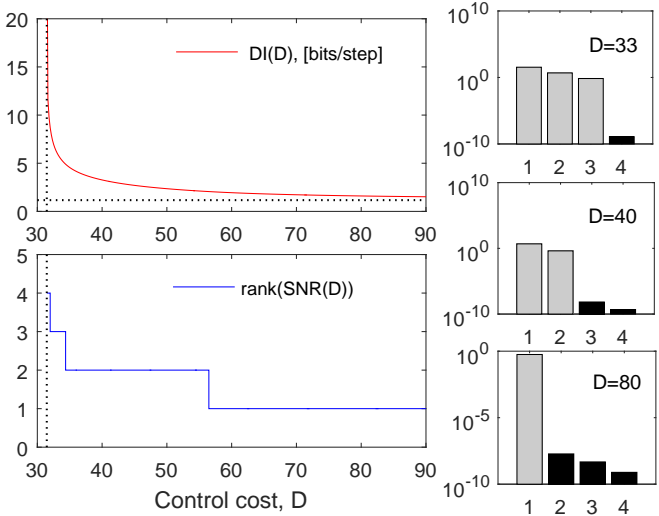

    \centering
    \begin{minipage}{0.4\columnwidth}
        \centering
        \includegraphics[width=\linewidth]{curve_rank.pdf}
    \end{minipage}%
    \begin{minipage}{0.2\columnwidth}
        \centering
        \includegraphics[width=\linewidth]{sv.pdf}
    \end{minipage}
            \caption{(Top left) Data rate $\mathsf{DI}(D)$ [bits/step] required to achieve control performance $D$. (Bottom left) Rank of $\mathsf{SNR(D)}$, evaluated after truncating singular values smaller than $0.1$\% of the maximum singular value. (Right) Singular values of $\mathsf{SNR(D)}$ evaluated at $D=33$, $40$ and $80$. Truncated singular values are shown in block bars. An SDP solver SDPT3 \cite{toh1999sdpt3} with YALMIP \cite{lofberg2004yalmip} interface is used.}
        \label{fig:tradeoff}
\end{figure}
\fi

Specifically, the minimum data-rate to achieve control performance $D=33$ is found to be $6.133$ [bits/sample]. The optimal sensor mechanism $\by_t=C\bx_t+\bv_t, \bv_t\sim\mathcal{N}(0,V)$ to achieve this performance is given by
\[
 C\!=\!\!{\small\left[\!\! \begin{array}{cccc}    -0.864   \!\!&\!\!  0.258 \!\!&\!\!   -0.205  \!\!&\!\!  -0.382\\
   -0.469  \!\!&\!\!  -0.329  \!\!&\!\!  0.662   \!\!&\!\!  0.483\\
   -0.130  \!\!&\!\!   0.332   \!\!&\!\!-0.502    \!\!&\!\! 0.780  \end{array}\!\!\right]}\!,
V\!=\!\!{\small\left[\!\! \begin{array}{ccc}      0.029   \!\! \!\!&\!\!      0     \!\!&\!\!     0\\
         0  \!\!&\!\!\!\!   0.208    \!\! \!\!&\!\!     0\\
         0     \!\!&\!\!     0    \!\!&\!\!\!\! 1.435   \end{array}\!\!\right]}.
\]
If $D=40$, required data-rate is $3.266$ [bits/sample] and the optimal sensor is given by
  \[
  C\!=\!\!{\small\left[\! \begin{array}{cccc} -0.886    \!\!&\!\! 0.241  \!\!&\!\! -0.170\!\!&\!\!   -0.359 \\
   -0.431 \!\!&\!\!  -0.350 \!\!&\!\!   0.647  \!\!&\!\!  0.523
    \end{array}\!\right]}, \;\;
  V\!=\!\!{\small\left[\! \begin{array}{cc}  0.208  \!\!\!\!&\!\!       0\\
         0  \!\!&\!\!\!\!  2.413 \end{array}\!\right]}.
    \]     
Similarly, minimum data-rate to achieve $D=80$ is $1.602$ [bits/sample], and this is achieved by a sensor mechanism with
    \[  
    C\!=\!\!{\small\left[\! \begin{array}{cccc}-0.876  \!\!&\!\!   0.271  \!\!&\!\!  -0.169  \!\!&\!\!  -0.362 \end{array}\!\right]},\;\;
    V=\!\!{\small \begin{array}{c}1.775\end{array}\!\!}.
    \]
 Figure~\ref{fig:sim} shows the closed-loop responses of the state trajectories simulated in each scenario.
 
 \ifdefined\DOUBLECOLUMN	
 \begin{figure}[t]
    \centering
    \includegraphics[width=\columnwidth]{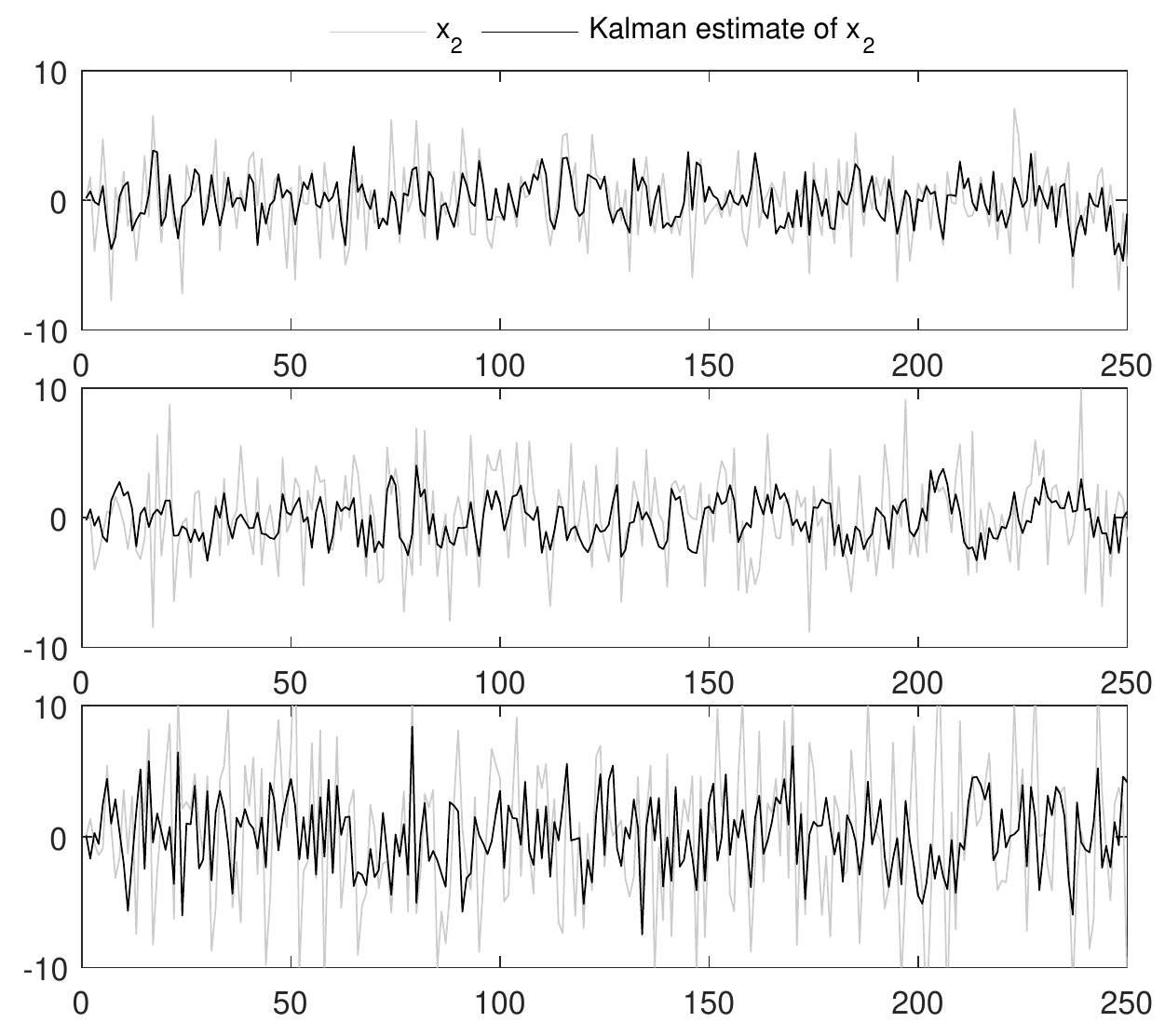}
    \caption{Closed-loop performances of the controllers designed for $D=33$ (top), $D=40$ (middle), and $D=80$ (bottom). Trajectories of the second component of the state vector and their Kalman estimates are shown. }
    \label{fig:sim}
    \vspace{-3ex}
\end{figure}
\fi
\ifdefined\SINGLECOLUMN	
 \begin{figure}[t]
    \centering
    \includegraphics[width=0.6\columnwidth]{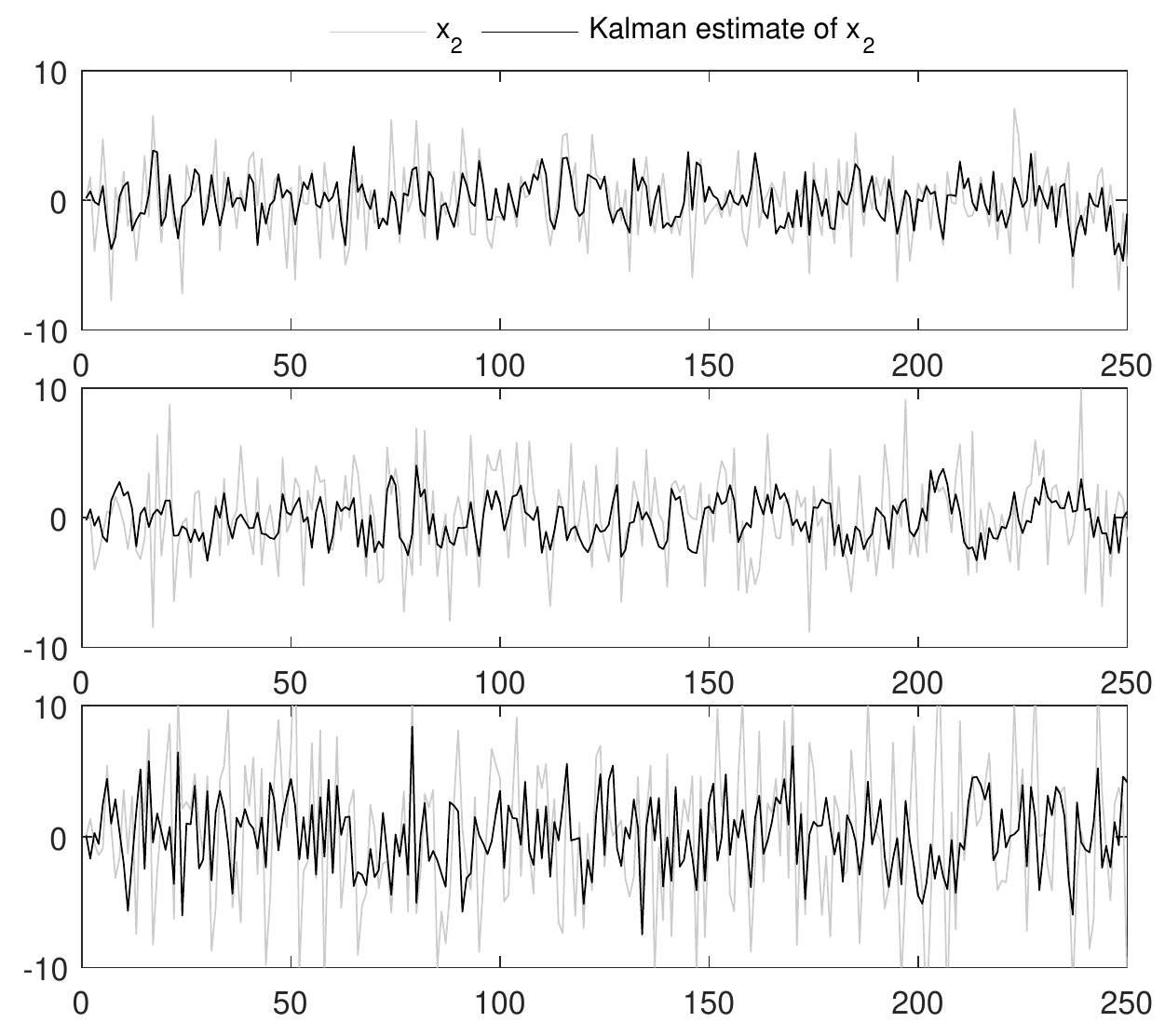}
    \caption{Closed-loop performances of the controllers designed for $D=33$ (top), $D=40$ (middle), and $D=80$ (bottom). Trajectories of the second component of the state vector and their Kalman estimates are shown. }
    \label{fig:sim}
\end{figure}
\fi

\section{Derivation of Main Result}
\label{secderivation}
This section is devoted to prove Theorem~\ref{maintheorem}. 
We first define subsets $\Gamma_0$, $\Gamma_1$, and $\Gamma_2$ of the policy space $\Gamma$ as follows. 
\begin{itemize}%[leftmargin=2ex]
\item[$\Gamma_0$]: The space of policies with three-stage separation structure explained in Section~\ref{secmainresult}.
\item[$\Gamma_1$]: The space of linear sensors without memory followed by linear deterministic feedback control. Namely, a policy $\PP(u^T\|x^T)$ in $\Gamma_1$ can be expressed as a composition of
\begin{equation}
\label{eqgamma1awgn} \by_t=C_t\bx_t+\bv_t, \;\; \bv_t\sim \mathcal{N}(0,V_t)
\end{equation}
and $\bu_t=l_t(\by^t)$, where $C_t\in \mathbb{R}^{r_t\times n_t}$, $r_t$ is some nonnegative integer, $V_t \succ 0$, and $l_t(\cdot)$ is a linear map.
\item[$\Gamma_2$]: The space of linear policies without state memory. Namely,  a policy $\PP(u^T\|x^T)$ in $\Gamma_2$ can be expressed as
\begin{equation}
\label{linpolicy2}
\bu_t=M_t\bx_t+N_t\bu^{t-1}+\bg_t, \;\; \bg_t\sim \mathcal{N}(0,G_t)
\end{equation}
with some matrices $M_t, N_t$, and $G_t\succeq 0$.
\end{itemize}
%It is easy to check that $\Gamma_0\subset\Gamma_1\subset\Gamma_2\subset\Gamma$.

\subsection{Proof outline}
To prove Theorem~\ref{maintheorem}, we establish a chain of inequalities:
\begin{subequations}
\label{ineqchain}
\begin{align}
&\inf_{\gamma\in\Gamma: J_\gamma\leq D} I_\gamma(\bx^T\rightarrow \bu^T) \label{ineqchain1}\\
\geq & \inf_{\gamma\in\Gamma: J_\gamma\leq D} \sum\nolimits_{t=1}^T I_\gamma(\bx_t;\bu_t|\bu^{t-1}) \label{ineqchain2}\\
\geq & \inf_{\gamma\in\Gamma_2: J_\gamma\leq D} \sum\nolimits_{t=1}^T I_\gamma(\bx_t;\bu_t|\bu^{t-1}) \label{ineqchain3}\\
\geq & \inf_{\gamma\in\Gamma_1: J_\gamma\leq D} \sum\nolimits_{t=1}^T I_\gamma(\bx_t;\by_t|\by^{t-1}) \label{ineqchain4}\\
\geq & \inf_{\gamma\in\Gamma_0: J_\gamma\leq D} \sum\nolimits_{t=1}^T I_\gamma(\bx_t;\by_t|\by^{t-1}) \label{ineqchain4_2}\\
\geq & \inf_{\gamma\in\Gamma_0: J_\gamma\leq D} I_\gamma(\bx^T\rightarrow \bu^T). \label{ineqchain5}
\end{align}
\end{subequations}
Since $\Gamma_0\subset\Gamma$, clearly (\ref{ineqchain1}) $\leq$ (\ref{ineqchain5}).
Thus, showing the above chain of inequalities proves that all quantities in (\ref{ineqchain}) are equal. This observation implies that 
the search for an optimal solution to our main problem (\ref{mainprob}) can be restricted to the class $\Gamma_0$ without loss of performance.
The first inequality (\ref{ineqchain2}) is immediate from the definition of directed information. 
We prove inequalities  (\ref{ineqchain3}), (\ref{ineqchain4}), (\ref{ineqchain4_2}) and (\ref{ineqchain5}) in subsequent subsections 
\ref{subsecproof_c}, \ref{subsecproof_d}, \ref{subsecproof_d2} and \ref{subsecproof_e}.
It will follow from the proof of inequality (\ref{ineqchain5}) that an optimal solution to  (\ref{ineqchain4_2}), if exists, is also an optimal solution to (\ref{ineqchain5}).
In particular, this implies that an optimal solution to the original problem \eqref{ineqchain1}, if exists, can be found by solving a simplified problem (\ref{ineqchain4_2}). This observation establishes the sensor-filter-controller separation principle depicted in Figure~\ref{fig:sep}.

Then, we focus on solving problem (\ref{ineqchain4_2}) in Subsection \ref{subsec12dsdp}.  We show that problem (\ref{ineqchain4_2}) can be reformulated as an optimization problem in terms of $\mathsf{SNR}_t\triangleq C_t^\top V_t^{-1} C_t$, which is further converted to an SDP problem.

\subsection{Proof of inequality (\ref{ineqchain3})}
\label{subsecproof_c}
We will show that for every $\gamma_\PP=\{\PP(u_t|x^t,u^{t-1})\}_{t=1}^T\in\Gamma$ that attains a finite objective value in (\ref{ineqchain2}), there exists $\gamma_\QQ=\{\QQ(u_t|x^t,u^{t-1})\}_{t=1}^T\in\Gamma_2$ such that $J_\PP=J_\QQ$ and 
\[
\sum\nolimits_{t=1}^T I_\PP(\bx_t;\bu_t|\bu^{t-1})\geq 
\sum\nolimits_{t=1}^T I_\QQ(\bx_t;\bu_t|\bu^{t-1})
\]
where subscripts of $I$ and $J$ indicate probability measures on which these quantities are evaluated.
Without loss of generality, we assume $\PP(x^{T+1}, u^T)$ has zero-mean.
Otherwise, we can consider an alternative policy $\gamma_{\tilde{\PP}}=\{\tilde{\PP}(u_t|x^t,u^{t-1})\}_{t=1}^T$, where
\[
\tilde{\PP}(u_t|x^t\!,u^{t-1})\!\triangleq\!\PP(u_t\!+\mathbb{E}_\PP(\bu_t)|x^t\!+\mathbb{E}_\PP(\bx^t),u^{t-1}\!+\mathbb{E}_\PP(\bu^{t-1}))
\]
which generates a zero-mean joint distribution $\tilde{\PP}(x^{T+1}, u^T)$.
We have $I_{\tilde{\PP}}=I_\PP$ in view of the translation invariance of mutual information, and $J_{\tilde{\PP}}\leq J_\PP$ due to the fact that the cost function is quadratic.

First, we consider a zero-mean, jointly Gaussian probability measure $\GG(x^{T+1}, u^T)$ having the same covariance matrix as $\PP(x^{T+1}, u^T)$. 
\begin{lemma}
\label{lemmaIpg}
The following inequality holds whenever the left hand side is finite.
\begin{equation}
\label{lemipg}
\sum\nolimits_{t=1}^T I_\PP(\bx_t;\bu_t|\bu^{t-1})\geq \sum\nolimits_{t=1}^T I_\GG(\bx_t;\bu_t|\bu^{t-1})
\end{equation}
\vspace{-2ex}
\end{lemma}
\begin{proof}
See Appendix~\ref{appproofPG}.
\end{proof}

Next, we are going to construct a policy $\gamma_\QQ=\{\QQ(u_t|x^t,u^{t-1})\}_{t=1}^T\in\Gamma_2$ using a jointly Gaussian measure $\GG(x^{T+1},u^T)$.
Let $E_t\bx_t+F_t\bu^{t-1}$ be the least mean-square error estimate of $\bu_t$ given $(\bx_t,\bu^{t-1})$ in $\GG(x^{T+1}, u^T)$, and let $V_t$ be the resulting estimation error covariance matrix. Define a stochastic kernel $\QQ(u_t|x_t,u^{t-1})$ by $\QQ(u_t|x_t,u^{t-1})=\mathcal{N}(E_t\bx_t+F_t\bu^{t-1}, V_t)$.
By construction, $\QQ(u_t|x_t,u^{t-1})$ satisfies\footnote{Equation $\diff \PP(x,y)=\diff\PP(y|x)\diff\PP(x)$ is a short-hand notation for $\PP(B_X \times B_Y)=\int_{B_X}\PP(B_Y|x)\diff\PP(x)\; \forall B_X\in\mathcal{B}_\mathcal{X}, B_Y\in \mathcal{B}_\mathcal{Y}$.}
\begin{equation}
\label{defconds}
\diff\GG(x_t, u^t)=\diff\QQ(u_t|x_t,u^{t-1})\diff\GG(x_t, u^{t-1}).
\end{equation}
Define $\QQ(x^{T+1}, u^T)$ recursively by
\begin{align}
\diff\QQ(x^t, u^{t-1})&=\diff\PP(x_t|x_{t-1},u_{t-1})\diff\QQ(x^{t-1}, u^{t-1}) \label{defs1}\\
\diff\QQ(x^t, u^t)&=\diff\QQ(u_t|x_t,u^{t-1})\diff\QQ(x^t, u^{t-1}) \label{defs2}
\end{align}
where $\PP(x_t|x_{t-1},u_{t-1})$ is a stochastic kernel defined by  (\ref{eqsystem}).
The following identity holds between two Gaussian measures $\GG(x^{T+1}, u^T)$ and $\QQ(x^{T+1}, u^T)$.
\begin{lemma}
\label{lemmags}
$\GG(x_{t+1}, u^t)=\QQ(x_{t+1}, u^t) \;\;  \forall t=1,\cdots,T.$
\end{lemma}
\begin{proof}
See Appendix~\ref{appprooflemmags}.
\end{proof}

We are now ready to prove (\ref{ineqchain3}).
First, replacing a policy $\gamma_\PP$ with a new policy $\gamma_\QQ$ does not change the LQG control cost.
\begin{subequations}
\begin{align}
J_{\gamma_\PP}=&\int \left( \|x_{t+1}\|_{Q_t}^2+\|u_t\|_{R_t}^2 \right) \diff\PP(x_{t+1}, u^t) \nonumber \\
=&\int \left( \|x_{t+1}\|_{Q_t}^2+\|u_t\|_{R_t}^2 \right) \diff\GG(x_{t+1}, u^t) \label{Jpgs1} \\
=&\int \left( \|x_{t+1}\|_{Q_t}^2+\|u_t\|_{R_t}^2 \right) \diff\QQ(x_{t+1}, u^t) \label{Jpgs2} \\
=&J_{\gamma_\QQ}. \nonumber
\end{align}
\end{subequations}
Equality (\ref{Jpgs1}) holds since $\PP$ and $\GG$ have the same second order moments. Step (\ref{Jpgs2}) follows from Lemma \ref{lemmags}.
Second, replacing $\gamma_\PP$ with $\gamma_\QQ$ does not increase the information cost.
\begin{subequations}
\begin{align}
\sum\nolimits_{t=1}^T I_\PP(\bx_t;\bu_t|\bu^{t-1}) &\geq \sum\nolimits_{t=1}^T I_\GG(\bx_t;\bu_t|\bu^{t-1}) \label{ineqipgs1}\\
&=\sum\nolimits_{t=1}^T I_\QQ(\bx_t;\bu_t|\bu^{t-1}). \label{ineqipgs2}
\end{align}
\end{subequations}
The inequality (\ref{ineqipgs1}) is due to Lemma \ref{lemmaIpg}. In (\ref{ineqipgs2}), $I_\GG(\bx_t;\bu_t|\bu^{t-1})=I_\QQ(\bx_t;\bu_t|\bu^{t-1})$ holds for every $t=1,\cdots, T$ because of Lemma \ref{lemmags}.

\subsection{Proof of inequality (\ref{ineqchain4})}
\label{subsecproof_d}
Given a policy $\gamma_2\in\Gamma_2$, we are going to construct a policy $\gamma_1\in  \Gamma_1$ such that $J_{\gamma_1}=J_{\gamma_2}$ and
\begin{equation}
\label{igamma12}
I_{\gamma_2}(\bx_t;\bu_t|\bu^{t-1})=I_{\gamma_1}(\bx_t;\by_t|\by^{t-1})
\end{equation}
for every $ t=1,\cdots,T$. Let $\gamma_2\in\Gamma_2$ be given by
\[
\bu_t=M_t\bx_t+N_t\bu^{t-1}+\bg_t, \;\; \bg_t\sim\mathcal{N}(0,G_t).
\]
Define $\tilde{\by}_t\triangleq M_t \bx_t+\bg_t$. If we write $N_t\bu^{t-1}=N_{t,t-1}\bu_{t-1}+\cdots+N_{t,1}\bu_1$, it can be seen that $\bu^t$ and $\tilde{\by}^t$ are related by an invertible linear map
\begin{equation}
\label{eqmatrixytilde}
\left[\begin{array}{c} \vspace{-1ex}
\tilde{\by}_1 \\ \vspace{-1ex}
\vdots \\ \vdots \\ \tilde{\by}_t
\end{array}\right]
=\left[\begin{array}{cccc} \vspace{-1ex}
I \!&\! 0 \!&\!\! \cdots \!&\! 0 \\ \vspace{-1ex}
-N_{2,1} \!&\! I \!&\!\!  \!&\! \vdots \\ 
\vdots \!&\!  \!&\!\! \ddots \!&\! 0 \\
-N_{t,1} \!&\! \cdots \!&\!\! -N_{t,t-1} \!&\! I
\end{array}\right]\left[\begin{array}{c} \vspace{-1ex}
\bu_1 \\ \vspace{-1ex}
\vdots \\ \vdots \\ \bu_t
\end{array}\right]
\end{equation}
for every $t=1,\cdots,T$. 
Hence,
\begin{align}
I(\bx_t;\bu_t|\bu^{t-1})&=I(\bx_t;\tilde{\by}_t+N_t\bu^{t-1}|\tilde{\by}^{t-1},\bu^{t-1}) \nonumber \\
&=I(\bx_t;\tilde{\by}_t|\tilde{\by}^{t-1}). \label{Ixuxytilde}
\end{align}
{\upd
Let $G_t=E_t^\top V_t E_t$ be the (thin) singular value decomposition.
Since we assume \eqref{Ixuxytilde} is bounded, we must have
\begin{equation}
\label{eqimincl}
\text{Im}(M_t)\subseteq \text{Im}(G_t)=\text{Im}(E_t^\top).
\end{equation}
Otherwise, the component of $\bu_t$ in $\text{Im}(G_t)^\perp$ depends deterministically on $\bx_t$ and \eqref{Ixuxytilde} is unbounded.
Now, define
$
\by_t  \triangleq  E_t \tilde{\by}_t  = E_t M_t \bx_t + E_t \bg_t, \;\; \bg_t\sim\mathcal{N}(0,G_t).$
Then, we have
\begin{align*}
E_t^\top \by_t & =  E_t^\top E_t M_t \bx_t + E_t^\top E_t \bg_t, \;\; \bg_t\sim\mathcal{N}(0,G_t) \\
& = M_t \bx_t + \bg_t  = \tilde{\by}_t.
\end{align*}
In the second line, we used the facts that $E_t^\top E_t M_t=M_t$ and $E_t^\top E_t \bg_t=\bg_t$ under \eqref{eqimincl}.
Thus, we have $\by_t= E_t \tilde{\by}_t$ and $\tilde{\by}_t=E_t^\top \by_t$. This implies that $\by_t$ and $\tilde{\by}_t$ contain statistically equivalent information, and that
}
\begin{equation}
\label{Ixytildexy}
I(\bx_t;\tilde{\by}_t|\tilde{\by}^{t-1})=I(\bx_t;\by_t|\by^{t-1}).
\end{equation}
Also, since $\bu_t$  depends linearly on $\tilde{\by}^t$ by (\ref{eqmatrixytilde}), there exists a linear map $l_t$ such that
\begin{equation}
\label{eqlinearmapu}
\bu_t=l_t(\by^t).
\end{equation}
Setting $C_t \triangleq E_tM_t$, construct a policy $\gamma_1\in\Gamma_1$ using $\by_t\triangleq E_t \tilde{\by}_t=C_t\bx_t+\bv_t$ with $\bv_t\sim\mathcal{N}(0,V_t)$ and a linear map (\ref{eqlinearmapu}).
Since joint distribution $\PP(x^{T+1},u^T)$ is the same under $\gamma_1$ and $\gamma_2$, we have $J_{\gamma_1}=J_{\gamma_2}$. 
From (\ref{Ixuxytilde}) and (\ref{Ixytildexy}), we also have (\ref{igamma12}).

\subsection{Proof of inequality (\ref{ineqchain4_2})}
\label{subsecproof_d2}
Notice that for every $\gamma\in\Gamma_1$, conditional mutual information can be written in terms of $P_{t|t}=\text{Cov}(\bx_t-\mathbb{E}(\bx_t|\by^t,\bu^{t-1}))$:
\begin{align}
&I_\gamma(\bx_t;\by_t|\by^{t-1})\nonumber \\
=&I_\gamma(\bx_t;\by_t|\by^{t-1},\bu^{t-1}) \nonumber \\
=&h(\bx_t|\by^{t-1},\bu^{t-1})-h(\bx_t|\by^t,\bu^{t-1}) \nonumber \\
=&\tfrac{1}{2}\log\det (A_{t-\!1}P_{t-1|t-1}A_{t-\!1}^\top\!\!+\!W_{t-\!1})\!-\!\tfrac{1}{2}\log\det P_{t|t}.
\label{eqmiinp}
\end{align}
Moreover, for every fixed sensor equation \eqref{eqgamma1awgn}, covariance matrices are determined by the Kalman filtering formula
\[
P_{t|t}=( (A_{t-1}P_{t-1|t-1}A_{t-1}^\top\!+\!W_{t-1})^{-1}+\mathsf{SNR}_t)^{-1}.
\]
Hence, conditional mutual information (\ref{eqmiinp}) depends only on the choice of $\{\mathsf{SNR}_t\}_{t=1}^T$, and is independent of the choice of a linear map $l_t$.
On the other hand, the LQG control cost $J_\gamma$ depends on the choice of $l_t$. 
In particular, for every fixed linear sensor (\ref{eqgamma1awgn}), it follows from the standard filter-controller separation principle in the LQG control theory that the optimal $l_t$ that minimizes $J_\gamma$ is a composition of a Kalman filter $\hat{\bx}_t=\mathbb{E}(\bx_t|\by^t,\bu^{t-1})$ and a certainty equivalence controller $\bu_t=K_t\hat{\bx}_t$.
This implies that an optimal solution $\gamma$ can always be found in the class $\Gamma_0$, establishing  the inequality in (\ref{ineqchain4_2}).

For a fixed linear sensor (\ref{eqgamma1awgn}), an explicit form of the Kalman filter and the certainty equivalence controller is given by Steps 1 and 3 in Section~\ref{secmainresult}.
Derivation is standard and hence is omitted. It is also possible to write $J_\gamma$ explicitly as
\begin{equation}
\label{eqjgamma}
J_\gamma\!=\!\text{Tr}(N_1P_{1|0})+\!\sum\nolimits_{t=1}^T\!\left(\text{Tr}(W_t S_t)\!+\!\text{Tr}(\Theta_tP_{t|t})\right).
\end{equation}
Derivation of (\ref{eqjgamma}) is also straightforward, and can be found in \cite[Lemma 1]{1503.01848}.

\subsection{Proof of inequality (\ref{ineqchain5})}
\label{subsecproof_e}
For every fixed $\gamma\in\Gamma_0$, by Lemma~\ref{lemmadirectedinfo} we have
\begin{align*}
\; I_\gamma(\bx^T\!\!\rightarrow \!\bu^T) 
\leq &\; I_\gamma(\bx^T\rightarrow \by^T\| \bu_+^{T-1}) \\
=& \sum\nolimits_{t=1}^T \!I_\gamma(\bx^t;\by_t|\by^{t-1},\bu^{t-1}) \\
=& \sum\nolimits_{t=1}^T \!I_\gamma(\bx^t;\by_t|\by^{t-1}) \\
=& \sum\nolimits_{t=1}^T \!I_\gamma(\bx_t;\by_t|\by^{t-1})\!+\!I_\gamma(\bx^{t-1};\by_t|\bx_t,\by^{t-1})\\
=&\sum\nolimits_{t=1}^T \!I_\gamma(\bx_t;\by_t|\by^{t-1}).
\end{align*}
The last equality holds since, by construction, $\by_t=C_t\bx_t+\bv_t$ is conditionally independent of $\bx^{t-1}$ given $\bx_t$. 

\subsection{SDP formulation of problem (\ref{ineqchain4_2})}
\label{subsec12dsdp}

Invoking \eqref{eqmiinp} and 
\eqref{eqjgamma} hold for every $\gamma\in\Gamma_0$, problem  (\ref{ineqchain4_2}) can be written as an optimization problem in terms of $\{P_{t|t},\mathsf{SNR}_t\}_{t=1}^T$ as
%\begin{subequations}
\begin{align*}
\min & \sum_{t=2}^T \left( \frac{1}{2} \log\det (A_{t-1}P_{t-1|t-1}A_{t-1}^\top\!+\!W_t)\!-\! \frac{1}{2} \log\det P_{t|t}  \right)\\
&+\frac{1}{2}\log\det P_{1|0}-\frac{1}{2}\log\det P_{1|1} \\
\text{s.t. } & \text{Tr}(N_1P_{1|0})\!+\!\sum\nolimits_{t=1}^T\!\left(\text{Tr}(W_t S_t)\!+\!\text{Tr}(\Theta_tP_{t|t})\right) \leq D, \\
&  P_{1|1}^{-1}\!=\!P_{1|0}^{-1}+\mathsf{SNR}_1, \\ 
&  P_{t|t}^{-1}\!=\!(A_{t-1}P_{\!t-1|t-1}A_{t-1}^\top\!\!+\!W_{t-1})^{-1}\!+\!\mathsf{SNR}_t, \; t=2, ... , T \\
&  \mathsf{SNR}_t \succeq 0, \;\; t=1, ... , T. 
\end{align*}
%\end{subequations} 
This problem can be reformulated as a max-det problem as follows.
First, variables $\{\mathsf{SNR}_t\}_{t=1}^T$ are eliminated from the problem by replacing the last three constraints with equivalent conditions
\begin{align*}
&0\prec P_{1|1} \preceq P_{1|0}, \\
&0\prec P_{t|t} \preceq A_{t-1}P_{t-1|t-1}A_{t-1}^\top+W_{t-1}, \; t=2, ... ,T.
\end{align*}
Second, the following equalities can be used for $t=1, ... , T-1$ to rewrite the objective function:
\begin{subequations}
\begin{align}
&\frac{1}{2}\log\det (A_tP_{t|t}A_t^\top+W_t)-\frac{1}{2}\log\det P_{t|t} \nonumber \\
=\;&\frac{1}{2}\log\det(P_{t|t}^{-1}+A_t^\top W_t^{-1}A_t)+\frac{1}{2}\log\det W_t  \label{derivemaxdet1} \\
=\;& \inf_{\Pi_t}\quad \frac{1}{2}\log\det \Pi_t^{-1}+\frac{1}{2}\log\det W_t \label{derivemaxdet2} \\
& \text{ s.t. }\quad  0\prec \Pi_t\preceq (P_{t|t}^{-1}+A_t^\top W_t^{-1}A_t)^{-1} \nonumber \\
=\;& \inf_{\Pi_t}\quad \frac{1}{2}\log\det \Pi_t^{-1}+\frac{1}{2}\log\det W_t \label{derivemaxdet4} \\
& \text{ s.t. }\quad  \Pi_t \succ 0, \; \left[\begin{array}{cc}P_{t|t}-\Pi_t & P_{t|t}A_t^\top \\ A_tP_{t|t} & A_tP_{t|t}A_t^\top+W_t \end{array}\right]\succeq 0. \nonumber
\end{align}
\end{subequations}
In step (\ref{derivemaxdet1}), we have used the matrix determinant theorem \cite[Theorem 18.1.1]{harville1997matrix}. An additional variable $\Pi_t$ is introduced in step (\ref{derivemaxdet2}). The constraint is rewritten using the matrix inversion lemma in (\ref{derivemaxdet4}).

These two techniques allow us to formulate the above problem as a max-det problem (\ref{optprob3}). 
Thus, we have shown that Steps 1-3 in Section~\ref{secmainresult} provide an optimal solution to problem (\ref{ineqchain4}), which is also an optimal solution to the original problem (\ref{ineqchain1}).

\section{Extension to partially observable plants}
\label{secpo}

 \ifdefined\DOUBLECOLUMN	
\begin{figure}[t]
    \centering
    \includegraphics[width=.7\columnwidth]{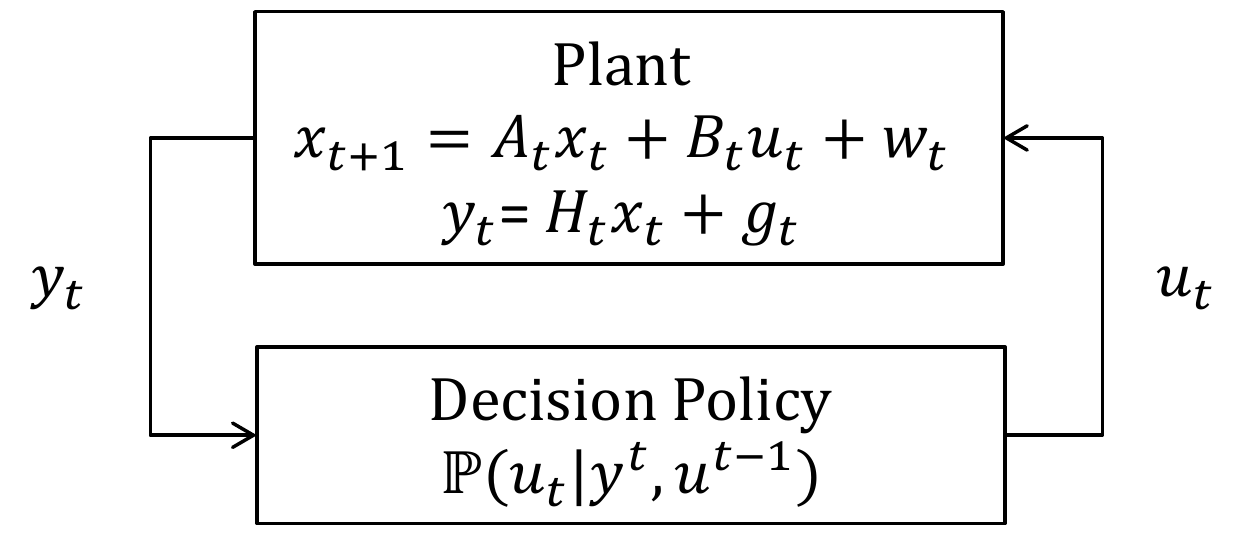}
    \caption{LQG control of partially observable plant with minimum directed information.}
    \label{fig:poprob}
\vspace{-3ex}
\end{figure}
\begin{figure*}[t]
    \centering
    \includegraphics[width=0.7\textwidth]{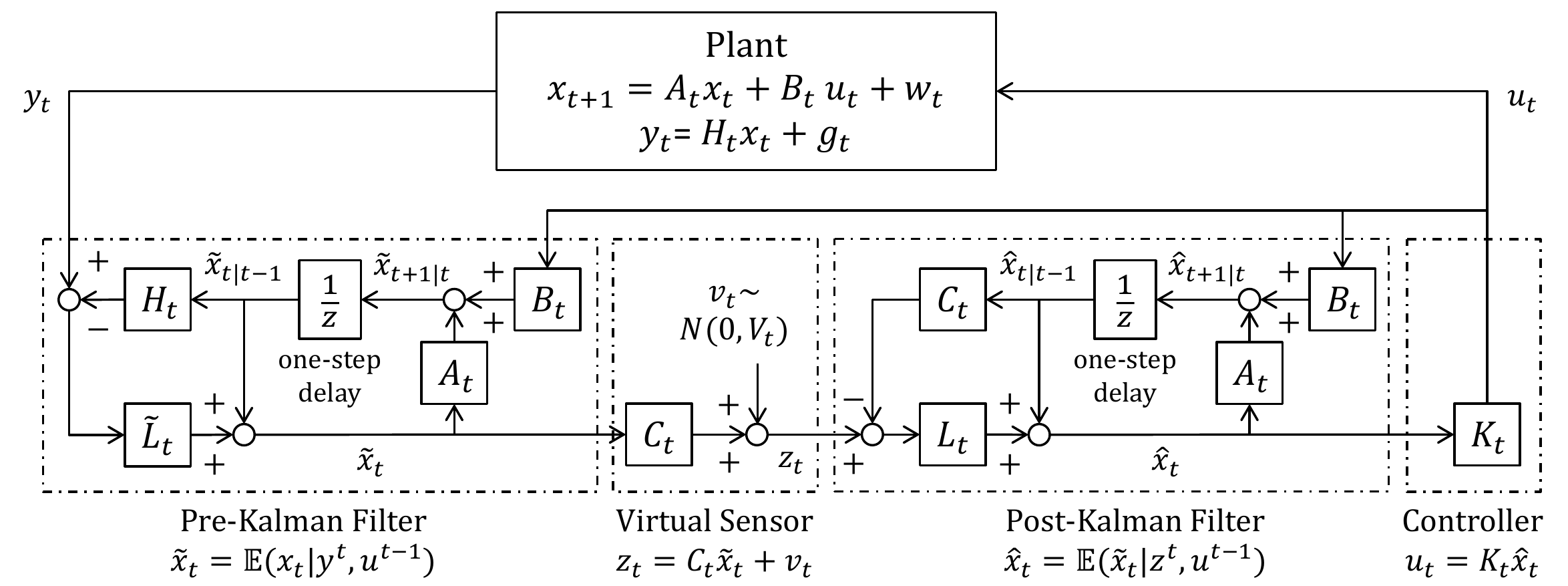}
    \caption{Structure of optimal control policy for problem \eqref{partiallyobservableprob}. Matrices $\tilde{L}_t$, $C_t$, $V_t$, $L_t$ and $K_t$ are determined by SDP-based algorithm in Section~\ref{secpo}.}
    \label{fig:sep4}
    \vspace{-3ex}
\end{figure*}
\fi
 \ifdefined\SINGLECOLUMN	
\begin{figure}[t]
    \centering
    \includegraphics[width=.5\columnwidth]{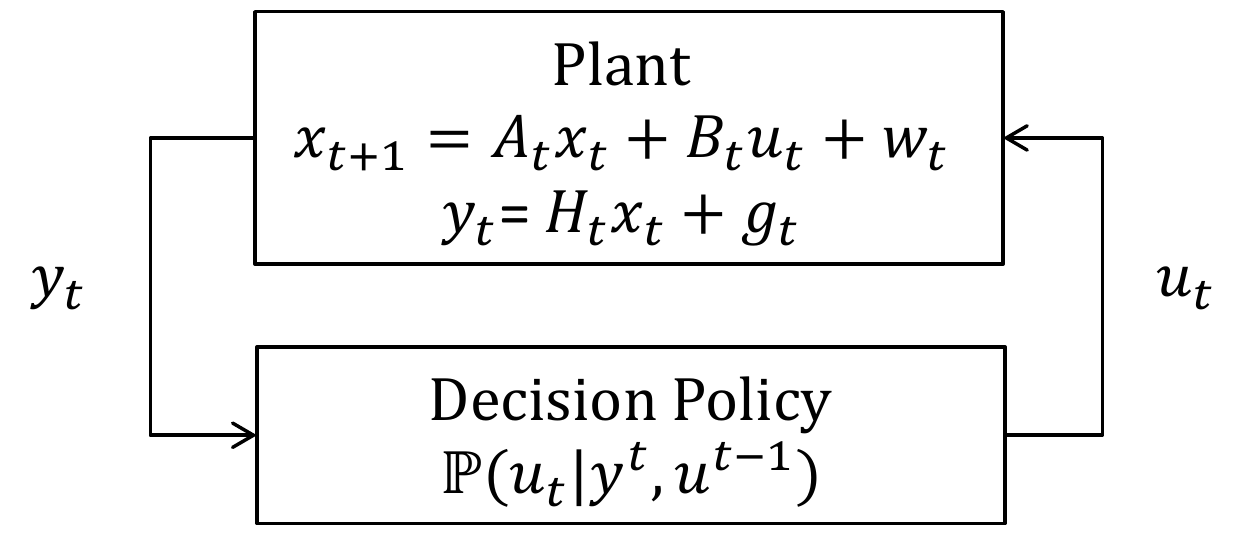}
    \caption{LQG control of partially observable plant with minimum directed information.}
    \label{fig:poprob}
\vspace{-3ex}
\end{figure}
\begin{figure*}[t]
    \centering
    \includegraphics[width=0.8\textwidth]{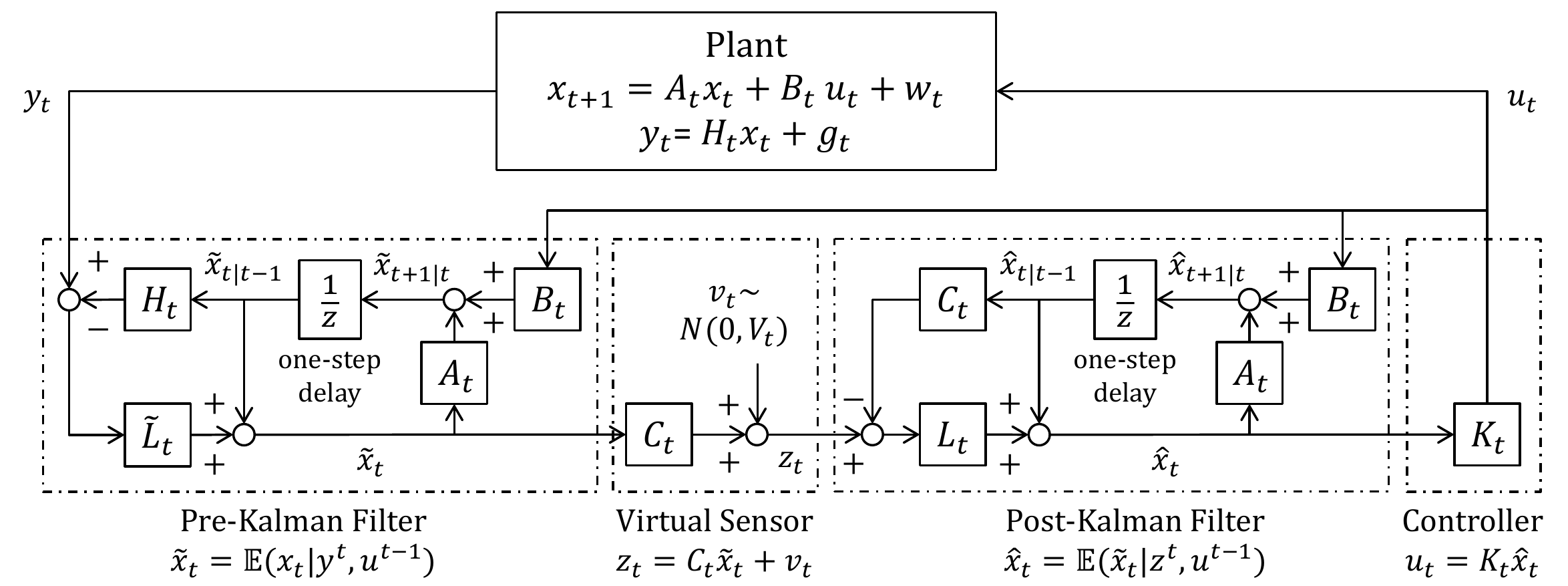}
    \caption{Structure of optimal control policy for problem \eqref{partiallyobservableprob}. Matrices $\tilde{L}_t$, $C_t$, $V_t$, $L_t$ and $K_t$ are determined by SDP-based algorithm in Appendix~\ref{apppartially}.}
    \label{fig:sep4}
\end{figure*}
\fi

So far, our focus has been on a control system in Figure~\ref{fig:mainprob} in which the decision policy has an access to the state $\bx_t$ of the plant. Often in practice, the state of the plant is only partially observable through a given physical sensor mechanism. We now consider an extension of the control synthesis to partially observable plants.

Consider a control system in Figure~\ref{fig:poprob} where a state space model \eqref{eqsystem} and a sensor model $\by_t=H_t\bx_t+\bg_t$ are given. We assume that initial state $\bx_1\sim \mathcal{N}(0,P_{1|0})$, $P_{1|0}\succ 0$ and noise processes $\bw_t\sim\mathcal{N}(0,W_t)$, $W_t\succ 0$,  $\bg_t\sim\mathcal{N}(0,G_t)$, $G_t\succeq 0$, $t=1, ... , T$ are mutually independent. We also assume that $H_t$ has full row rank for $t=1, ... , T$. Consider the following problem:
\begin{subequations}
\label{partiallyobservableprob}
\begin{align}
\min_{\gamma\in\Gamma} & \quad  I_\gamma (\by^T\rightarrow \bu^T) \label{objdi3}\\
\text{ s.t. } &\quad  J_\gamma (\bx^{T+1},\bu^T) \leq D.
\end{align}
\end{subequations}
where $\Gamma$ is the space of policies $\gamma=\PP(u^T\|y^T)$. Relevant optimization problems to \eqref{partiallyobservableprob} are considered in \cite{silva2011framework,silva2011achievable,silva2013characterization}  in the context of Section~\ref{secmotivation}.
Based on the control synthesis developed so far for fully observable plants, it can be shown that the optimal control policy can be realized by the architecture shown in Figure~\ref{fig:sep4}. Moreover, as in the fully observable cases, the optimal control policy can be synthesized by an SDP-based algorithm. 

{\bf Step~1.} (Pre-Kalman filter design)  Design a Kalman filter
\begin{subequations}
\label{eqprekf}
\begin{align}
&\tilde{\bx}_t=\tilde{\bx}_{t|t-1}+\tilde{L}_t(\by_t-H_t\tilde{\bx}_{t|t-1}) \\
&\tilde{\bx}_{t+1|t}=A_t\tilde{\bx}_t+B_t\bu_t, \;\tilde{\bx}_{1|0}=0
\end{align}
\end{subequations}
where  the Kalman gains $\{\tilde{L}_t\}_{t=1}^{T+1}$ are computed by
\begin{align*}
&\tilde{L}_t=\tilde{P}_{t|t-1}H_t^\top(H_t\tilde{P}_{t|t-1}H_t^\top+G_t)^{-1} \\
&\tilde{P}_{t|t}=(I-\tilde{L}_tH_t)\tilde{P}_{t|t-1} \\
&\tilde{P}_{t+1|t}=A_t\tilde{P}_{t|t}A_t^\top+W_t.
\end{align*}
Matrices $\Psi_t=\tilde{L}_{t+1}(H_{t+1}\tilde{P}_{t+1|t}H_{t+1}^\top+G_{t+1})\tilde{L}_{t+1}^\top$ will be used in Step 3.

{\bf Step~2.} (Controller design) Determine  feedback control gains $K_t$ via the backward Riccati recursion:
\begin{subequations}
\begin{align}
S_t&=\begin{cases} Q_t & \text{ if } t=T \\ Q_t+N_{t+1} & \text{ if } t=1,\cdots, T-1 \end{cases} \\
M_t&=B_t^\top S_t B_t + R_t \\
N_t&=A_t^\top (S_t-S_t B_t M_t^{-1} B_t^\top S_t) A_t \\
K_t&= -M_t^{-1} B_t^\top S_t A_t \\
\Theta_t&= K_t^\top M_t K_t 
\end{align}
\end{subequations}
Positive semidefinite matrices $\Theta_t$ will be used in Step~3.

{\bf Step~3.} (Virtual sensor design) Solve a max-det problem with respect to $\{P_{t|t}, \Pi_t\}_{t=1}^T$:
\begin{subequations}
\label{probmaxdetpartially}
\begin{align}
\min & \quad \frac{1}{2} \sum\nolimits_{t=1}^T \log\det \Pi_t^{-1} + c_1 \\
\text{s.t.} & \quad \sum\nolimits_{t=1}^T \text{Tr}(\Theta_t P_{t|t}) + c_2 \leq D \\
& \quad  \Pi_t \succ  0,  \label{optprob3Pi_po}\\
& \quad  P_{1|1}\preceq P_{1|0}, P_{T|T}=\Pi_T,  \\
& \quad  P_{t+1|t+1}\preceq A_t P_{t|t}A_t^\top +\Psi_t, \label{optprob3e_po} \\
&\hspace{1ex} \left[\!\! \begin{array}{cc}P_{t|t}\!-\!\Pi_t\!\!\! &\!\! P_{t|t}A_t^\top \\
A_tP_{t|t} \!\!\!&\!\! A_t P_{t|t}A_t^\top +\Psi_t \end{array}\!\!\right]\! \succeq\! 0. \label{optprob3f_po}
\end{align}
\end{subequations} 
The constraint (\ref{optprob3Pi_po}) is imposed for every $t=1, \cdots, T$, while (\ref{optprob3e_po}) and (\ref{optprob3f_po}) are for every $t=1, \cdots, T-1$.
Constants $c_1$ and $c_2$ are  given by
\begin{align*}
c_1&=\frac{1}{2}\log\det{P_{1|0}}+\frac{1}{2}\sum\nolimits_{t=1}^{T-1} \log\det \Psi_t \\
c_2&=\text{Tr}(N_1P_{1|0})+\sum\nolimits_{t=1}^T \text{Tr}(\Psi_t S_t).
\end{align*} 
If $\Psi_t$ is singular for some $t$, consider an alternative max-det problem suggested in Remark~\ref{remsingularw}.
Set $r_t=\text{rank}(P_{t|t}^{-1}-P_{t|t-1}^{-1})$, where 
\[
P_{t|t-1}\triangleq A_{t-1}P_{t-1|t-1}A_{t-1}^\top + W_{t-1}, t=2,\cdots, T.
\]
 Choose matrices $C_t\in \mathbb{R}^{r_t\times n_t}$ and $V_t \in \mathbb{S}_{++}^{r_t}$ so that
\begin{equation}
\label{cvconst}
 C_t^\top V_t^{-1}C_t= P_{t|t}^{-1}-P_{t|t-1}^{-1}
\end{equation}
for $t=1,\cdots, T$.  In case of $r_t=0$, $C_t$ and $V_t$ are considered to be null (zero dimensional) matrices.

{\bf Step~4.} (Post-Kalman filter design) Design a Kalman filter
\begin{subequations}
\label{eqkalmanfilter}
\begin{align}
&\hat{\bx}_t=\hat{\bx}_{t|t-1}+\hat{L}_t(\bz_t-C_t\hat{\bx}_{t|t-1}) \label{eqkalmanfilter1_po}\\
&\hat{\bx}_{t+1|t}=A_t\hat{\bx}_t+B_t \bu_t. \label{eqkalmanfilter2_po}
\end{align}
\end{subequations}
where Kalman gains $\hat{L}_t$ are computed by
\begin{equation}
\label{eqkalmangain}
\hat{L}_t=P_{t|t-1}C_t^\top (C_t P_{t|t-1}C_t^\top + V_t)^{-1}.
\end{equation}
If $r_t=0$, $L_t$ is a null matrix and (\ref{eqkalmanfilter1_po}) is simply replaced by $\hat{\bx}_t=\hat{\bx}_{t|t-1}$.

\begin{theorem}
\label{theorempartially}
An optimal policy for the problem \eqref{partiallyobservableprob} exists if and only if the max-det problem \eqref{probmaxdetpartially} is feasible, and the optimal value of \eqref{partiallyobservableprob} coincides with the optimal value of \eqref{probmaxdetpartially}.
If the optimal value of \eqref{partiallyobservableprob} is finite, an optimal policy can be realized by an interconnection of a pre-Kalman filter, a virtual sensor, post-Kalman filter, and a certainty equivalence controller as shown in Figure~\ref{fig:sep4}.
Moreover, each of these components can be constructed by an SDP-based algorithm summarized in Steps 1-4 above.
\end{theorem}
\begin{proof}
\ifdefined\LONGVERSION
See Appendix~\ref{apppartially}.
\fi
\ifdefined\SHORTVERSION
See \cite{tanaka2015lqg}.
\fi
\end{proof}

\section{Conclusion}
\label{secconclusion}
In this paper, we considered an optimal control problem in which directed information from the observed output of the plant to the control input is minimized subject to the constraint that the control policy achieves the desired LQG control performance. When the state of the plant is directly observable, the optimal control policy can be realized by a three-stage structure comprised of (1) linear sensor with additive Gaussian noise, (2) Kalman filter, and (3) certainty equivalence controller. An extension to partially observable plants was also discussed. In both cases, the optimal policy is synthesized by an efficient numerical algorithm based on SDP.

\appendix
\subsection{Data-processing inequality for directed information}
\label{adddataprocessing}
Lemma~\ref{lemmadirectedinfo} is shown as follows.
Notice that the following chain of equalities hold for every $t=1,\cdots,T$.
\begin{subequations}
\label{eqlb1-4}
\begin{align}
&I(\bx^t;\ba_t|\ba^{t-1},\bu^{t-1})-I(\bx^t;\bu_t|\bu^{t-1}) \nonumber \\
=&I(\bx^t;\ba_t, \bu_t|\ba^{t-1},\bu^{t-1})-I(\bx^t;\bu_t|\bu^{t-1}) \label{eqlb1} \\
=&I(\bx^t;\ba^t|\bu^t)-I(\bx^t;\ba^{t-1}|\bu^{t-1}) \label{eqlb2} \\
=&I(\bx^t;\ba^t|\bu^t)-I(\bx^{t-1};\ba^{t-1}|\bu^{t-1})\nonumber \\
&\hspace{12ex}-I(\bx_t;\ba^{t-1}|\bx^{t-1},\bu^{t-1}) \label{eqlb3} \\
=&I(\bx^t;\ba^t|\bu^t)-I(\bx^{t-1};\ba^{t-1}|\bu^{t-1}). \label{eqlb4}
\end{align}
\end{subequations}
When $t=1$, the above identity is understood to mean
$I(\bx_1;\ba_1)-I(\bx_1;\bu_1)=I(\bx_1;\ba_1|\bu_1)$
which clearly holds as $\bx_1$--$\ba_1$--$\bu_1$ form a Markov chain.
Equation (\ref{eqlb1}) holds because $I(\bx^t;\ba_t,\bu_t|\ba^{t-1},\bu^{t-1})=I(\bx^t;\ba_t|\ba^{t-1},\bu^{t-1})+I(\bx^t;\bu_t|\ba^t,\bu^{t-1})$
and the second term is zero since $\bx^t$--$(\ba^t,\bu^{t-1})$--$\bu_t$ form a Markov chain. Equation (\ref{eqlb2}) is obtained by applying the chain rule for mutual information in two different ways:
\begin{align*}
&I(\bx^t;\ba^t,\bu_t|\bu^{t-1}) \\
&=I(\bx^t;\ba^{t-1}|\bu^{t-1})+I(\bx^t;\ba_t,\bu_t|\ba^{t-1},\bu^{t-1}) \\
&=I(\bx^t;\bu_t|\bu^{t-1})+I(\bx^t;\ba^t|\bu^t).
\end{align*}
The chain rule is applied again in step (\ref{eqlb3}). Finally, (\ref{eqlb4}) follows as
$\ba^{t-1}$--$(\bx^{t-1},\bu^{t-1})$--$\bx_t$ form a Markov chain.

Now, the desired inequality can be verified by computing the right hand side minus the left hand side as
\begin{subequations}
\label{telescope}
\begin{align}
&\sum\nolimits_{t=1}^T \left[ I(\bx^t;\ba_t|\ba^{t-1},\bu^{t-1})-I(\bx^t;\bu_t|\bu^{t-1})\right] \nonumber \\
=&\sum\nolimits_{t=1}^T \left[ I(\bx^t;\ba^t|\bu^t)-I(\bx^{t-1};\ba^{t-1}|\bu^{t-1})\right] \label{telescope1} \\
=&I(\bx^T;\ba^T|\bu^T) \geq 0. \label{telescope2}
\end{align}
\end{subequations}
In step (\ref{telescope1}), the identity (\ref{eqlb1-4}) is used. The telescoping sum (\ref{telescope1}) cancels all but the final term (\ref{telescope2}).

\ifdefined\LONGVERSION
\subsection{Some basic lemmas for probability measures}
\begin{lemma}\label{lemfubini}
Let $\PP_{\bx,\by}$ be a joint probability measure on $(\mathcal{X}\times\mathcal{Y},\mathcal{B}_\mathcal{X}\otimes\mathcal{B}_\mathcal{Y})$. Let $\PP_\bx$ and $\PP_\by$ be the marginal probability measures, $\PP_\bx\otimes \PP_\by $ be the product measure, and $\PP_{\bx|y}$ be a Borel measurable stochastic kernel such that
\begin{equation}
\label{existstockernel}
\PP_{\bx,\by}(B_X\times B_Y)
=\int_{B_Y}
\PP_{\bx|y}(B_X|y)\PP_\by(\diff y)
\end{equation}
for every $B_X\in \mathcal{B}_\mathcal{X}$ and $B_Y\in \mathcal{B}_\mathcal{Y}$.
If $\PP_{\bx,\by}\ll\PP_\bx\otimes \PP_\by$, then $\PP_{\bx|y}\ll\PP_\bx, \PP_\by -a.e.$, and
\[
\frac{\diff \PP_{\bx,\by}}{\diff (\PP_\bx\otimes \PP_\by)}=\frac{\diff \PP_{\bx|y}}{\diff \PP_\bx}, \PP_\by-a.e..
\]
\vspace{-2ex}
\end{lemma}
\begin{proof}
Suppose $\PP_{\bx,\by}\ll\PP_\bx\otimes \PP_\by$, and let $f(x,y)=\frac{\diff\PP_{\bx,\by}}{\diff (\PP_\bx\otimes\PP_\by)}$ be the Radon-Nikodym derivative. For every $B_X\in \mathcal{B}_\mathcal{X}$ and $B_Y\in \mathcal{B}_\mathcal{Y}$, we have
\begin{align}
\PP(B_X\times B_Y)=&\int_{B_X\times B_Y} f(x,y) \diff(\PP_\bx\otimes\PP_\by)(x,y) \nonumber \\
=&\int_{B_Y}\left[\int_{B_X}f(x,y)\diff \PP_\bx(x) \right]\diff \PP_\by(y) \label{eqfubini}
\end{align}
The first line is by definition of the Radon-Nikodym derivative. The second equality holds due to the Fubini's theorem \cite{folland1999real}, since clearly $f\in L^1(\PP_\bx\otimes \PP_\by)$. Comparing (\ref{existstockernel}) and (\ref{eqfubini}), we have 
\begin{equation}
\label{resultlemfubini}
\PP_{\bx|y}(B_X|y)=\int_{B_X} f(x,y)\diff \PP_\bx(x), \PP_\by -a.e..
\end{equation}
It follows from (\ref{resultlemfubini}) that $\PP_\bx(B_X)=0\Rightarrow \PP_{\bx|y}(B_X|y)=0$ holds 
$\PP_\by-a.e.$. Also,  (\ref{resultlemfubini}) implies $f(x,y)=\frac{\diff \PP_{\bx|y}}{\diff \PP_\bx}, \PP_\by-a.e.$.
\end{proof}

\begin{lemma}\label{lemkim}
Let $\PP_{\bx,\by,\bz}$ be a zero-mean Borel probability measure on $\mathcal{X}\times\mathcal{Y}\times\mathcal{Z}$, where $\mathcal{X}$, $\mathcal{Y}$, and $\mathcal{Z}$ are Euclidean spaces. 
Suppose $\PP_{\bx,\by,\bz}$ has a covariance matrix $\Sigma_{x,y,z}$, and there exists a matrix $L$ such that $\bz-L\by$ is independent of $\bx$ and $\by$ on $\PP_{\bx,\by,\bz}$. (This implies $\bx$--$\by$--$\bz$ form a Markov chain on $\PP_{\bx,\by,\bz}$.) Let $\GG_{\bx,\by,\bz}$ be a zero-mean, jointly Gaussian probability measure with the same covariance matrix $\Sigma_{x,y,z}$. Then $\bx$--$\by$--$\bz$ form a Markov chain on $\GG_{\bx,\by,\bz}$.
\end{lemma}
\begin{proof}
See \cite[Lemma 3.2]{kim2010feedback}.
\end{proof}

\begin{lemma}\label{lemlineargaussian}
Let $\bx$ be an $(\mathbb{R}^n,\mathcal{B}_{\mathbb{R}^n})$-valued zero mean random variable with covariance $\Sigma_x \succeq 0$. Define an $(\mathbb{R}^m,\mathcal{B}_{\mathbb{R}^m})$-valued random variable $\by$ by $\by=A\bx+\bv$ where $A$ is a matrix and $\bv\sim\mathcal{N}(0, \Sigma_v)$ is a random variable independent of $\bx$. Let $(\bx_G, \by_G)$ be zero-mean, jointly Gaussian random variables, and suppose that $(\bx, \by)$ and $(\bx_G, \by_G)$ have the same covariance matrix.
Then $\by_G$ can be written as $\by_G=A\bx_G+\bv$ with $\bv\sim\mathcal{N}(0, \Sigma_v)$ independent of $\bx_G$.
\end{lemma}
\begin{proof}
Observe
\begin{equation}
\label{covxy}
\text{cov}\left[\begin{array}{c}\bx \\ \by\end{array}\right]=
\left[\begin{array}{cc}\Sigma_x & \Sigma_x A^\top \\ A\Sigma_x & A\Sigma_x A^\top+\Sigma_v \end{array}\right].
\end{equation}
Since it must be that $\bx_G\sim\mathcal{N}(0,\Sigma_x)$, introducing a matrix $R$ with full column rank such that $\Sigma_x=RR^\top$, $\bx_G$ can be written as $\bx_G=R\bz_G$ with $\bz_G\sim\mathcal{N}(0,I)$. Since $(\bx_G, \by_G)$ are jointly Gaussian, there exists a matrix $\bar{A}$ such that 
\begin{equation}
\label{ygabar}
\by_G=\bar{A}\bx_G+\bv, \;\; \bv\sim\mathcal{N}(0,\bar{\Sigma}_v)
\end{equation}
where $\bar{\bv}$ is independent of $\bx_G$. Thus
\begin{equation}
\label{covxgyg}
\text{cov}\left[\begin{array}{c}\bx_G \\ \by_G\end{array}\right]=
\left[\begin{array}{cc}\Sigma_x & \Sigma_x \bar{A}^\top \\ \bar{A}\Sigma_x & \bar{A}\Sigma_x \bar{A}^\top+\bar{\Sigma}_v \end{array}\right].
\end{equation}
By comparing (\ref{covxy}) and (\ref{covxgyg}), it can be seen that $\bar{A}=A+S$ with $S$ satisfying $SR=0$, and $\bar{\Sigma}_v=\Sigma_v$. Then from (\ref{ygabar}),
\begin{align*}
\by_G&=(A+S)\bx_G+\bv \\
&=(A+S)R\bz_G+\bv \\
&=AR\bz_G+\bv \\
&=A\bx_G+\bv, \;\; \bv\sim\mathcal{N}(0,\Sigma_v).
\end{align*}
\end{proof}

\begin{lemma}
\label{lemdensae}
Let $\PP_{\bx,\by}$ be a zero-mean joint probability measure on $(\mathcal{X}\times\mathcal{Y},\mathcal{B}_\mathcal{X}\otimes\mathcal{B}_\mathcal{Y})$, $\mathcal{X}=\mathbb{R}^n$, $\mathcal{Y}=\mathbb{R}^m$, with a covariance matrix $\Sigma_{x,y}$. Let $\GG_{\bx,\by}$ be a zero-mean Gaussian joint probability measure with the same covariance matrix $\Sigma_{x,y}$.
If there exists a subset $\mathcal{C}_y \subseteq \mathbb{R}^m$ with $\PP_\by(\mathcal{C}_y)>0$ such that $\PP_{\bx|y}$ admits density for every $y\in \mathcal{C}_y$, then $\GG_{\bx|y}$ admits density for every $y\in\text{supp}(\GG_\by)$.
\end{lemma}
\begin{proof}
Suppose $\PP_{\bx,\by}(x, y)$ and $\GG_{\bx,\by}(x, y)$ share a covariance matrix
\[
\Sigma_{x,y}= \left[\begin{array}{cc}\Sigma_{xx} & \Sigma_{xy} \\ \Sigma_{yx} & \Sigma_{yy}  \end{array}\right] \succeq 0.
\]
Since $\GG_{\bx,\by}$ is a zero-mean Gaussian distribution, we know
\begin{equation}
\label{eqgxgiveny}
\GG_{\bx|y}\sim \mathcal{N}(\hat{x}(y), \Sigma_e) \;\; \forall y \in \mathcal{Y}
\end{equation}
with $\hat{x}(y)=\Sigma_{xy}\Sigma_{yy}^\dagger y$, $\Sigma_e=\Sigma_{xx}
-\Sigma_{xy}\Sigma_{yy}^\dagger\Sigma_{yx}$, where $\Sigma_{yy}^\dagger$ is the Moore--Penrose pseudoinverse of $\Sigma_{yy}$.
To show contrapositive, assume that there exists $y\in \text{supp}(\GG_\by)$ such that $\GG_{\bx|y}(x|y)$ does not admit a density. From (\ref{eqgxgiveny}), this means that $\Sigma_e$ is a singular matrix.
For every $y\in\mathcal{Y}$, define a covariance matrix $M(y)\triangleq \int_\mathcal{X} (x-\hat{x}(y))(x-\hat{x}(y))^\top \diff \PP_{\bx|y}(x|y)$.
Observe that
\begin{align}
 \int_\mathcal{Y} M(y) \diff\PP(y) 
=&\int_{\mathcal{X},\mathcal{Y}} (x-\hat{x}(y))(x-\hat{x}(y))^\top \diff\PP(x, y) \nonumber \\
=&\int_{\mathcal{X},\mathcal{Y}} (x-\hat{x}(y))(x-\hat{x}(y))^\top \diff\GG(x, y)  \nonumber \\
=&\Sigma_e. \label{eqmysigmae}
\end{align}
Since $\Sigma_e$ is singular, there exists a full row rank matrix $U\in\mathbb{R}^{r\times n}$, $1\leq r \leq n$ such that $U\Sigma_e U^\top=0$. From (\ref{eqmysigmae}), it follows that $UM(y) U^\top=0, \PP_\by-a.e..$
For every $y\in\mathcal{Y}$, define a subset 
 $\mathcal{C}_{x|y}\subset \mathbb{R}^n$ by
\[ \mathcal{C}_{x|y}=\{x\in\mathbb{R}^n : U(x-\hat{x}(y))=0\}. \]
By construction, $\PP_{\bx|y}(\mathcal{C}_{x|y}|y)=1, \;\; \PP_\by-a.e.$. However, clearly $\LL_\bx(\mathcal{C}_{x|y})=0$, where $\LL_\bx(x)$ is the Lebesgue measure on $\mathbb{R}^n$.
Thus $\PP_{\bx|y} \ll \LL_\bx$ fails to hold $\PP_\by-a.e.$. Hence, $\PP_{\bx|y}$ fails to admit a density $\PP_\by-a.e.$. This is a contradiction to the assumption that there exists a subset $\mathcal{C}_y \subseteq \mathbb{R}^m$ with $\PP_\by(\mathcal{C}_y)>0$ such that $\PP_{\bx|y}$ admits density for every $y\in \mathcal{C}_y$.
\end{proof}
\fi

\subsection{Proof of Lemma~\ref{lemmaIpg}}
\label{appproofPG}
\ifdefined\SHORTVERSION
We use the following technical lemmas~\ref{lemsupport},\ref{lemabcd}, and \ref{lemmagpg}. 
Proofs can be found in \cite{tanaka2015lqg}.
\fi
\begin{lemma}
\label{lemsupport}
Let $\PP$ be a zero-mean Borel probability measure on $\mathbb{R}^n$ with covariance matrix $\Sigma$. Suppose $\GG$ is a zero-mean Gaussian probability measure on $\mathbb{R}^n$ with the same covariance matrix $\Sigma$. Then $\text{supp}(\PP)\subseteq \text{supp}(\GG)$.
\end{lemma}
\ifdefined\LONGVERSION
\begin{proof}
When $\Sigma$ is positive definite, the claim is trivial since $\text{supp}(\GG)=\mathbb{R}^n$. So assume that $\Sigma$ is singular. Then, there exists an orthonormal matrix 
\[
U=\left[\begin{array}{c}U_1 \\ U_2\end{array}\right] \text{ with } U_1\in\mathbb{R}^{p\times n}, \; U_2\in\mathbb{R}^{(n-p)\times n}, \; p<n
\]
 such that $U\Sigma U^\top=\text{diag}(\Sigma_{zz},0)$, where $\Sigma_{zz}\in\mathbb{R}^{p\times p}$ is a positive definite matrix. 
Notice that $U_2 x=0$ for every $x\in\text{supp}(\GG)$, and $U_2 x\neq 0$ for every $x\in\text{supp}(\GG)^c$. 
Suppose $\text{supp}(\PP)\subseteq \text{supp}(\GG)$ does not hold, i.e., there exists a closed set $\mathcal{C}\subseteq\text{supp}(\GG)^c$ such that $\PP(\mathcal{C})>0$. Then
\begin{align*}
&U_2\Sigma U_2^\top=\int U_2 xx^\top U_2^\top d\PP(x) \\
&=\int_{\text{supp}(\GG)}U_2 xx^\top U_2^\top d\PP(x)+\int_{\text{supp}(\GG)^c}U_2 xx^\top U_2^\top d\PP(x) \\
&=\int_{\text{supp}(\GG)^c}U_2 xx^\top U_2^\top d\PP(x)  \succeq \int_{\mathcal{C}}U_2 xx^\top U_2^\top d\PP(x) 
\end{align*}
Since $U_2 x\neq 0$ for every $x\in\mathcal{C}$, the last expression is a non-zero positive semidefinite matrix. However, by construction, we have $U_2\Sigma U_2^\top=0$. Thus, the above inequality leads to a contradiction.
\end{proof}
\fi
\begin{lemma}
\label{lemabcd}
Let $\PP(x^{T+1}, u^T)$ be  a joint probability measure generated by a policy $\gamma_\PP=\{\PP(u_t|x^t,u^{t-1})\}_{t=1}^T$ and (\ref{eqsystem}).
\begin{itemize}%[leftmargin=*]
\item[(a)] For each $t=1,\cdots, T$, $\PP(x_{t+1}|u^t)$ and $\PP(x_{t+1}|x_t, u^t)$ are non-degenerate Gaussian probability measures for every $x_t$ and $u^t$.
\end{itemize}
Moreover, if $I_\PP (\bx_t;\bu_t|\bu^{t-1})<+\infty$ for all $t=1,\cdots,T$, then the following statements hold.
\begin{itemize}%[leftmargin=*]
\item[(b)] For every $t=1,\cdots, T$,
\begin{align*}
&\PP(x_t|u^t) \ll \PP(x_t|u^{t-1}), \;\; \PP(u^t)-a.e., \text{ and } \\
&I_\PP(\bx_t;\bu_t|\bu^{t-1})=\int \log\left(\frac{\diff \PP(x_t|u^t)}{\diff\PP(x_t|u^{t-1})}\right)\diff \PP(x_t, u^t).
\end{align*}
\item[(c)] For every $t=1,\cdots, T$,
\[\PP(x_t|x_{t+1},u^t) \ll \PP(x_t|u^{t-1}), \;\; \PP(x_{t+1},u^t)-a.e..\]
Moreover, the following identity holds $\PP(x_{t+1},u^t)-a.e.$:
\begin{equation}
\label{eqrndprod}
 \frac{\diff\PP(x_t|u^t)}{\diff\PP(x_t|u^{t-1})}=\frac{\diff\PP(x_{t+1}|u^t)}{\diff\PP(x_{t+1}|x_t,u^t)}\frac{\diff\PP(x_t|x_{t+1},u^t)}{\diff\PP(x_t|u^{t-1})}.
\end{equation}
\end{itemize}
\end{lemma}

\ifdefined\LONGVERSION
\begin{proof}

(a) This is clear since $\PP(x^{T+1}, u^T)$ is constructed using  (\ref{eqsystem}).

(b) By definition of conditional mutual information, $I_\PP (\bx_t;\bu_t|\bu^{t-1})<+\infty$ requires $I_\PP (\bx_t;\bu_t|u^{t-1})<+\infty, \PP_{\bu^{t-1}}-a.e.$.
For a fixed $u^{t-1}$, $I_\PP (\bx_t;\bu_t|u^{t-1})<+\infty$ requires $\PP_{\bx_t,\bu_t|u^{t-1}} \ll\PP_{\bx_t|u^{t-1}}\otimes\PP_{\bu_t|u^{t-1}}$ by definition of mutual information. By Lemma \ref{lemfubini}, this implies 
\vspace{-1ex}
\begin{align}
&\PP_{\bx_t|u^t}\ll\PP_{\bx_t|u^{t-1}} \text{ and } \label{lemabcd_b1}\\
&\frac{\diff\PP_{\bx_t,\bu_t|u^{t-1}}}{\diff(\PP_{\bx_t|u^{t-1}}\otimes\PP_{\bu_t|u^{t-1}})}=\frac{\diff \PP_{\bx_t|u^t}}{\diff \PP_{\bx_t|u^{t-1}}} \label{lemabcd_b2}
\end{align}
must hold $\PP(u_t|u^{t-1})-a.e.$. Since this is the case $\PP(u^{t-1})-a.e.$, we have both (\ref{lemabcd_b1}) and  (\ref{lemabcd_b2})  $\PP(u^t)-a.e.$.
Hence
\vspace{-1ex}
\begin{align*}
&I_\PP(\bx_t;\bu_t|\bu^{t-1}) \\
&=\int \log\left(\frac{\diff\PP_{\bx_t,\bu_t|u^{t-1}}}{\diff (\PP_{\bx_t|u^{t-1}}\otimes\PP_{\bu_t|u^{t-1}})}\right)\diff\PP(x_t,u^t) \\
&=\int \log\left(\frac{\diff\PP_{\bx_t|u^t}}{\diff\PP_{\bx_t|u^{t-1}}}\right)\diff\PP(x_t,u^t).
\end{align*}

(c) Let $B_{X_t}\in\mathcal{B}_{\mathcal{X}_t}$, $B_{X_{t+1}}\in\mathcal{B}_{\mathcal{X}_{t+1}}$, $B_{U^t}\in\mathcal{B}_{\mathcal{U}^t}$ be arbitrary Borel sets. Since both $\PP(x_{t+1}|u^t)$ and $\PP(x_{t+1}|x_t,u^t)$ are non-degenerate Gaussian probability measures, there exists a continuous map $f:\mathcal{X}_t\times\mathcal{X}_{t+1}\times\mathcal{U}^t\rightarrow (0,+\infty)$ such that 
\vspace{-1ex}
\begin{subequations}
\begin{align}
&k(x_t,x_{t+1},u^t)=\frac{\diff \PP(x_{t+1}|u^t)}{\diff \PP(x_{t+1}|x_t,u^t)}, \text{ or} \label{eqkrnddef0}\\
&\!\int_{B_{X_{t+1}}}\hspace{-4ex}k(x_t,x_{t+1},u^t)\diff \PP(x_{t+1}|x_t,u^t)=\!\int_{B_{X_{t+1}}}\hspace{-4ex}\diff \PP(x_{t+1}|u^t). \label{eqkrnddef}
\end{align}
\end{subequations}
In what follows, we suppress the arguments of $k(x_t,x_{t+1},u^t)$ and simply write it as $k$. Next, we express
\vspace{-1ex}
\begin{equation}
\label{eqtwoways}
\int_{B_{X_t}\times B_{X_{t+1}}\times B_{U^t}} \hspace{-10ex}k \diff\PP(x_t,x_{t+1},u^t)
\end{equation}
in two different ways:
\vspace{-1ex}
\begin{subequations}
\begin{align}
\text{(\ref{eqtwoways})} 
&=\int_{B_{X_{t+1}}\!\!\times B_{U^t}}\!\!\int_{B_{X_t}}\hspace{-2ex} k\diff \PP(x_t|x_{t+1},u^t)\diff \PP(x_{t+1},u^t); \label{eqtwowaysfirst}\\
\text{(\ref{eqtwoways})}
&=\int_{B_{X_t}\times B_{U^t}}\int_{B_{X_{t+1}}}\hspace{-3ex}k \diff\PP(x_{t+1}|x_t,u^t)\diff\PP(x_t,u^t) \\
&=\int_{B_{X_t}\times B_{U^t}}\int_{B_{X_{t+1}}}\hspace{-3ex} \diff \PP(x_{t+1}|u^t)\diff\PP(x_t,u^t) \label{eqkrndapp}\\
&=\int_{B_{U^t}}\int_{B_{X_t}}\int_{B_{X_{t+1}}}
\hspace{-3ex}\diff \PP(x_{t+1}|u^t)\diff \PP(x_t|u^t)\diff \PP(u^t) \\
&=\int_{B_{X_{t+1}}\times B_{U^t}}\int_{B_{X_t}} \diff \PP(x_t|u^t)\diff \PP(x_{t+1},u^t) \\
&=\int_{B_{X_{t+1}}\times B_{U^t}}\hspace{-3ex} \PP_{\bx_t|u^t}(B_{X_t}|u^t)\diff \PP(x_{t+1},u^t). \label{eqtwowayssecond}
\end{align}
\end{subequations}
Notice that (\ref{eqkrnddef}) is used in step (\ref{eqkrndapp}). Comparing (\ref{eqtwowaysfirst}) and (\ref{eqtwowayssecond}), we have the following identity $\PP(x_{t+1},u^t)-a.e.$:
\vspace{-1ex}
\begin{equation}
\label{eqrndk}
\int_{B_{X_t}}\hspace{-2ex} k\diff \PP(x_t|x_{t+1},u^t)=\PP_{\bx_t|u^t}(B_{X_t}|u^t).
\end{equation}
Since $k(x_t,x_{t+1},u^t)$ assumes values in $(0,+\infty)$, the first claim follows from (\ref{eqrndk}).

Now, the next equalities hold $\PP(x_{t+1},u^t)-a.e.$, which establishes the second claim.
\begin{align*}
&\int_{B_{X_t}} \frac{\diff\PP(x_{t+1}|u^t)}{\diff\PP(x_{t+1}|x_t,u^t)}\frac{\diff\PP(x_t|x_{t+1},u^t)}{\diff\PP(x_t|u^{t-1})}  \diff\PP(x_t|u^{t-1}) \\
&=\int_{B_{X_t}} k(x_t,x_{t+1},u^t)\diff\PP(x_t|x_{t+1},u^t) \\
&=\PP_{\bx_t|u^t}(B_{X_t}|u^t).
\end{align*}
The identity (\ref{eqkrnddef0}) is used in the first step, and (\ref{eqrndk}) is used in the second step.
\end{proof}
\fi

\begin{lemma} 
\label{lemmagpg}
Let $\PP(x^{T+1},u^T)$ be  a joint probability measure generated by a policy $\gamma_\PP=\{\PP(u_t|x^t,u^{t-1})\}_{t=1}^T$ and (\ref{eqsystem}), and $\GG(x^{T+1}, u^T)$ be a zero-mean jointly Gaussian probability measure having the same covariance as $\PP(x^{T+1}, u^T)$.
For every $t=1,\cdots,T$, we have
\begin{itemize}[leftmargin=4ex]
\item[(a)] $\bu^{t-1}$--$(\bx_t,\bu_t)$--$\bx_{t+1}$ form a Markov chain in $\GG$.
Moreover, for every $t=1,\cdots, T$, we have
\begin{align*}
\GG(x_{t+1}|x_t,u^t)&=\GG(x_{t+1}|x_t,u_t) \\
&=\PP(x_{t+1}|x_t,u_t) \\
&=\PP(x_{t+1}|x_t,u^t) 
\end{align*}
all of which have a nondegenerate Gaussian distribution $\mathcal{N}(A_tx_t+B_tu_t,W_t)$.
\item[(b)] For each $t=1,\cdots, T$, $\GG(x_t|x_{t+1},u^t)$ is a non-degenerate Gaussian measure for every $(x_{t+1},u^t)\in\text{supp}(\GG(x_{t+1},u^t))$.
\end{itemize}
\end{lemma}
\ifdefined\LONGVERSION
\begin{proof}
(a) Since $\bu^{t-1}$--$(\bx_t,\bu_t)$--$\bx_{t+1}$ form a Markov chain in $\PP$, and $\bx_{t+1}$ and $(\bx_t,\bu_t)$ are related by a linear map (\ref{eqsystem}), by Lemma \ref{lemkim} (borrowed from \cite{kim2010feedback}), $\bu^{t-1}$--$(\bx_t,\bu_t)$--$\bx_{t+1}$ form a Markov chain also in $\GG$. 
Notice that $\PP(x_{t+1}|x_t,u^t)=\PP(x_{t+1}|x_t,u_t)$ and $\GG(x_{t+1}|x_t,u^t)=\GG(x_{t+1}|x_t,u_t)$ hold since 
$\bu^{t-1}$--$(\bx_t,\bu_t)$--$\bx_{t+1}$ form a Markov chain both in $\PP$ and $\GG$. Since $\PP(x_{t+1},x_t,u_t)$ and $\GG(x_{t+1},x_t, u_t)$ have the same covariance, by Lemma \ref{lemlineargaussian}, a linear relationship (\ref{eqsystem}) holds both in $\PP$ and $\GG$. Thus $\GG(x_{t+1}|x_t,u_t)=\PP(x_{t+1}|x_t,u_t)$.

(b) From Lemma \ref{lemabcd} (c), $\PP(x_t|x_{t+1}, u^t)$ admits a density $\PP(x_{t+1} u^t)-a.e.$. Thus, by Lemma \ref{lemdensae}, $\GG(x_t|x_{t+1}, u^t)$ admits density for every $(x_{t+1},u^t) \in \text{supp}(\GG(x_{t+1},u^t))$.
\end{proof}
\fi

If the left hand side of (\ref{lemipg}) is finite, by Lemma~\ref{lemabcd}, it can be written as follows. 
\begin{subequations}
\begin{align}
&\sum\nolimits_{t=1}^T I_\PP(\bx_t;\bu_t|\bu^{t-1}) \nonumber \\
=&\sum\nolimits_{t=1}^T \int \log\left( \frac{\diff\PP(x_t|u^t)}{\diff\PP(x_t|u^{t-1})}\right)\diff\PP(x^{T+1}, u^T)  \nonumber\\
=&\!\int \!\log \left(\prod_{t=1}^T \frac{\diff\PP(x_t|u^t)}{\diff\PP(x_t|u^{t-1})}\right)\diff\PP(x^{T+1},  u^T) \nonumber \\
=&\!\int \!\!\log \!\left(\prod_{t=1}^T\! \frac{\diff\PP(x_t|x_{t+1},u^t)}{\diff\PP(x_t|u^{t-1})}\frac{\diff\PP(x_{t+1}|u^t)}{\diff\PP( x_{t+1}|x_t,u^t)}\!\right)\!\diff\PP(x^{T+1}\!\!,  u^T)\nonumber \\
=&\int \log \left(\frac{\diff\PP(x_1|x_2,u_1)}{\diff\PP(x_1)}\right)\diff\PP(x^{2}, u_1) \label{ptog1}\\
&+\sum_{t=2}^T\int \log \left(\frac{\diff\PP(x_t|x_{t+1},u^t)}{\PP(x_t|x_{t-1},u^{t-1})}\right)\diff\PP(x^{t+1}, u^t) \label{ptog2}\\
&+\int \log \left(\frac{\diff\PP(x_{T+1}|u^T)}{\diff\PP(x_{T+1}|x_T,u^T)}\right)\diff\PP(x^{T+1}, u^T) \label{ptog3}
\end{align}
\end{subequations}
The result of Lemma \ref{lemabcd} (c) is used in the third equality. 
In the final step, the the chain rule for the Radon-Nikodym derivatives  \cite[Proposition 3.9]{folland1999real} is used multiple times for telescoping cancellations.
We show that each term in (\ref{ptog1}),  (\ref{ptog2}) and  (\ref{ptog3}) does not increase by replacing the probability measure $\PP$ with $\GG$. Here we only show the case for (\ref{ptog2}), but a similar technique is also applicable to (\ref{ptog1}) and (\ref{ptog3}).
\begin{subequations}
\begin{align}
&\int \log \left(\frac{\diff\PP(x_t|x_{t+1},u^t)}{\diff\PP(x_t|x_{t-1},u^{t-1})}\right) \diff\PP(x^{t+1},u^t) \nonumber \\
&-\int \log \left(\frac{\diff\GG(x_t|x_{t+1},u^t)}{\diff\GG(x_t|x_{t-1},u^{t-1})} \right)\diff\GG(x^{t+1}, u^t) \label{rndgauss} \\
=&\int \log \left(\frac{\diff\PP(x_t|x_{t+1},u^t)}{\diff\PP(x_t|x_{t-1},u^{t-1})} \right)\diff\PP(x^{t+1}, u^t) \nonumber \\
&-\int \log \left(\frac{\diff\GG(x_t|x_{t+1},u^t)}{\diff\GG(x_t|x_{t-1},u^{t-1})}\right) \diff\PP( x^{t+1}, u^t) \label{measurereplace}\\
=&\!\int \!\!\log\!\left(\!\frac{\diff\PP(x_t|x_{t+1},u^t)}{\diff\PP(x_t|x_{t-1},u^{t-1})}\frac{\diff\GG(x_t|x_{t-1},u^{t-1})}{\diff\GG(x_t|x_{t+1},u^t)} \!\right)\!\diff\PP(x^{t+1}\!\!, u^t) \nonumber \\
=&\int \log\left(\frac{\diff\PP(x_t|x_{t+1},u^t)}{\diff\GG(x_t|x_{t+1},u^t)}\right) \diff\PP(x^{t+1}, u^t) \label{rndmult}\\
=&\!\int \!\!\left[ \int \!\!\log\left(\!\frac{\diff\PP(x_t|x_{t+1},\!u^t)}{\diff\GG(x_t|x_{t+1},\!u^t)}\!\right) \diff\PP(x_t|x_{t+1},\!u^t)\right]\! \diff\PP(x_{t+1},u^t) \nonumber \\
=& \int D \left( \PP(x_t|x_{t+1},u^t) \| \GG(x_t|x_{t+1},u^t) \right) \diff\PP(x_{t+1}, u^t) \nonumber \\
\geq & \; 0. \nonumber 
\end{align}
\end{subequations}
Due to Lemma \ref{lemmagpg}, $\log \frac{\diff\GG(x_t|x_{t+1},u^t)}{\diff\GG(x_t|x_{t-1},u^{t-1})}$ in (\ref{rndgauss}) is a quadratic function of $x^{t+1}$ and $u^{t}$ everywhere on $\text{supp}(\GG(x^{t+1}, u^t))$.
This is also the case everywhere on $\text{supp}(\PP(x^{t+1}, u^t))$ since it follows from Lemma \ref{lemsupport} that $\text{supp}(\PP(x^{t+1}, u^t)) \subseteq \text{supp}(\GG(x^{t+1}, u^t))$.
Since $\PP$ and $\GG$ have the same covariance, $\diff\GG(x^{t+1}, u^t)$ can be replaced by $\diff\PP(x^{t+1}, u^t)$ in (\ref{measurereplace}).
In (\ref{rndmult}), the chain rule of the Radon-Nikodym derivatives is used invoking that $\PP(x_t|x_{t-1}, u^{t-1})=\GG(x_t|x_{t-1}, u^{t-1})$ from Lemma \ref{lemmagpg} (a).

\subsection{Proof of Lemma~\ref{lemmags}}
\label{appprooflemmags}
Clearly $\GG(x_1)=\QQ(x_1)$ holds. Following an induction argument, assume that the claim holds for $t=k-1$. Then
\begin{subequations}
\begin{align}
\hspace{-1ex}&\diff\QQ(x_{k+1}, u^k) \nonumber \\
\hspace{-1ex}=&\!\!\int_{\mathcal{X}_k}\!\!\! \diff\QQ(x_k,x_{k+1},u^k) \nonumber \\
\hspace{-1ex}=&\!\!\int_{\mathcal{X}_k}\!\!\!\! \diff\PP(x_{k+1}|x_k,u_k)\diff\QQ(x_k,u^k) \label{lemma4_1}\\
\hspace{-1ex}=&\!\!\int_{\mathcal{X}_k}\!\!\!\! \diff\PP(x_{k+1}|x_k,u_k)\diff\QQ(u_k|x_k,u^{k-1})\diff\QQ(x_k,u^{k-1}) \label{lemma4_2}\\
\hspace{-1ex}=&\!\!\int_{\mathcal{X}_k}\!\!\!\! \diff\PP(x_{k+1}|x_k,u_k)\diff\QQ(u_k|x_k,u^{k-1})\diff\GG(x_k,u^{k-1}) \label{lemma4_3}\\
\hspace{-1ex}=&\!\!\int_{\mathcal{X}_k}\!\!\!\! \diff\PP(x_{k+1}|x_k,u_k)\diff\GG(x_k,u^k) \label{lemma4_4}\\
\hspace{-1ex}=&\!\!\int_{\mathcal{X}_k}\!\!\! \diff\GG(x_k,x_{k+1},u^k) \label{lemma4_5}\\
\hspace{-1ex}=&\diff\GG(x_{k+1},u^k). \nonumber 
\end{align}
\end{subequations}
The integral signs ``$\int_{B_{X_{k+1}}\times B_{U^k}}$" in front of each of the above expressions are omitted for simplicity.
Equations (\ref{lemma4_1}) and (\ref{lemma4_2}) are due to (\ref{defs1}) and (\ref{defs2}) respectively. In (\ref{lemma4_3}), the induction assumption $\GG(x_{k},u^{k-1})=\QQ(x_{k},u^{k-1})$ is used. Identity (\ref{lemma4_4}) follows from the definition (\ref{defconds}). The result of Lemma \ref{lemmagpg}(b) was used in (\ref{lemma4_5}).

\subsection{Proof of Theorem~\ref{theostationary} (Outline only)}
\label{apptheostationary}

First, it can be shown that the three-stage separation principle continues to hold for the infinite horizon problem (\ref{stationaryprob}).
The same idea of proof as in Section \ref{secderivation} is applicable; for every policy $\gamma_\PP=\{\PP(u_t|x^t,u^{t-1})\}_{t\in\mathbb{N}}$, there exists a linear-Gaussian policy $\gamma_\QQ=\{\QQ(u_t|x^t,u^{t-1})\}_{t\in\mathbb{N}}$ which is at least as good as $\gamma_\PP$.
Second, the optimal certainty equivalence controller gain is time-invariant.
This is because, since $(A,B)$ is stabilizable, for every finite $t$, the solution $S_t$ of the Riccati recursion (\ref{backwardriccati}) converges to the solution $S$ of (\ref{algriccati}) as $T\rightarrow \infty$ \cite[Theorem 14.5.3]{kailath2000linear}.
Third, the optimal AWGN channel design problem becomes an SDP over an infinite sequence $\{P_{t|t},\Pi_t\}_{t\in\mathbb{N}}$ similar to (\ref{optprob3}) in which ``$\sum_{t=1}^T$'' is replaced by ``$\limsup_{T\rightarrow \infty}\frac{1}{T}\sum_{t=1}^T$'' and parameters $A_t, W_t, S_t, \Theta_t$ are time-invariant. It is shown in \cite{srdstationary} that the optimality of this SDP over $\{P_{t|t},\Pi_t\}_{t\in\mathbb{N}}$ is attained by a time-invariant sequence $P_{t|t}=P, \Pi_t=\Pi \; \forall t\in\mathbb{N}$, where $P$ and $\Pi$ are the optimal solution to (\ref{ltisdp}).

\subsection{Proof of Corollary~\ref{cordatarate}}
\label{appdatarate}
We write $v^*(A,W)\triangleq \lim_{D\rightarrow +\infty} R(D)$ to indicate its dependency on $A$ and $W$.
From (\ref{ltisdp}), we have
\vspace{-0.5ex}
\begin{align}
& v^*(A,W)= \label{valprimal} \\
& \begin{cases} \inf\limits_{P, \Pi}  \quad \tfrac{1}{2} \log\det \Pi^{-1} + \tfrac{1}{2} \log \det W \\
 \text{ s.t.}  \quad  \Pi \!\succ\!  0,    P \!\preceq\! A P A^\top \!+\!W,  \left[\!\! \begin{array}{cc}P\!-\!\Pi \!\!\! &\!\! PA^\top \\
AP \!\!\!&\!\! A PA^\top \!+\!W \end{array}\!\!\right]\! \succeq\! 0. \end{cases} \nonumber
\end{align}
Due to the strict feasibility, Slater's constraint qualification \cite{boyd2009} guarantees that the duality gap is zero. Thus, we have an alternative representation of $v^*(A,W)$ using the dual problem of (\ref{valprimal}).
\vspace{-0.5ex}
\begin{align}
& v^*(A,W)= \label{valdual} \\
& \begin{cases}\sup\limits_{X,Y}  \quad \tfrac{1}{2} \log\det X_{11}\!-\!\tfrac{1}{2}\text{Tr}(X_{22}\!+\!Y)W + \tfrac{1}{2}\log\det W \!+\! \tfrac{n}{2} \\
 \text{ s.t.}  \quad  A^\top YA\!-\!Y\!+\!X_{11}\!+\!X_{12}A\!+\!A^\top X_{21}\!+\!A^\top X_{22}A \preceq 0,  \\
  \hspace{6ex}Y \succeq 0, X=\left[\!\! \begin{array}{cc}X_{11} & X_{12} \\ X_{21} & X_{22}\end{array}\!\!\right]\! \succeq\! 0.\end{cases} \nonumber
\end{align}
The primal problem (\ref{valprimal}) can be also rewritten as
\begin{align}
&v^*(A,W) \nonumber \\
&=\begin{cases}
\inf\limits_{P} \quad \frac{1}{2}\log\det (APA^\top\!+\!W)-\frac{1}{2}\log\det P \\
\text{s.t.}\quad P \preceq APA^\top+W, P\in\mathbb{S}_{++}^n 
\end{cases} \label{RinfAW1}\\
&=
\begin{cases}
\inf\limits_{P,C,V}  \quad -\frac{1}{2}\log\det (I\!-\!V^{-\frac{1}{2}}CPC^\top V^{-\frac{1}{2}}) \\
\text{ s.t.}\quad P^{-1}-(APA^\top+W)^{-1}=C^\top V^{-1}C \\
\hspace{6ex}P\in\mathbb{S}_{++}^n, V\in\mathbb{S}_{++}^n, C\in\mathbb{R}^{n\times n}. \label{RinfAW2}
\end{cases}
\end{align}
To see that (\ref{RinfAW1}) and (\ref{RinfAW2}) are equivalent, note that the feasible set of $P$ in (\ref{RinfAW1}) and (\ref{RinfAW2}) are the same. Also
\vspace{-0.5ex}
\begin{align*}
&\tfrac{1}{2}\log\det(APA^\top+W)-\tfrac{1}{2}\log\det P \\
&=-\tfrac{1}{2}\log\det(APA^\top+W)^{-1}-\tfrac{1}{2}\log\det P \\
&=-\tfrac{1}{2}\log\det(P^{-1}-C^\top V^{-1}C)-\tfrac{1}{2}\log\det P \\
&=-\tfrac{1}{2}\log\det(I-P^{\frac{1}{2}}C^\top V^{-1}CP^{\frac{1}{2}}) \\
&=-\tfrac{1}{2}\log\det(I-V^{-\frac{1}{2}}C PC^\top V^{-\frac{1}{2}})
\end{align*}
The last step follows from Sylvester's determinant theorem. 

\subsubsection{Case 1: When all eigenvalues of $A$ satisfy $|\lambda_i|\geq 1$}
We first show that if all eigenvalues of $A$ are outside the open unit disc, then $v^*(A,W)=\sum_{\lambda_i\in\sigma (A)} \log |\lambda_i|$, where $\sigma (A)$ is the set of all eigenvalues of $A$ counted with multiplicity.
To see that $v^*(A,W) \leq \sum_{\lambda_i\in\sigma (A)} \log |\lambda_i|$, note that the value $\sum_{\lambda_i\in\sigma (A)} \log |\lambda_i|+\epsilon$ with arbitrarily small $\epsilon >0$ can be attained by $P=kI$ in (\ref{RinfAW1}) with sufficiently large $k>0$. To see that $v^*(A,W) \geq \sum_{\lambda_i\in\sigma (A)} \log |\lambda_i|$, note that the value $\sum_{\lambda_i\in\sigma (A)} \log |\lambda_i|$ is attained by the dual problem (\ref{valdual}) with $X=[A \;\;-I]^\top W^{-1}[A \;\;-I]$ and $Y=0$.

\subsubsection{Case 2: \mbox{When all eigenvalues of $A$ satisfy $|\lambda_i|<1$}}
In this case, we have $v^*(A,W)=0$. 
The fact that $v^*(A,W)\geq 0$ is immediate from the expression (\ref{RinfAW1}). To see that $v^*(A,W)=0$, consider $P=P^*$ in (\ref{RinfAW1}) where $P^*\succ 0$ is the unique solution to the Lyapunov equation $P^*=AP^*A^\top+W$.

\subsubsection{Case 3: General case}
In what follows, we assume without loss of generality that $A$ has a structure (e.g., a Jordan form) 
\vspace{-1ex}
$$A=\left[\begin{array}{cc}A_1 & 0 \\ 0 & A_2\end{array}\right]$$ where all eigenvalues of $A_1\in\mathbb{R}^{n_1\times n_1}$ satisfy $|\lambda_i|\geq 1$ and all eigenvalues of $A_2\in\mathbb{R}^{n_2\times n_2}$ satisfy $|\lambda_i|< 1$.
We first recall the following basic property of the algebraic Riccati equation.
\vspace{-1ex}
\begin{lemma}
\label{lemriccatimonotone}Suppose $V\succ 0$ and $(A,C)$ is a detectable pair and $0\prec W_1 \preceq W_2$. Then, we have $\tilde{P}\preceq \tilde{Q}$ where  $\tilde{P}$ and $\tilde{Q}$ are the unique positive definite solutions to 
\begin{align}
A\tilde{P}A^\top\!-\!\tilde{P}\!-\!A\tilde{P}C^\top (C\tilde{P}C^\top\!+\!V)^{-1}C\tilde{P}A^\top\!+\!W_1&\!=\!0  \label{eqriccatip}\\
A\tilde{Q}A^\top\!-\!\tilde{Q}\!-\!A\tilde{Q}C^\top (C\tilde{Q}C^\top\!+\!V)^{-1}C\tilde{Q}A^\top\!+\!W_2&\!=\!0.  \label{eqriccatiq}
\end{align}
\end{lemma}
\begin{proof}
Consider Riccati recursions
\begin{align}
&\tilde{P}_{t+1}\!=\!A\tilde{P}_t A^\top \!\!-\! A\tilde{P}_tC^\top(C\tilde{P}_tC^\top\!\!\!+\!V)^{-1}C\tilde{P}_tA^\top\!+\!W_1  \label{ricrecp} \\
&\tilde{Q}_{t+1}\!=\!A\tilde{Q}_t A^\top \!\!\!-\! A\tilde{Q}_tC^\top(C\tilde{Q}_tC^\top\!\!\!+\!V)^{-1}C\tilde{Q}_tA^\top\!\!\!+\!W_2  \label{ricrecq}
\end{align}
with $\tilde{P}_0=\tilde{Q}_0 \succ 0$. 
Since (RHS of (\ref{ricrecp})) $\preceq$ (RHS of (\ref{ricrecq})) for every $t$, we have $\tilde{P}_t \preceq \tilde{Q}_t$ for every $t$ (see also \cite[Lemma 2.33]{kumar1986stochastic} for the monotonicity of the Riccati recursion). 
Under the detectability assumption, we have $\tilde{P}_t\rightarrow \tilde{P}$ and $\tilde{Q}_t\rightarrow \tilde{Q}$ as $t\rightarrow +\infty$ \cite[Theorem 14.5.3]{kailath2000linear}.
Thus $\tilde{P}\preceq \tilde{Q}$.
\end{proof}
Using the above lemma, we obtain the following result.
\begin{lemma}
\label{lemvmonotone}
$0\prec W_1 \preceq W_2$, then $v^*(A, W_1)\leq v^*(A, W_2)$.
\end{lemma}
\begin{proof}
Due to the characterization (\ref{RinfAW2}) of $v^*(A,W_2)$, there exist $Q\succ 0, V \succ 0, C\in\mathbb{R}^{n \times n}$ such that $v^*(A,W_2)=-\frac{1}{2}\log\det(I\!-\!V^{-\frac{1}{2}}CQC^\top V^{-\frac{1}{2}})$
and
\begin{equation}
\label{eqricqq}
Q^{-1}-(AQA^\top+W_2)^{-1}=C^\top V^{-1}C.
\end{equation}
Setting $\tilde{Q}\triangleq AQA^\top+W_2 \succ 0$, it is elementary to show that (\ref{eqricqq}) implies $\tilde{Q}$ satisfies the algebraic Riccati equation (\ref{eqriccatiq}).
Setting $\tilde{L}\triangleq A\tilde{Q}C^\top (C\tilde{Q}C^\top+V)^{-1}$, (\ref{eqriccatiq}) implies a Lyapunov inequality $(A-\tilde{L}C)\tilde{Q}(A-\tilde{L}C)^\top-\tilde{Q}\prec 0$, showing that $A-\tilde{L}C$ is Schur stable. Hence $(A,C)$ is a detectable pair.
By Lemma~\ref{lemriccatimonotone}, a Riccati equation (\ref{eqriccatip}) admits a positive definite solution $\tilde{P}\preceq \tilde{Q}$. 
Setting $P\triangleq (\tilde{P}^{-1}+C^\top V^{-1}C)^{-1}$, $P$ satisfies 
\vspace{-1ex}
\begin{equation}
\label{eqricpp}
P^{-1}-(A PA^\top+W_1)^{-1}=C^\top V^{-1}C
\end{equation}
Moreover, we have $P\preceq Q$ since
\[
0\prec Q^{-1} = \tilde{Q}^{-1}+C^\top V^{-1}C\preceq \tilde{P}^{-1}+C^\top V^{-1}C=P^{-1}.
\]
Since $P$ satisfies (\ref{eqricpp}), we have thus constructed a feasible solution $(P, C, V)$ that upper bounds $v^*(A,W_1)$. That is,
\begin{align*}
v^*(A,W_2)&=-\tfrac{1}{2}\log\det(I\!-\!V^{-\frac{1}{2}}CQC^\top V^{-\frac{1}{2}}) \\
&\geq -\tfrac{1}{2}\log\det(I\!-\!V^{-\frac{1}{2}}CPC^\top V^{-\frac{1}{2}}) \\
&\geq v^*(A,W_1). 
\end{align*}
\end{proof}

Next, we prove that $v^*(A,W)$ is both upper and lower bounded by $\sum_{\lambda_i\in\sigma(A_1)} \log |\lambda_i|$. To establish an upper bound, note that the following inequalities hold with a sufficiently large $\delta>0$ with $W \preceq \delta I_n$.
\begin{align*}
v^*(A,W)&\leq v^*(A, \delta I_n) \\
&\leq v^*(A_1, \delta I_{n_1}) +v^*(A_2, \delta I_{n_2})  =\sum_{\lambda_i\in\sigma(A_1)} \log |\lambda_i|. 
\end{align*}
Lemma~\ref{lemvmonotone} is used in the first step.
To see the second inequality, consider the primal representation (\ref{valprimal}) of $v^*(A, \delta I_n)$.
If we restrict decision variables to have block-diagonal structures
\[
P=\left[\begin{array}{cc}P_1 & 0 \\ 0& P_2 \end{array}\right], \;\; \Pi=\left[\begin{array}{cc}\Pi_1 & 0 \\ 0& \Pi_2 \end{array}\right]
\] 
according to the partitioning $n=n_1+n_2$, then the original primal problem (\ref{valprimal}) with $(A, \delta I_n)$ is decomposed into a problem in terms of decision variables $(P_1, \Pi_1)$ with data $(A_1, \delta I_{n_1})$ and a problem in terms of decision variables $(P_2,\Pi_2)$ with data $(A_2, \delta I_{n_2})$. 
Due to the additional structural restriction, the sum of $v^*(A_1, \delta I_{n_1})$ and $v^*(A_2, \delta I_{n_2})$ cannot be smaller than $v^*(A, \delta I_n)$.
Finally, by the arguments in Cases 1 and 2, we have $v^*(A_1, \delta I_{n_1})=\sum_{\lambda_i\in\sigma(A_1)} \log |\lambda_i|$ and $v^*(A_2, \delta I_{n_2})=0$.

To establish a lower bound, we show the following inequalities using a sufficiently small $\epsilon>0$ such that $\epsilon I \preceq W$.
\begin{align*}
v^*(A,W)&\geq v^*(A, \epsilon I_n) \\
&\geq v^*(A_1, \epsilon I_{n_1}) +v^*(A_2, \epsilon I_{n_2}) =\sum_{\lambda_i\in\sigma(A_1)} \log |\lambda_i|. 
\end{align*}
The first inequality is due to Lemma~\ref{lemvmonotone}.
To prove the second inequality, consider the dual representation (\ref{valdual}) of $v^*(A, \epsilon I_n)$.
By restricting decision variables $X_{11}, X_{12}, X_{21},X_{22}$ and $Y$ to have block-diagonal structures according to the partitioning $n=n_1+n_2$, the original dual problem is decomposed into two problems of the form (\ref{valdual}) with $(A_1, \epsilon I_{n_1})$ and $(A_2, \epsilon I_{n_2})$.
Since the additional constraints in the dual problem never increase the optimal value, we have the second inequality. Discussions in Cases 1 and 2 are again used in the last step.

%\ifdefined\LONGVERSION
\subsection{Proof of Theorem~\ref{theorempartially}}
\label{apppartially}

To prove Theorem~\ref{theorempartially}, we reduce the original problem \eqref{partiallyobservableprob} for partially observable plants to a problem for fully observable plants so that the results obtained in Section~\ref{secmainresult} are applicable. To this end, a key technique is the \emph{innovations approach} \cite{kailath1968innovations}, which has been used in the context of zero-delay rate-distortion theory for a similar purpose \cite{tanaka2015zero}. 

First, observe that the least-mean square error estimate $\tilde{\bx}_t=\mathbb{E}(\bx_t|\by^t, \bu^{t-1})$ can be computed recursively by the pre-Kalman filter (\ref{eqprekf}).
In particular, it is apparent from (\ref{eqprekf}) that the expectation $\mathbb{E}(\bx_t|\by^t, \bu^{t-1})$ can be written as a linear function $\mathcal{L}_t^{KF}$ of $\by^t$ and $\bu^{t-1}$, i.e.,
\begin{equation}
\label{eqLkf}
\int x_t\PP(dx_{t+1}|y^t,u^{t-1})=\mathcal{L}_t^{KF}(y^t,u^{t-1}).
\end{equation}
The Kalman filter is also known as the whitening filter since it has a property that the \emph{innovation} process $\bpsi_t\triangleq \tilde{L}_{t+1}(\by_{t+1}-H_{t+1}\tilde{\bx}_{t+1|t})$ is white. 
\begin{lemma}
The innovation $\bpsi_t$ is a zero-mean, white (i.e., independent of $\tilde{\bx}^t, \bu^t$ and $\bpsi^{t-1}$) Gaussian random variable such that $\bpsi_t\sim\mathcal{N}(0,\Psi_t)$ with $\Psi_t=\tilde{L}_{t+1}(H_{t+1}\tilde{P}_{t+1|t}H_{t+1}^\top+G_{t+1})\tilde{L}_{t+1}^\top$.
\end{lemma}
\begin{proof}See, e.g., \cite[Section 10.1]{simon2006optimal}.
\end{proof}
The next observation is important to show that the pre-Kalman filter can be always introduced without loss of performance.
\begin{lemma}
\label{lemprekfinv}
The pre-Kalman filter (\ref{eqprekf}) is causally invertible; that is, for every $t=1,\cdots,T$, $(\by^t,\bu^{t-1})$ can be reconstructed from  $(\tilde{\bx}^t,\bu^{t-1})$.
\end{lemma}
\begin{proof}
Since we are assuming $W_t\succ 0$ and $H_t$ has full row rank, $\tilde{L}_t$ has full column rank. Thus $\tilde{L}_t^\dagger \triangleq (\tilde{L}_t^\top \tilde{L}_t)^{-1}\tilde{L}_t^\top$ exists for every $t=1,\cdots, T$. 
From (\ref{eqprekf}), it is clear that $\by_t$ can be constructed by
\begin{equation}
\label{eqprekfinv}
\by_t=\tilde{L}_t^\dagger \tilde{\bx}_t+\tilde{L}_t^\dagger (\tilde{L}_t H_t-I)(A_{t-1}\tilde{\bx}_{t-1}+B_{t-1}\bu_{t-1}).
\end{equation}
\end{proof}

 \ifdefined\DOUBLECOLUMN	
\begin{figure}[t]
    \centering
    \includegraphics[width=0.7\columnwidth]{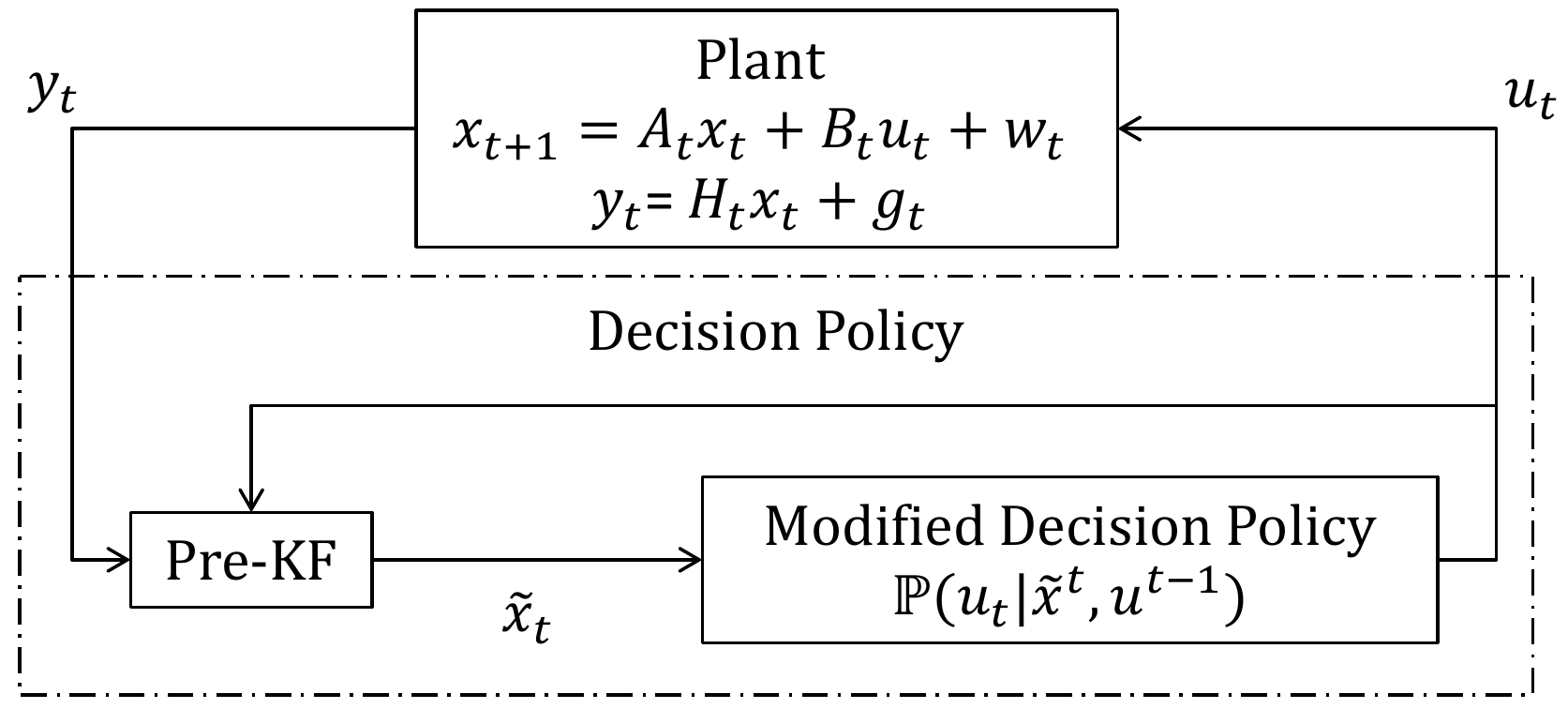}
    \caption{The space of original decision policy is parametrized by the modified decision policies.}
    \label{fig:preKF}
    \vspace{-2ex}
\end{figure}
\begin{figure}[t]
    \centering
    \includegraphics[width=0.7\columnwidth]{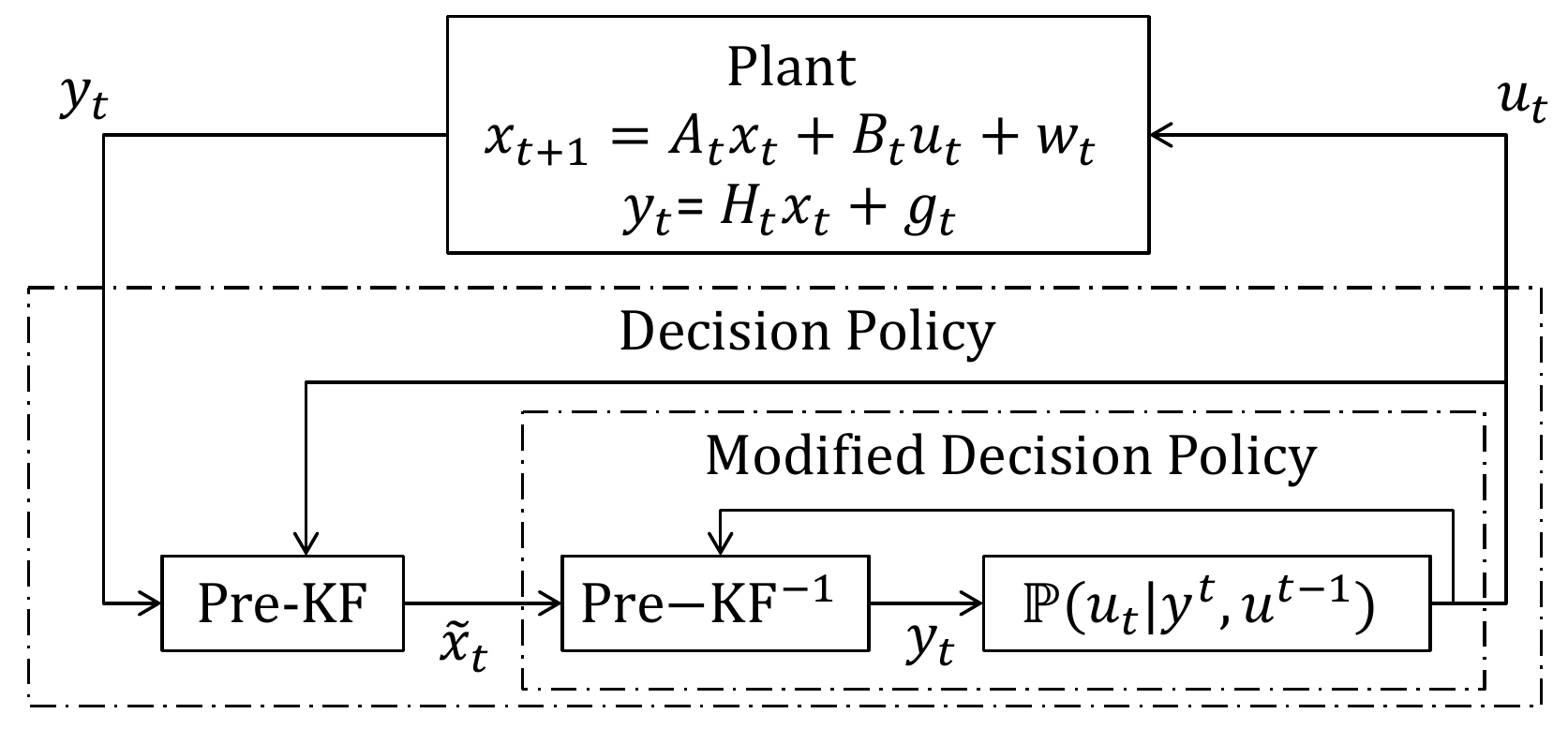}
    \caption{Proof of Lemma~\ref{propprekf}.}
    \label{fig:preKFinv}
    \vspace{-2ex}
\end{figure}
\fi
 \ifdefined\SINGLECOLUMN	
\begin{figure}[t]
    \centering
    \includegraphics[width=0.6\columnwidth]{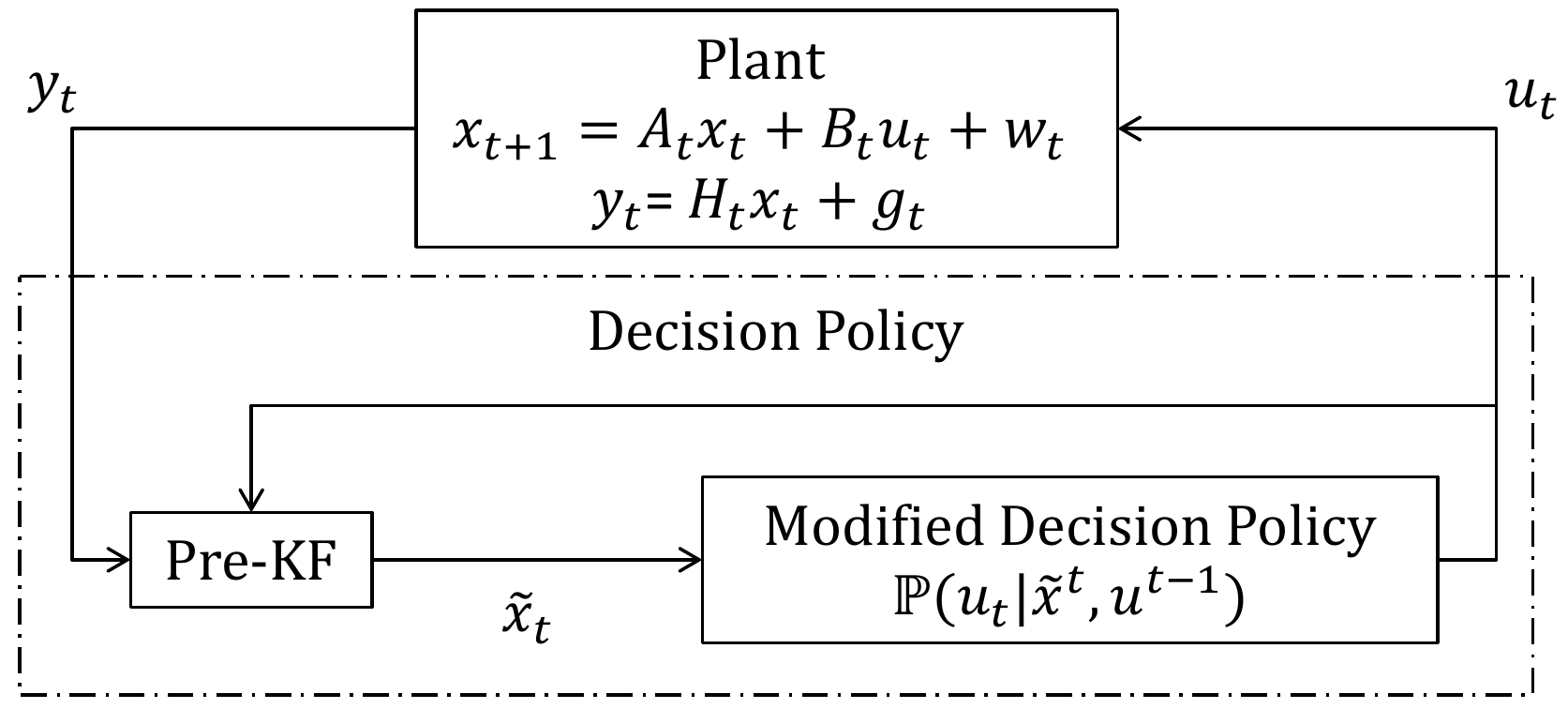}
    \caption{The space of original decision policy is parametrized by the modified decision policies.}
    \label{fig:preKF}
\end{figure}
\begin{figure}[t]
    \centering
    \includegraphics[width=0.6\columnwidth]{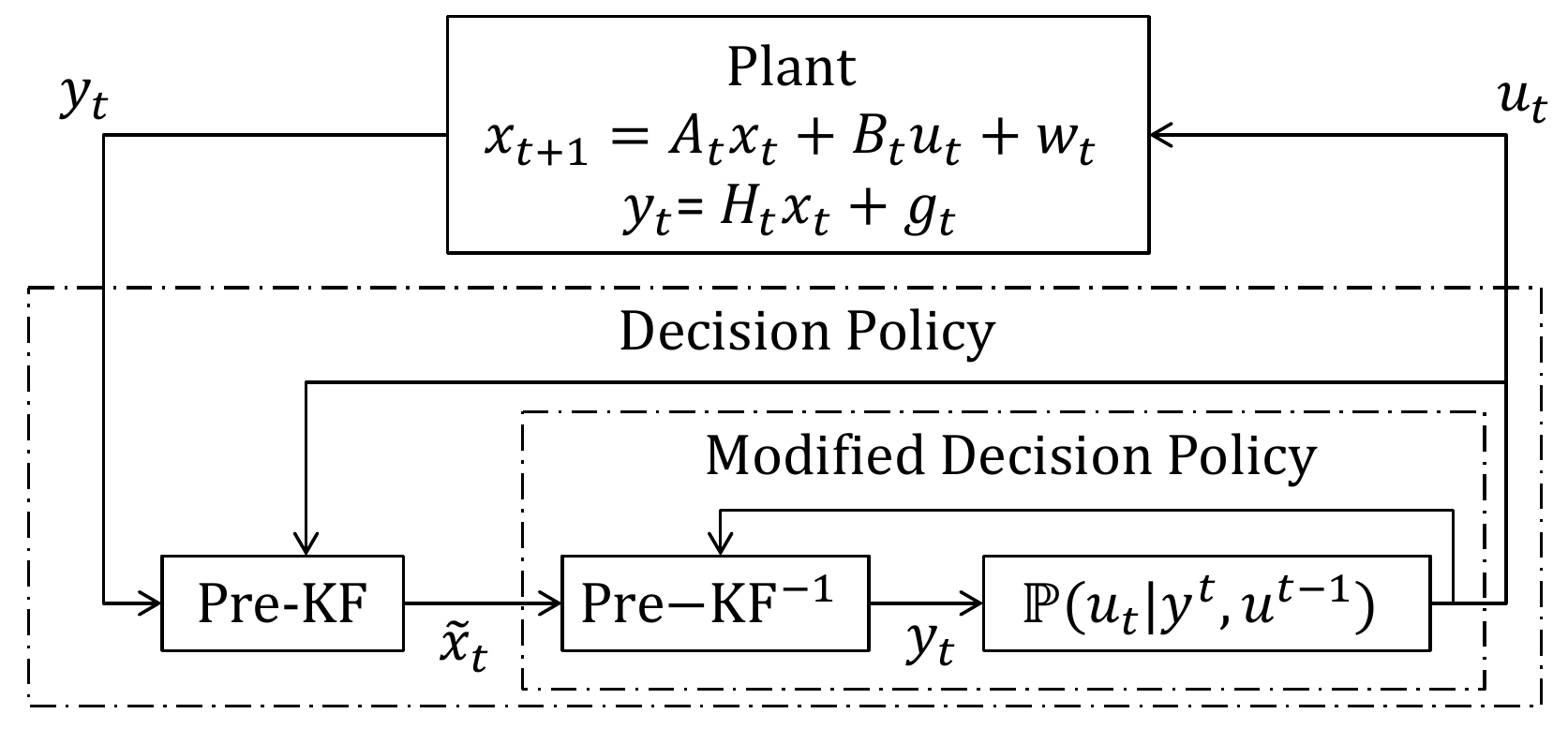}
    \caption{Proof of Lemma~\ref{propprekf}.}
    \label{fig:preKFinv}
\end{figure}
\fi 
\begin{lemma}
\label{propprekf}
Without loss of performance, one can assume a decision architecture in Fig.~\ref{fig:preKF}, where the ``Pre-KF" block represents the pre-Kalman filter (\ref{eqprekf}).
\end{lemma}
\begin{proof}
Suppose $\gamma=\PP(u^T\|y^T)$ is an optimal solution to the original problem \eqref{partiallyobservableprob}. 
Construct a modified decision policy block in Figure~\ref{fig:preKF} as shown in Figure~\ref{fig:preKFinv} using the optimal policy $\gamma=\PP(u^T\|y^T)$ for the original problem, where the ``$\text{Pre-KF}^{-1}$" block represents the causal inverse (\ref{eqprekfinv}) of the pre-Kalman filter. 
Then, the ``decision policy" block in Fig.~\ref{fig:preKFinv} is equivalent to the ``decision policy" block in the original problem in Figure~\ref{fig:poprob}.
Thus, we have shown by construction that there exists a modified decision policy with which the ``decision policy'' block in Figure~\ref{fig:preKF} attains optimality in the the original problem \eqref{partiallyobservableprob}.
\end{proof}

\begin{lemma}
\label{propxtilde}
For every $t=1,\cdots, T$,
\[
\mathbb{E}\|\bx_{t+1}\|_{Q_t}^2=\text{Tr}(Q_t\tilde{P}_{t+1|t+1})+\mathbb{E}\|\tilde{\bx}_{t+1}\|_{Q_t}^2.
\]
\end{lemma}
\begin{proof}
Observe 
\begin{align*}
&\mathbb{E}\|\bx_{t+1}\|_{Q_t}^2=\mathbb{E}\|\bx_{t+1}-\tilde{\bx}_{t+1}+\tilde{\bx}_{t+1}\|_{Q_t}^2 \\
&=\mathbb{E}\|\bx_{t+1}\!-\!\tilde{\bx}_{t+1}\|_{Q_t}^2\!+\!\mathbb{E}\|\tilde{\bx}_{t+1}\|_{Q_t}^2 \!+\!2\mathbb{E}\tilde{\bx}_{t+1}^\top Q_t(\bx_{t+1}\!-\!\tilde{\bx}_{t+1}).
\end{align*}
Since $\mathbb{E}\|\bx_{t+1}-\tilde{\bx}_{t+1}\|_{Q_t}^2=\text{Tr}(Q_t\tilde{P}_{t+1|t+1})$, it suffices to prove $\mathbb{E}\tilde{\bx}_{t+1}^\top Q_t(\bx_{t+1}-\tilde{\bx}_{t+1})=0$. This can be directly verified as in (\ref{bigintegral}).
\begin{align}
&\mathbb{E}\left\{\tilde{\bx}_{t+1}^\top Q_t(\bx_{t+1}-\tilde{\bx}_{t+1}) \right\} \nonumber \\
&=\!\int \tilde{x}_{t+1}^\top Q_t (x_{t+1}-\tilde{x}_{t+1})\PP(dx_{t+1},dy^{t+1},du^t) \nonumber\\
&=\!\int \mathcal{L}_{t+1}^{KF}(y^{t+1}\!\!,u^t)^{\top} \!Q_t \!\left(x_{t+1}\!-\!\mathcal{L}_{t+1}^{KF}(y^{t+1}\!\!,u^t)\right) \nonumber \\
& \hspace{35ex} \times \PP(dx_{t+1},dy^{t+1}\!\!\!,du^t) \nonumber\\ 
&=\!\int  \mathcal{L}_{t+1}^{KF}(y^{t+1},u^t)^\top Q_t \nonumber \\
& \times \left(\int\! \left(x_{t+1}\!-\!\mathcal{L}_{t+1}^{KF}(y^{t+1}\!\!\!,u^t)\right)\! \PP(dx_{t+1}|y^{t+1}\!\!\!,u^t) \right) \PP(dy^{t+1}\!\!\!,du^t) \nonumber\\
&=\!\int  \mathcal{L}_{t+1}^{KF}(y^{t+1},u^t)^\top Q_t \nonumber \\
& \times \left(\!\int \!x_{t+1}\PP(dx_{t+1}|y^{t+1}\!\!\!,u^t)\!-\!\mathcal{L}_{t+1}^{KF}(y^{t+1}\!\!\!,u^t) \right) \PP(dy^{t+1}\!\!\!,du^t) \nonumber \\
&=0 \label{bigintegral}
\end{align}
The last step is due to (\ref{eqLkf}).
\end{proof}

Finally, we are ready to reduce the original problem 
\eqref{partiallyobservableprob} for partially observable plants to a problem for fully observable plants.
Let $\tilde{\gamma}=\PP(u^T \|\tilde{x}^T)$  be a modified decision policy in Fig.~\ref{fig:preKF}, and let $\tilde{\Gamma}$ be the space of such policies. Notice that if a policy $\tilde{\gamma}\in\tilde{\Gamma}$ is fixed, then the system equation (\ref{eqsystem}) and the pre-Kalman filter equation (\ref{eqprekf}) uniquely define a joint probability measure $\PP(x^{T+1},y^T,\tilde{x}^T,u^T)$. Expectation and the mutual information with respect to this probability measure will be denoted by $\mathbb{E}_{\tilde{\gamma}}$ and $I_{\tilde{\gamma}}$.
By Lemma~\ref{propprekf}, our main optimization problem (\ref{mainprob}) can be equivalently written as
\begin{subequations}
\label{equivalentprob1}
\begin{align}
\min_{\tilde{\gamma}\in\tilde{\Gamma}} & \quad  \sum\nolimits_{t=1}^T I_{\tilde{\gamma}} (\by^t; \bu_t|\bu^{t-1}) \\
\text{ s.t. } &\quad \sum\nolimits_{t=1}^T \mathbb{E}_{\tilde{\gamma}}\left(\|\bx_{t+1}\|_{Q_t}^2+\|\bu_t\|_{R_t}^2 \right) \leq D.
\end{align}
\end{subequations}
By Lemma \ref{lemprekfinv}, given $\bu^{t-1}$, $\by^t$ can be reconstructed from $\tilde{\bx}^t$ and \emph{vice versa}. Thus, we have
$
I_{\tilde{\gamma}} (\by^t; \bu_t|\bu^{t-1})=I_{\tilde{\gamma}} (\tilde{\bx}^t; \bu_t|\bu^{t-1}).
$
Therefore, using Lemma~\ref{propxtilde}, problem (\ref{equivalentprob1}) can be further rewritten as
\begin{align}
\min_{\tilde{\gamma}\in\tilde{\Gamma}} &   \sum\nolimits_{t=1}^T I_{\tilde{\gamma}} (\tilde{\bx}^t; \bu_t|\bu^{t-1}) \label{equivalentprob2} \\
\text{ s.t. } & \sum\nolimits_{t=1}^T \!\!\!\!\mathbb{E}_{\tilde{\gamma}}\!\left(\|\tilde{\bx}_{t+1}\|_{Q_t}^2\!\!\!+\!\|\bu_t\|_{R_t}^2 \right) +\!\text{Tr}(Q_t\tilde{P}_{t+1|t+1}) \!\leq \! D. \nonumber
\end{align}
Now, notice that the terms $\text{Tr}(Q_t\tilde{P}_{t+1|t+1})$ do not depend on $\tilde{\gamma}$.
Thus, by rewriting the pre-Kalman filter (\ref{eqprekf}) equation as
\begin{equation}
\label{eqxtilde}
\tilde{\bx}_{t+1}=A_t\tilde{\bx}_t+B_t\bu_t+\bpsi_t
\end{equation}
and considering (\ref{eqxtilde}) as a new ``system" with white Gaussian process noise $\bpsi_t\sim\mathcal{N}(0,\Psi_t)$, problem (\ref{equivalentprob2}) can be written as
\begin{subequations}
\label{equivalentprob3}
\begin{align}
\min_{\tilde{\gamma}\in\tilde{\Gamma}} & \quad  I_{\tilde{\gamma}}(\tilde{\bx}^T\rightarrow \bu^T) \\
\text{ s.t. } &\quad J_{\tilde{\gamma}} (\tilde{\bx}^{T+1},\bu^T) \leq \tilde{D}
\end{align}
\end{subequations}
where $\tilde{D}=D-\sum_{t=1}^T \text{Tr}(Q_t\tilde{P}_{t+1|t+1})$.
Since the state $\tilde{X}_t$ of (\ref{eqxtilde}) is fully observable by the modified control policy $\tilde{\gamma}$, (\ref{equivalentprob3}) is now the problem for fully observable systems considered in Section~\ref{secmainresult}.

%\fi


%%%%%%%%%%%%%%%%%%%%%%%%%%%%%%%%%%%%%%%%%%%%%%%%%%%%%%%%%%%%%%%%%%%%%%%%%%%%%%%%

%%%%%%%%%%%%%%%%%%%%%%%%%%%%%%%%%%%%%%%%%%%%%%%%%%%%%%%%%%%%%%%%%%%%%%%%%%%%%%%%

%%%%%%%%%%%%%%%%%%%%%%%%%%%%%%%%%%%%%%%%%%%%%%%%%%%%%%%%%%%%%%%%%%%%%%%%%%%%%%%%
%\section*{APPENDIX}

%\section*{ACKNOWLEDGMENT}

%%%%%%%%%%%%%%%%%%%%%%%%%%%%%%%%%%%%%%%%%%%%%%%%%%%%%%%%%%%%%%%%%%%%%%%%%%%%%%%%

	\bibliographystyle{ieeetran}
	\bibliography{refs}

\ifdefined\SHORTVERSION
\begin{IEEEbiography}[{\includegraphics[width=1in,height=1.25in,clip,keepaspectratio]{Takashi_Tanaka.jpg}}]{Takashi Tanaka}
received the B.S. degree from the University of Tokyo, Tokyo, Japan, in 2006, and the M.S. and Ph.D. degrees in Aerospace Engineering (automatic control) from the University of Illinois at Urbana-Champaign (UIUC), Champaign, IL, USA, in 2009 and 2012, respectively.
From 2012 to 2015, he was a Postdoctoral Associate with the Laboratory for Information and Decision Systems (LIDS) at the Massachusetts Institute of Technology (MIT), Cambridge, MA, USA. Currently, he is a postdoctoral researcher at KTH Royal Institute of Technology, Stockholm, Sweden, where he has been since 2015.
From Fall 2017, he will be an Assistant Professor in the Department of Aerospace Engineering and Engineering Mechanics at the University of Texas at Austin.
His research interests include control theory and its applications; most recently the information-theoretic perspectives of optimal control problems. 
Dr. Tanaka was a recipient of the IEEE Conference on Decision and Control (CDC) best student paper award in 2011.

\end{IEEEbiography}
\begin{IEEEbiography}[{\includegraphics[width=1in,height=1.25in,clip,keepaspectratio]{Peyman.pdf}}]{Peyman Mohajerin Esfahani}
received the B.Sc. and M.Sc. degrees from Sharif University of Technology, Tehran, Iran, in 2005 and 2008, respectively, and the Ph.D. degree from the Automatic Control Laboratory at ETH Zurich, Switzerland, in 2014. 
He is currently an Assistant Professor with the Delft Center for Systems and Control, Delft University of Technology, the Netherlands. Prior to joining TU Delft, he held several research appointments at EPFL, ETH Zurich, and Massachusetts Institute of Technology between 2014 and 2016. His research interests include theoretical and practical aspects of decision-making problems in uncertain and dynamic environments, with applications to control and security of large-scale and distributed systems. He was selected for the Spark Award by ETH Zurich for the 20 best inventions of the year in 2012 and received the SNSF Postdoc Mobility fellowship in 2015. He was one of the three finalists for the Young Researcher Prize in Continuous Optimization awarded by the Mathematical Optimization Society in 2016. He received the 2016 George S. Axelby Outstanding Paper Award from the IEEE Control Systems Society.
\end{IEEEbiography}
\begin{IEEEbiography}[{\includegraphics[width=1in,height=1.25in,clip,keepaspectratio]{sanjoy-mitter.jpg}}]{Sanjoy K. Mitter}
received the Ph.D. degree in Electrical Engineering (automatic control) from the Imperial College, London, U.K., in 1965.
He taught at Case Western Reserve University from 1965 to 1969. He joined the Massachusetts Institute of Technology (MIT), Cambridge, MA, in 1969, where he has been a Professor of electrical engineering since 1973. He was the Director of the MIT Laboratory for Information and Decision Systems from 1981 to 1999. He has also been a Professor of mathematics at the Scuola Normale, Pisa, Italy, from 1986 to 1996. He has held visiting
positions at Imperial College, London; University of Groningen, Holland; INRIA, France; Tata Institute of Fundamental Research, India and ETH, Zurich, Switzerland; and several American universities. He was the McKay Professor at the University of California, Berkeley in March 2000, and held the Russell-Severance-Springer Chair in Fall 2003. His current research interests are communication and control in a networked environment, the relationship of statistical and quantum physics to information theory and control, and autonomy and adaptiveness for integrative organization.
Dr. Mitter was awarded the AACC Richard E. Bellman Control Heritage Award for 2007 and the IEEE Eric E. Sumner Award  in 2015. He is a Member of the National Academy of Engineering. He is the winner of the 2000 IEEE Control Systems Award.
\end{IEEEbiography}
\fi

\end{document}